\documentclass[leqno, 11pt, a4paper]{amsart}
\usepackage{amsmath}
\usepackage{amssymb, latexsym}

 \newtheorem{definition}{Definition}[section]
 \newtheorem{theorem}[definition]{Theorem}
 \newtheorem{lemma}[definition]{Lemma}
 \newtheorem{proposition}[definition]{Proposition}
 \newtheorem{corollary}[definition]{Corollary}

 \usepackage[cmtip,all]{xy}
 \newtheorem*{theorem*}{Theorem}
\newtheorem*{proposition*}{Proposition}
\newtheorem*{lemma*}{Lemma}

 \theoremstyle{remark}
 
 \newtheorem{remark}[definition]{Remark}

\newcommand{\op}[1]{\operatorname{#1}}

\newcommand{\Tr}{\ensuremath{\op{Tr}}}
\newcommand{\tr}{\op{tr}}

\def\XXint#1#2#3{{\setbox0=\hbox{$#1{#2#3}{\int}$}
\vcenter{\hbox{$#2#3$}}\kern-.5\wd0}}

\newcommand{\ind}{\op{ind}}

\newcommand{\Td}{\op{Td}}
\newcommand{\Ind}{\op{Ind}}

\newcommand{\C}{\ensuremath{\mathbb{C}}} 
 
\newcommand{\N}{\ensuremath{\mathbb{N}}} 
\newcommand{\Q}{\ensuremath{\mathbb{Q}}} 
\newcommand{\R}{\ensuremath{\mathbb{R}}} 
 
\newcommand{\Z}{\ensuremath{\mathbb{Z}}}

\newcommand{\Ca}[1]{\ensuremath{\mathcal{#1}}}
\newcommand{\cA}{\Ca{A}}
\newcommand{\cB}{\Ca{B}}

\newcommand{\cE}{\Ca{E}}

\newcommand{\cG}{\ensuremath{\mathcal{G}}}
\newcommand{\cH}{\ensuremath{\mathcal{H}}}

\newcommand{\cK}{\ensuremath{\mathcal{K}}}

\newcommand{\cN}{\Ca{N}}
\newcommand{\cO}{\Ca{O}}
\newcommand{\cP}{\Ca{P}}

\newcommand{\cS}{\ensuremath{\mathcal{S}}}
\newcommand{\cT}{\ensuremath{\mathcal{T}}}

\newcommand{\cW}{\ensuremath{\mathcal{W}}}

\newcommand{\Diff}{{\text{Diff}}}
\newcommand{\Cliff}{{\text{Cliff}}}

\newcommand{\pdo}{\ensuremath{\Psi}}

\newcommand{\disp}{\displaystyle}

\newcommand{\vol}{\op{vol}}

\newcommand{\End}{\ensuremath{\op{End}}}

\newcommand{\Irr}{\ensuremath{\op{Irr}}}

\newcommand{\ch}{\op{ch}}

\newcommand{\f}{\mathfrak}

\newcommand{\ba}{\begin{eqnarray}}
   \newcommand{\na}{\end{eqnarray}}

\numberwithin{equation}{section}

  \def\cancel#1#2{\ooalign{$\hfil#1\mkern1mu/\hfil$\crcr$#1#2$}}
\def\Dirac{\mathpalette\cancel D}

\begin{document}

 \title[Localized Index  and Lefschetz fixed formula for   orbifolds]{Localized Index  and $L^2$-Lefschetz Fixed Point Formula for Orbifolds}
 \author{Bai-Ling Wang}
  \address{Department of Mathematics\\
  Australian National University\\
  Canberra ACT 0200 \\
  Australia}
  \email{bai-ling.wang@anu.edu.au}
 \author{Hang Wang}
 \address{Mathematical Sciences Center, Tsinghua University, Beijing 100084, China}
 \address{(Current) School of Mathematical Sciences\\ University of Adelaide, Adelaide SA 5005 Australia}
 \email{hang.wang01@adelaide.edu.au, hwang@math.tsinghua.edu.cn}
\maketitle 

\begin{abstract}
We study a class of localized indices for  the Dirac type operators on a complete Riemannian orbifold, where a discrete group acts properly, co-compactly and isometrically.
These localized indices, generalizing the $L^2$-index of Atiyah, are obtained by taking  certain  traces of the higher index for the Dirac type operators along   conjugacy classes of the discrete group.  Applying the local index technique, we also obtain an $L^2$-version of  the  Lefschetz fixed point formula for  orbifolds.  These cohomological formulae for the localized indices give rise to a class of refined topological invariants for the quotient orbifold.  
\end{abstract}

   \tableofcontents

\section{Introduction}

 The  Lefschetz fixed point formula    calculates the supertrace of the action of  a diffeomorphism on a closed Riemannian  manifold with isolated fixed points.  This was generalized by Atiyah and Bott in \cite{AB} to an elliptic complex with a geometric endomorphism  defined by a 
 diffeomorphism  with isolated fixed points.  When the diffeomorphism  $\gamma$ comes from a {\em compact}  group $H$  of orientation preserving isometries  of a {\em compact}  even dimensional  manifold $X$, the fixed point formulae are special case of the equivariant index formulae for the  equivariant index 
 \ba\label{gamma:index}
 \ind_H(\gamma,  \Dirac) = \Tr(\gamma |_{\ker \Dirac^+} ) - 
 \Tr(\gamma |_{\ker \Dirac^+} )
 \na
assocaited to an $H$-invariant Dirac operator $\Dirac$  on an equivariant Clifford module $\cE$.  The local index formula, also called the   Lefschetz fixed point formula, 
 \ba\label{Lefs}
  \ind_H(\gamma,  \Dirac)  = \int_{X^\gamma} \hat{A}_\gamma(X) \ch_\gamma(\cE/\cS)
  \na
    is  obtained by an asymptotic expansion of the equivariant heat kernel to the operator
$\gamma e^{-t\Dirac^2}$.  Here $X^\gamma$ is the fixed point set of the $\gamma$-action,  consisting of closed submanifolds of $X$, and    $\hat{A}_\gamma(X) \ch_\gamma(\cE/\cS)$ represents the local index density for the  equivariant index.  See {\cite[Chapter 6]{BGV}} and \cite{PW:NCGCGI.PartII}   for a detailed  account of these developments.   When $X$ is not compact, the equivariant index  is not defined and there does not exist any Lefschetz fixed point formula.  We remark that the Lefschetz fixed point formulae  for compact orbifolds have been established   in \cite{Dui, FX, EEK}.

To motivate our study, let us consider a compact  even dimensional  Riemannian  manifold $M$ with an isometric action of a \emph{finite} group $H$. Let $\Dirac$ be an $H$-invariant Dirac operator on an equivariant Clifford module $\cE$. Then $\Dirac$ 
defines a Dirac operator  $\Dirac_{\f{X}}$
on the quotient orbifold $\f{X} = H\backslash M$.   An orbifold locally looks like Euclidean space equipped with a finite group action, and is a basic geometric model used in many areas such as mathematical physics, algebraic geometry, representation theory and number theory. Index theory of Dirac type operators on orbifolds is an operator algebraic approach in detecting the topological, geometrical, arithmetic information concerning the orbifolds.  

 The Kawasaki orbifold  index for $\Dirac_{\f{X}}$ in \cite{Kawasaki:1981} for the finite quotient orbifold $\f{X} = H\backslash M $ can be obtained by applying the 
 Lefschetz fixed point formula (the local equivariant index formula) for  $\ind_H (\gamma, \Dirac)$ to  the following identity
\ba\label{index:finite}
\ind \Dirac_{\f{X}} =  \dfrac {1}{|H| } \sum_{\gamma \in H}  \ind_H (\gamma, \Dirac)  =  \sum_{(\gamma)}    \dfrac {1}{|Z_H(\gamma)| }   \ind_H (\gamma , \Dirac). 
\na
Here, the second summand  is taken over  representatives for the conjugacy classes in $H$.  
Applying the Lefschetz fixed point formula (\ref{Lefs}), this is exactly the 
 Kawasaki orbifold  index for the Dirac operator  $\Dirac_{\f{X}}$ on the orbifold $\f{X}$.   
From the view point of topology, the formula~(\ref{index:finite}) for $\Dirac_{\f{X}}$ follows from its analogue of calculating the equivariant $K$-theory of the compact proper $H$-manifold $M$ using the extended quotient $\sqcup_{(\gamma)} Z_H(\gamma)\backslash M^\gamma$ indexed by the conjugacy classes of the finite group $H$ (See~\cite{BC:88, EE}):
\begin{equation}
\label{eq:K-theoryProperGspaceCompFin}
 K_H^*(M)\otimes\C\cong \oplus_{(\gamma)}K^*(Z_H(\gamma)\backslash M^\gamma)\otimes\C.
\end{equation}
 
In this paper, we are interested in the index problems for a Dirac type operator on  an orbifold that is not necessarily compact. 
Namely,  we consider a complete 
even dimensional Riemannian   orbifold $\f{X}$ where a discrete group $G$ acts properly, co-compactly and isometrically. 
Let $\Dirac^{\cE}$ be a $G$-invariant Dirac type operator on $\f{X}$ acting on the $L^2$-sections of a Dirac bundle $\cE$ ({\em cf.}~Definition~\ref{def:DiracBundle})
and let $\Dirac^{G\backslash \cE}$ be the Dirac operator $\Dirac^{\cE}$ passed from $\f{X}$ to the compact quotient orbifold $G\backslash \f{X}$. 
We study  the index theory  of $\Dirac^{\cE}$ as well as its relation to $\Dirac^{G\backslash \cE}$, based on the following facts when $\f{X}$ is a manifold.  

\begin{enumerate}
\item The $G$-invariant elliptic operator $\Dirac^{\cE}$ defines a cycle in the $K$-homology group of $C_0(\f{X})$ and is ``Fredholm" with respect to the $C^*$-algebra $C^*(G)$ of the discrete group $G$. In other words, $\Dirac^{\cE}$ has a higher index in the $K$-theory of the maximal group $C^*$-algebra of $G$ (See, for example~\cite{Kasparov:1983}): 
\begin{equation}
\label{eq:HigherIndexManifoldIntro}
\mu: K^0_G(C_0(\f{X}))\longrightarrow K_0(C^*(G))\qquad [\Dirac^{\cE}]\mapsto \mu[\Dirac^{\cE}].
\end{equation}

\item The Dirac operator $\Dirac^{G\backslash \cE}$ on the quotient orbifold $G\backslash \f{X}$ has a Fredholm index according to the work of Kawasaki~\cite{Kawasaki:1981} and is related to the higher index $\mu[\Dirac^{\cE}]$ of $\Dirac^{\cE}$ 
by composing the homomorphism $\rho: C^*(G)\rightarrow\C$, given by the trivial representation of $G$, on the $K$-theory level (See~\cite{Bunke:2007}):
\begin{equation*}
\ind \Dirac^{G\backslash \cE}=\rho_*(\mu[\Dirac^{\cE}])\in\Z.
\end{equation*} 

\item The $L^2$-index  is an important topological invariant for a proper co-compact $G$-manifold. It was initially defined by Atiyah in \cite{Atiyah:1976} for a manifold admitting a free co-compact action of a discrete group. The $L^2$-index for $\Dirac^{\cE}$ is a numerical index measuring the size of the space of $L^2$-solutions for $\Dirac^{\cE}$.  
In~\cite{W:2012}, we showed that the $L^2$-index of $\Dirac^{\cE}$ follows from taking the von Neumann trace of the higher index~(\ref{eq:HigherIndexManifoldIntro}):
\begin{equation}
\label{eq:L^2-indexIntro}
L^2\text{-}\ind \Dirac^{\cE}=\tau_*(\mu[\Dirac^{\cE}])\in\R
\end{equation}
where $\tau_*: K_0(C^*(G))\rightarrow\R$ is induced by the trace map 
\begin{equation*}
\tau: C^*(G)\longrightarrow \C\qquad \sum_{g\in G}{\alpha_g g}\mapsto\alpha_e.
\end{equation*}
Moreover, by the $L^2$-index formula derived in \cite{W:2012}, the $L^2$-index of $\Dirac^{\cE}$ is the top stratum of the Kawasaki index formula for $\Dirac^{G\backslash \cE}.$
\end{enumerate}

The paper is devoted to answering the following concrete questions for the index theory for $\Dirac^{\cE}$:
\begin{enumerate}
\item How do we formulate the higher index and $L^2$-index in the setting when $\f{X}$ is a non-compact orbifold?
\item How is the orbifold index $\ind \Dirac^{G\backslash \cE}$ related to the $L^2$-index for $\Dirac^{\cE}$?
\item Is there any  Lefschetz fixed point formula  for a \emph{non-compact} orbifold with a transformation preserving all the geometric data for the Dirac type operator?
\end{enumerate}

The main results of the paper are then summarized as follows.

First of all, we have a positive answer to the first question, that is, the indices $\mu[\Dirac^{\cE}]$, $\ind \Dirac^{G\backslash \cE}$ and $L^2$-$\ind \Dirac^{\cE}$ can be formulated when $\f{X}$ is a complete Riemannian  orbifold admitting a proper, co-compact and isometric action of a discrete group $G$. Namely, we construct in~(\ref{eq:K-theoreticIndex}) a higher index map:
\begin{equation*}
\mu: K^0_G(C_{red}^*(\f{X}))\longrightarrow K_0(C^*(G))\qquad [\Dirac^{\cE}]\mapsto \mu[\Dirac^{\cE}],
\end{equation*}
where the Dirac operator $\Dirac^{\cE}$ on the $G$-orbifold $\f{X}$ gives rise to an equivariant $K$-homology class of the reduced $C^*$-algebra $C_{red}^*(\f{X})$, see Lemma~\ref{le:KHomologyElement}.
The orbifold index $\ind \Dirac^{G\backslash \cE}$ of the quotient $G\backslash\f{X}$ is related to the higher index by the trivial representation of $G$, see Theorem~\ref{Thm:TrivialRepK-theoreticIndex=OrbifoldIndex}.

Secondly, in searching for a relationship between the orbifold index of $\Dirac^{G\backslash \cE}$ and the $L^2$-index of $\ind \Dirac^{\cE}$, we find it is necessary to introduce the notion of localized indices for $\Dirac^{\cE}$, which is a major novelty of our paper. 
See Definition~\ref{DefEquiInd}. For each conjugacy class $(g)$ of $g$ in $G$, the localized $(g)$-index of $\Dirac^{\cE}$, denoted by $\ind_{(g)}(\Dirac^{\cE})$, is defined as the pairing between the Banach algebra version higher index $\mu_{\cS(G)}[\Dirac^{\cE}]$ (\emph{cf.} Definition~\ref{def:BanachVHigherIndex}) and a localized $(g)$-trace $\tau^{(g)}$, which is a cyclic cocycle of degree $0$  given  by the sum of the coefficients over elements in  the conjugacy class $(g)$ for an element in  
the group algebra $\C G$  (\emph{cf.} Definition~\ref{def:Localized(g)Trace}).
It then follows from Theorem~\ref{Thm:OrbifoldIndex=SumLocalizedIndex} that the higher index $\mu_{\cS(G)}[\Dirac^{\cE}]$ is linked to the numerical indices $\ind \Dirac^{G\backslash \cE}$ and $L^2$-$\ind \Dirac^{\cE}$ by the localized indices as follows:  
\begin{enumerate}
\item The localized indices factor through the Banach algebra version higher index by definition: 
\begin{equation}
\label{eq:LocIndIntro}
\ind_{(g)}(\Dirac^{\cE})=\tau^{(g)}_*(\mu_{\cS(G)}[\Dirac^{\cE}]);
\end{equation}
In particular, the nonvanishing of the localized $(g)$-index implies the existence of $L^2$-solutions invariant under the action of $g\in G.$
\item The localized $(e)$-index generalizes the $L^2$-index and factors through the {\em usual} higher index $\mu[\Dirac^{\cE}]$:
\begin{equation*}
L^2\text{-}\ind \Dirac^{\cE}=\ind_{(e)}(\Dirac^{\cE})=\tau^{(e)}_*(\mu[\Dirac^{\cE}]);
\end{equation*}

\item The orbifold index for $\Dirac^{G\backslash \cE}$ on $G\backslash \f{X}$ is the sum of localized indices for $\Dirac^{\cE}$ over all conjugacy classes: 
\begin{equation}
\label{index:infinite}
\ind \Dirac^{G\backslash \cE}=\sum_{(g)\in G}\ind_{(g)}(\Dirac^{\cE}).
\end{equation}
\end{enumerate} 

On the  one hand, localized indices~(\ref{eq:LocIndIntro}) give rise to a $K$-theoretic interpretation of each component of the orbifold index for $\Dirac^{G\backslash\f{X}}$ on the quotient orbifold $G\backslash\f{X}$. Hence, as an analogue of~(\ref{eq:K-theoryProperGspaceCompFin}) in our more general setting, localized indices produce finer topological invariants on $G\backslash\f{X}$. 
On the other hand, comparing~(\ref{index:infinite}) with~(\ref{index:finite}), localized index  should be thought as a replacement of  the equivariant index (\ref{gamma:index})   for the  Lefschetz fixed point formula in the compact situation. This  motivated us  to find a cohomological formula for $\ind_{(g)}(\Dirac^{\cE})$, see Theorem~\ref{thm:MainTheorem}.

Thirdly,  combining the techniques of $KK$-theory and heat-kernel method, we explicitly compute
the $(g)$-index as the orbifold integration of the local index density over $\f{X}_{(g)}$, which is an orbifold of type $Z_G(g)\backslash \f{X}^g$ (the fixed point sub-orbfiold).  The formulae are presented in Theorem~\ref{thm:MainTheorem}
and they are regarded as an analogue of the ``Lefschetz fixed point formulae" for the isometries given by the action of the discrete group $G$ on the complete Riemannian orbifold $\f{X}.$
There are two nontrivial technical points in deriving the cohomological formulae for the localized indices.
The first one is Theorem~\ref{thm:HeatKernelAsympMain}, which provides  an  asymptotic expansion of the heat kernel $K_t(x,y)$ of $\Dirac^{\cE}$ and the uniform convergence of $\sum_{g\in G}K_t(x,gx)g$ over a relative compact subset of $\f{X}$ which have nontrivial intersection with each orbit for the $G$-action. The second one is Lemma~\ref{lem:VeryImportantLemma}. The lemma relates two notions of traces in topological and analytic settings, namely, the continuous trace defined on a certain completion of the group algebra $\C G$ and some supertrace of the heat kernel. See Definition~\ref{def:Localized(g)Trace} and Definition~\ref{def:(g)trace} for the precise descriptions.

Finally, we provide  some observations and possible applications for the localized indices and their local index formulae:
\begin{enumerate}
\item When $\f{X}$ is replaced by a manifold $M$, we derive the $L^2$-version of the Lefschtez fixed  point formula for a complete Riemannian manifold $M$ with a proper co-compact action of a discrete group $G$, by introducing an average function, called cut-off function, with respect to some group action  on the fixed point sub-manifold  (See Theorem~\ref{thm:LocalizedIndexKawasakiIndex}). This result extends the $L^2$-index formulae in~\cite{Atiyah:1976, W:2012} from the top stratum of $G\backslash M$ to the lower strata.

\item Let $\f{X}=G/K$ be a Riemannian symmetric manifold of noncompact type, where $G$ is a real semi-simple Lie group and $K$ is a maximal compact subgroup. Let $\Gamma$ be a discrete co-compact  subgroup of $G$. The  cohomological formulae for the localized indices for $\Dirac^{\cE}$ give rise to the orbital integrals in the Selberg trace formula. See Theorem~\ref{thm:SelbergTraceFormula}. This means that our localized indices in this special case are related to the multiplicity of unitary representations of the Lie group $G.$ 

\item The nonvanishing of the higher index for $\Dirac^{\cE}$ is useful in solving topological and geometric problems for the orbifold $\f{X}$. Explicit formulae for the localized indices, which are derived from the higher index, provide a tool to tackle these problems for $\f{X}.$ Theorem~\ref{thm:PositiveScalarCurvature} provides an example of how the non-vanishing of the localized indices derive the non-existance of positive scalar curvature for complete spin  orbifolds.
\end{enumerate}

The paper is organized as follows: 
In Section~\ref{sec:A review of orbifolds and orbifold index theorem} we review basic definitions concerning orbifolds and the Kawasaki orbifold index theorem. 
In Section~\ref{sec:EllipticPDO} we give the definition of invariant elliptic operators on a proper co-compact $G$-orbifold and use heat kernel method to calculate the localized supertraces of the heat kernel $e^{-t(\Dirac^{\cE})^2}$ of the Dirac operator $\Dirac^{\cE}.$
In Section~\ref{sec:K-theoretic index} we formulate the higher index map for $\Dirac^{\cE}$ and show that the Fredholm index for $\Dirac^{G\backslash \cE}$ factors through this higher index map.
In Section~\ref{sec:Localized indices} we introduce the localized indices and present their various connections to the other indices. We also use the analytical results from Section~\ref{sec:EllipticPDO} to derive the cohomological index formulae for the localized indices which correspond to the $L^2$-version of the Lefschtez fixed point formulae for a non-compact orbifold. 
In Section~\ref{sec:Applications and Further Remarks} we show several interpretations and remarks of the localized indices.

\vskip .2in
\noindent
{\bf Acknowledgments}
  This work is  supported by   the Australian Research Council's Discovery Projects funding scheme (DP1092682).  We would like to express our appreciation  to Alan Carey,   Xiaonan Ma, Mathai Varghese,  Guoliang Yu and Weiping Zhang for numerous inspirational discussions. HW  also likes to thank Mathematical Sciences Institute at Australian National University and Shanghai Center of Mathematical Sciences  at Fudan for
their hospitalities when part of work was carried out during her visits.

\section{A Review of Orbifolds and Orbifold Index Theorem}
\label{sec:A review of orbifolds and orbifold index theorem}

In this section, we give a preliminary review of  orbifolds in terms of orbifold atlas and orbifold groupoid,  and Kawasaki's orbifold index theorem.  Some of basic references are
\cite{Thurston:1997}, \cite{Kawasaki:1981} and    \cite{ALR}. Chapter 2 of  \cite{KL} contains a nice account for differential and Riemannian geometry of orbifold. Due  to  the  subtle nature of  orbifolds and the fact that there is not any universal agreement in describing orbifolds, we try   to give a sufficiently self-contained review as best as we can.

\subsection{Orbifolds and   discrete  group actions}  

Recall that an $n$-dimensional   orbifold $\f{X}$ is a paracompact Hausdorff  space $|\f{X}|$ equipped with an equivalent class of  orbifold atlases.  Here  we recall the definition of an  orbifold atlas.

\begin{definition} {\bf (Orbifold and local groups)}   Let $|\f{X}|$ be  a  paracompact Hausdorff  space.  An orbifold atlas on $|\f{X}|$ is a coherent  system of  orbifold charts $\cO= \{(\tilde U_i, H_i, \pi_i)\}$ 
such that
\begin{enumerate}
\item  $\{U_i\}$ is an open cover  of $|\f{X}|$ which is closed under finite intersections.
\item  For each  $U_i$, there is an obifold chart  $(\tilde U_i, H_i, \pi_i)$ 
where $\tilde U_i$ is  a connected  open neighborhood   in some Euclidean space $\R^n$ with a right action of a   finite group $H_i$ with  the quotient map   $ \pi_i: \tilde U_i \to U_i$.  
\item  For  any  inclusion  $\iota_{ij}:  U_i\to U_j$,  there is an embedding  of orbifold charts
\ba\label{def:1}
 (\phi_{ij}, \lambda_{ij}):   (\tilde U_i, H_i, \pi_i)  \hookrightarrow   (\tilde U_j, H_j, \pi_j), 
\na
 which is given by an injective  group 
 homomorphism  $\lambda_{ij}:  H_i \to H_j$  and an embedding 
$
 \phi_{ij}:    \tilde U_i   \hookrightarrow    \tilde U_j $ covering the inclusion $\iota_{ij}$  such that $\phi_{ij}$ is $H_i$-equivariant with respect to  $\phi_{ij}$, that is, 
 \ba\label{equ:1}
 \phi_{ij}(   x \cdot g) =       \phi_{ij}(x )  \cdot  \lambda_{ij}( g),
 \na
 for $x\in \tilde U_i$ and $g\in H_i$.  
 When the action of $H_i$  is not effective, 
the subgroup of $H_i$ acting trivially on $U_i$ is isomorphically mapped  to the subgroup
of $H_j$ acting trivially on $U_j$.  
\end{enumerate}
Here the coherent  condition for $\cO$ is  described as follows: given   $U_i\subset U_j \subset U_k$, there exists an element
$g\in H_k$ such that 
\ba\label{coherent}
g \circ \phi_{ik} =  \phi_{jk} \circ  \phi_{ij}, \qquad 
g \lambda_{ik}(h) g^{-1} = \lambda_{jk}\circ \lambda_{ij} (h) 
\na for any
$ h\in H_i$.
 Two orbifold atlases  are called equivalent  if they are  included in a third orbifold atlas (called a common refinement). 
 
   Given an orbifold  $\f{X}$ and a point $x\in |\f{X}|$, let $(\tilde U, G, \pi)$ be an orbifold chart around $x$.  Then the {\bf local group}  at $x$ is defined to   be the  
stabilizer of  $\tilde x \in \pi^{-1}(x)$, which is uniquely defined up to conjugation.

 \end{definition} 

As proved in  Corollary 1.2.5 of \cite{MoePr:1997}, any orbifold atlas admits a refinement  $\cO$ such  that  each orbifold atlas $(\tilde U, H, \pi, U)$ in $\cO$, both $\tilde U$ and $U$ are contractible. Such an orbifold atlas is called a {\bf good}  atlas.
 For convenience, we may choose each  orbifold chart $(\tilde U, H, \pi)$ in $\cO$  such that $\tilde U$ is an open ball centred at the origin in $\R^n$ and $H$ is a finite group of linear transformations. 

 An orbifold $\f{X}$  is called {\bf compact}  (resp. {\bf connected}) if $|\f{X}|$ is compact  (resp. connected).  An orbifold  $\f{X}$ is {\bf oriented}  if there exists an orbifold atlas  $\{(\tilde U_i, H_i, \pi_i)\}$ such that, $\tilde U_i$ is an open set of an oriented Euclidean space $\R^n$ and  the $H_i$-action preserves the orientation, moreover, 
all  the embeddings $ \{\phi_{ij}\}$ in (\ref{def:1})  are orientation preserving.

 A smooth map $f: \f{X} \to  \f{Z}$ between two orbifolds is a continuous map $|f|: |\f{X}|  \to  |\f{Z}|$ with the following property: 
 there exist orbifold atlases $\cO_{\f{X}} = \{(\tilde U_i, H_i, \pi_i)\}$ and $\cO_{\f{Z}} 
 = \{(\tilde V_\alpha, G_\alpha, \pi_\alpha)\} $ for $\f{X} \to  \f{Z}$ respectively, such that
 \begin{enumerate}
\item For each $(\tilde U_i, H_i, \pi_i) \in \cO_{\f{X}}$, there is an orbifold chart
$ (\tilde V_{\alpha_i}, G_{\alpha_i}, \pi_{\alpha_i})  \in \cO_{\f{X}}$ with a  local smooth 
map  $f_i:  (\tilde U_i, H_i) \to  (\tilde V_{\alpha_i}, G_{\alpha_i})$ making the following diagram commute
\[
\xymatrix{
\tilde U_i \ar[r]^{f_i} \ar[d]^{\pi_i} & \tilde V_{\alpha_i}\ar[d]^{\pi_{\alpha_i}}\\
 U_i \ar[r]^{|f|_i}  & V_{\alpha_i},
}
\]
where $f_i$ is $G_{\alpha_i}$-equivariant
\item For any embedding of orbifold charts $(\tilde U_i, H_i, \pi_i)  \hookrightarrow   (\tilde U_j, H_j, \pi_j)$ in
$\cO_{\f{X}}$, there is an corresponding  embedding of orbifold charts 
\[
(\tilde V_{\alpha_i}, G_{\alpha_i}, \pi_{\alpha_i})  \hookrightarrow   (\tilde  V_{\alpha_j}, G_{\alpha_j}, \pi_{\alpha_j})
\]
 in
$\cO_{\f{Z}}$ such that the following diagram commutes  
\[
\xymatrix{
( \tilde U_i , H_i, \pi_i)  \ar[r]^{f_i} \ar[d]  &(\tilde V_{\alpha_i}, G_{\alpha_i}, \pi_{\alpha_i})\ar[d] \\
(\tilde U_j, H_j, \pi_j) \ar[r]^{f_j}  &  (\tilde  V_{\alpha_j}, G_{\alpha_j}, \pi_{\alpha_j}).
}
\]
in the obvious equivariant sense. 
\item The coherent condition (\ref{coherent}) is preserved under $f_i$'s.  
\end{enumerate}
A diffeomorphism $f:   \f{X} \to  \f{Z}$ is a smooth map with a smooth inverse.  The set of
all diffeomorphisms from $ \f{X}$ to itself is denoted by $\Diff(\f{X})$. 

A Riemannian metric on an orbifold  $\f{X} = (|\f{X}|, \cO)$ is a collection of Riemannian metrics on $\tilde U_i$'s   such that for each orbifold chart $(\tilde U_i, H_i, \pi_i)$ in $\cO$, 
$H_i$ acts isometrically on $\tilde U_i$ and all  the embeddings $ \{\phi_{ij}\}$ in (\ref{def:1})  are isometric.  We have already assumed that the orbifold atlas for $\f{X}$ is a good atlas.     Note that when $\f{X}$ is an
effective Riemannian $n$-dimensional  orbifold, the orthonormal frame  of  $ \f{X}$ is a smooth manifold  with a
locally free $O(n)$-action whose quotient space can be equipped a natural orbifold structure isomorphic to 
$\f{X}$.

\begin{remark}
An orbifold is called {\bf presentable}  if it arises from a locally free action of a compact Lie group on a smooth manifold.  It is obvious to see that any effective orbifold is presentable. Conjecturally, every orbifold is presentable ({\em cf.} Conjecture 1.55 in \cite{ALR}).   An orbifold is called a {\bf good orbifold}  if it is a quotient of a smooth manifold  by a  proper action of   a discrete group. Equivalently, an orbifold is good if its orbifold
universal cover is smooth, see \cite{Thurston:1997} and \cite{ALR} for detailed discussions. 
\end{remark}

In this paper, we  are mainly interested in  a noncompact oriented Riemannian  orbifold  $\f{X}$ with    a discrete group  $G$-action which is  proper, co-compact and isometric.  Given a discrete group $G$ and
an orbifold $\f{X}$, a smooth action of $G$ on $\f{X}$ is a group homomorphism
\[
G\longrightarrow \Diff (\f{X}).
\]
An action of $G$ on $\f{X}$ is called proper if the map 
\ba\label{proper}
G \times  \f{X} \longrightarrow  |\f{X}| \times |\f{X}| \qquad (g,x)\mapsto (x, gx)
\na
is proper.  It is easy to check that the quotient
space $G\backslash|\f{X}|$ can be equipped with an orbifold structure  as given by the following lemma. 

 \begin{lemma}
 \label{lem:quotient} 
 Let $\f{X}$ be a noncompact Riemannian orbifold, and $G$ is a discrete group  acting   properly, co-compactly  and isometrically on  $\f{X}$.  
\begin{enumerate} 
\item Then  $|\f{X}|$ is covered by finite number of $G$-slices of the  form $G\times_{G_i} U_i$ for $i=1, \cdots, N$, where $G_i$ is a finite  subgroup of $G$ and $U_i$ is the quotient space of some Euclidean ball $\tilde U_i$ by a finite group $H_i$, where $\tilde U_i$ admits a left $G_i$-action and right $H_i$-action.
\item Let $|\f{X}| =\cup_{i=1}^N G\times_{G_i} U_i$ be covered by $G$-slices as in (1). Then $\f{X}$ has the orbifold atlas generated by $\{G\times_{G_i}\tilde U_i, H_i\}_{i=1}^N.$
\item The orbit space $G\backslash \f{X}$, being a compact orbifold as a result of the properness of the action, admits an orbifold atlas generated by $\{\tilde U_i, G_i\times H_i\}_{i=1}^N.$
\end{enumerate}
\end{lemma}

\begin{proof} As the action is proper, given an open set $U_i$ of $|\f{X}|$, then the pre-image of  the closure of $U_i\times U_i$ under the map (\ref{proper}) is compact. This implies that there is a finite subgroup $G_i$ such that $U_i$ is $G_i$-invariant. By choosing
$U_i$ small enough, we can assume that 
\[
\{g\in G| g (U_i) \cap U_i \neq \emptyset \} = G_i
\]
For each $U_i$, let $(\tilde U_i, H_i, \pi_i)$ be an orbifold chart over $U_i$, then taking a smaller $U_i$,  by the definition of a diffeomorphism of  $\f{X}$, we know that 
$G_i$-action on $U_i$ can be lifted to an action  on $\tilde U_i$ which is $H_i$-equivariant. This means that $G_i$-action and $H_i$-action commute.  We can choose $\tilde U_i$ to be a  Euclidean ball
with a linear action of $H_i$. 
As $G$ acts co-compactly on $\f{X}$, there are finitely many $G$-slices of the form 
\[
G\times_{G_i} U_i. 
\]
This completes the proof of the Claim (1). Claims (2) and (3) are obvious.
\end{proof}

It becomes a majority  view that  the language of groupoids provides a convenient and economical  way to describe orbifolds.    In this paper,  the groupoid viewpoint is also essential to get a correct   $C^*$-algebra associated to an orbifold in order to define
a correct version of $K$-homology for orbifolds.    We briefly recall the   definition of  an orbifold groupoid which is just a proper \'etale groupoid constructed from its orbifold atlas.

\begin{definition} {\bf (Proper \'etale groupoids)} 
 A Lie groupoid
 $\cG = (\cG_1 \rightrightarrows \cG_0)$  consists of two smooth manifolds $\cG_1$ and $\cG_0$, together with five smooth maps $(s, t, m, u, i)$ satisfying the following properties. 
  \begin{enumerate}
\item  The source map  and the target map $s, t: \cG_1 \to  \cG_0$ are submersions.
\item The composition map
\[
m:  \cG_1^{[2]}: =\{(g_1, g_2) \in  \cG_1 \times  \cG_1: t(g_1) = s(g_2)\} \longrightarrow \cG_1
\]
written as $m(g_1, g_2) = g_1\cdot g_2$ for composable elements $g_1$ and $g_2$. 
satisfies the obvious associative property.
\item The unit map $u: \cG_0 \to \cG_1$ is a two-sided unit for the composition.
\item The inverse map $i: \cG_1 \to \cG_1$, $i(g) = g^{-1}$,  is a two-sided inverse for the composition $m$. 
\end{enumerate}
A Lie groupoid $\cG = (\cG_1 \rightrightarrows \cG_0)$   is {\bf  proper } if $(s, t): \cG_1 \to \cG_0 \times \cG_0$ is proper, and called {\bf   \'etale } if $s$ and $t$ are local diffeomorphisms.   
\end{definition}

In the category of proper \'etale groupoids, a very important notion of morphism is the so-called generalized morphisms developed by Hilsum and Skandalis (for the category of  general Lie groupoids in  \cite{HS:1987}).
Let $\cG_1 \rightrightarrows \cG_0$ and 
$\cH_1  \rightrightarrows \cH_0$ be two   proper \'etale  Lie  groupoids.   A {\bf   generalized morphism }
between 
$ \cG$ and  $ \cH  $, denoted by $f:  \xymatrix{\cG \ar@{-->}[r] &  \cH}$, 
  is   a right
principal $\cH$-bundle $P_f$  over $\cG_0$ which is also a left $\cG$-bundle over $\cH_0$ such that the left $\cG$-action and the right $\cH$-action commute,  formally denoted by
\ba\label{H-S}
\xymatrix{
\cG_1 \ar@<.5ex>[d]\ar@<-.5ex>[d]&P_f \ar@{->>}[ld] \ar[rd]&
\cH_1 \ar@<.5ex>[d]\ar@<-.5ex>[d]\\
\cG_0&&\cH_0
}
\na
See \cite{CarW:2011} for more details.  Note that a 
generalized morphism $f$
between 
$ \cG$ and  $ \cH  $ is invertible if $P_f$ in (\ref{H-S}) is also a principal $\cG$-bundle
over $\cH_0$. Then $ \cG$ and  $ \cH  $ are called Morita equivalent. 

\begin{remark}
\label{rem:OrbifoldGroupoidStructure}
Given a proper \'etale Lie groupoid $\cG$, there is  a canonical orbifold structure on its orbit space 
$|\cG|$ (\cite[Prop. 1.44]{ALR}). Two Morita equivalent 
proper \'etale Lie groupoids define two diffeomorphic  orbifolds  (\cite[Theorem 1.45]{ALR}).  
Conversely,     given an orbifold $\f{X}$ with an orbifold  atlas $\cO =\{(\tilde U_i, H_i, \pi_i)\}$, there is a canonical proper \'etale Lie groupoid $\cG_{\f{X}}$,  locally given by the action groupoid $\tilde U_i \rtimes H_i\rightrightarrows \tilde U_i $, see \cite{MoePr} and \cite{LU}.    Two equivalent  orbifold atlases define
two  Morita equivalent proper \'etale Lie groupoids.   A proper \'etale Lie groupoid  arising from an orbifold
atlas  will be called an {\bf orbifold groupoid} for simplicity in this paper.  Our notion of orbifold groupiods is slightly different from that in
literature where any  proper \'etale Lie groupoid is called an orbifold groupoid.  One can check that a smooth map between two orbifolds
corresponds to a generalized morphism between the associated orbifold groupoids.
\end{remark}

 Using the language of orbifold groupoids, it is simpler to describe de Rham cohomology and orbifold $K$-theory for a compact
 orbifold $\f{X}$. Let $\cG = \cG_{\f{X}}$ be the associated orbifold groupoid. Then
 the   de Rham cohomology   of  $\f{X}$, denoted by  $H^*_{orb}(\f{X}, \R)$,  is just  the cohomology of the $\cG$-invariant de Rham complex 
   $(\Omega^*(\f{X}), d)$, where 
   \[
   \Omega^*(\f{X}) =  \{\omega \in \Omega^* (\cG_0) | s^* \omega = t^*\omega\}.
   \]
  The  Satake-de Rham theorem  says that there is  a natural  isomorphism
\[
H^*_{orb}(\f{X}, \R) \cong H^*(|\f{X}|, \R). 
\] 
Hence, the orbifold de Rham cohomology   of  $\f{X}$ does not provide any orbifold information about   $\f{X}$. 
 In terms of orbifold atlas $\cO =\{(\tilde U_i, H_i, \pi_i)\} $ of $\f{X}$, a differential form $\omega \in  \Omega^*(\f{X})$ can be written as
 a family of local equivariant  differential forms on $\tilde U_i$'s which  respect the equivariant condition (\ref{equ:1}) and the
 coherent condition (\ref{coherent}).    Let
 $\omega$ be a compactly supported $n$-form on $\f{X}$, we define 
 \ba\label{orbi:integral}
 \disp{\int}_{\f{X}}^{orb} \omega = \sum_{i} \dfrac{1}{|H_i|}\disp{\int}_{\tilde U_i} \rho_i \omega,
 \na
 where $ \rho_i $ is a smooth partition of unity subordinate to $\{\tilde U_i\}$ in an obvious sense. 
  Then this is a well-defined integration map for an $n$-dimensional orbifold $\f{X}$.

An orbifold vector bundle $\cE$  over an orbifold $\f{X} = (|\f{X}|, \cO)   $ is a coherent family  of 
equivariant vector bundles 
\[
\{ (\tilde E_i \to \tilde U_i, H_i )  \}
\]
  such  that 
for any embedding of orbifold charts $\phi_{ij}:   (\tilde U_i, H_i)  \hookrightarrow   (\tilde U_j, H_j)$,  there is   a  $H_i$-equivariant  bundle map   $\tilde \phi_{ij}: \tilde E_i \to \tilde E_j$ covering 
$ \phi_{ij}: \tilde U_i \to \tilde U_j$.  The coherent condition for $\tilde \phi_{ij}$ is given by (\ref{coherent})  with $\phi_{ij}$'s replaced by
 $\tilde \phi_{ij}$'s. Then the total space $E=  \bigcup (\tilde E_i/H_i)$ of an orbifold vector bundle $\cE \to \f{X}$  has a canonical orbifold structure  given by $\{(\tilde E_i, H_i)\}$.     We remark that orbifold vector bundles in this paper are those called proper orbifold vector bundles in \cite{Kawasaki:1981} and \cite{ALR}.  An orbifold vector bundle is called  real (resp.  complex)  if each $E_i$ is real (resp. complex). For example, for an orbifold $\f{X}$, the local tangent  and cotangent bundles form real orbifold vector bundles over $\f{X}$, denoted by
 $T\f{X}$ and $T^*\f{X}$ respectively. 
 
 A smooth section of an orbifold vector bundle is a family of  invariant smooth sections of $\tilde E_i$'s which behave well under 
 the equivariant condition (\ref{equ:1}) and the coherent condition (\ref{coherent}).  The space of smooth sections of $\cE$ will be
 denoted by $\Gamma(\f{X}, \cE)$. 
A connection $\nabla$ on a complex  orbifold vector bundle $\cE\to\f{X}$ is a family of invariant connections $\{ \nabla_i\}$ 
 on
$
\{ ( \tilde E_i \to \tilde U_i, H_i)  \}
$
which  are compatible with the  bundle maps $\{\tilde \phi_{ij}\}$. A connection $\nabla$ on $\cE$ defines a covariant derivative, which is 
a differential operator 
\[
\nabla: \Gamma(\f{X}, \cE) \longrightarrow \Gamma(\f{X}, T^*\f{X}\otimes \cE)
\]
satisfying Leibniz's rule in the usual sense. 
   
 Recall that a  vector bundle over a Lie groupoid $\cG = (\cG_1 \rightrightarrows \cG_0)$ is a $\cG$-vector bundle $E$  over $\cG_0$, that is, a  vector bundle $E$  with a fiberwise linear action of $\cG_1$   covering the canonical  action of $\cG_1$ on $\cG_0$. The corresponding principal bundle can be described in terms of  a Hilsum-Skandalis morphism
 \[
  \xymatrix{\cG \ar@{-->}[r] & GL(k) } 
  \]
  where $k$ is the rank of $E$, and  the general linear group $GL(k)$, over $\R$ or $\C$,  is treated as  a Lie groupoid 
   $GL(k) \rightrightarrows \{e\}$. See \cite{CarW:2011}  for details.
  
  One can check that an orbifold vector bundle over  $\f{X}$  defines a canonical vector bundle over the orbifold  groupoid $\cG$.  
  The orbifold $K$-theory of $\f{X}$, denoted by $K^0_{orb}( \f{X})$,   is defined to be the Grothendieck ring of stable isomorphism classes of   complex  oribifold  vector bundles over $\f{X}$.   Let 
 $K^0(\cG)$ be the  Grothendieck ring of stable isomorphism classes of  complex  vector bundles over $\cG$. Then we have an obvious isomorphism
 \[
 K^0_{orb}( \f{X}) \cong K^0(\cG).
 \]
 Here, $\cG$ is the canonical orbifold groupiod associated to  $\f{X}$, so there is no ambiguity
 in the above isomorphism caused by another   Morita equivalent proper   \'etale Lie groupoid.

Given an orbifold groupoid $ \cG = (\cG_1 \rightrightarrows \cG_0)$ for an orbifold $\f{X}$, 
there are two canonical  convolution  $C^*$-algebras:  
the reduced and maximal  $C^*$-algebras, denoted by $C_{red}^*(\f{X})$ and  $C_{max}^*(\f{X})$.  For readers' benefit, we recall the definition from  Chapter 2.5 in \cite{Connes:1994} and \cite{Ren:1980}.

\begin{definition} {\bf  (Reduced and maximal  $C^*$-algebras)}\label{GroupCstarAlgebra}  
The reduced  $C^*$-algebra $C_{red}^*(\f{X})$ is the completion of the convolution algebra $C_c^\infty(\cG_1)$  of smooth compactly supported functions on $\cG_1$ for the norm 
\[
\|f\| =\sup_{x\in \cG_0} \|\pi_x(f)\|
\]
for $f\in C_c^\infty(\cG_1)$. Here $\pi_x$ is the involution representation of $C_c^\infty(\cG_1)$ 
in the Hilbert space $l^2(s^{-1}(x))$.   The maximal  $C^*$-algebra $C_{max}^*(\f{X})$ is
 the completion of the convolution algebra $C_c^\infty(\cG_1)$ for the norm
 \[
 \|f\|_{max} = \sup_\pi \{\|\pi (f)\|\}
 \]
 where the supremum is taken over all possible Hilbert space representation of $C_c^\infty(\cG_1)$. 
\end{definition}

\begin{remark} For an effective  or a presentable orbifold, it is straightforward to see that there is an isomorphism
\[
K^*_{orb}(\f{X}) \cong K_*(C_{red}^*(\f{X})),
\]
where $K_*(C_{red}^*(\f{X}))$ is the $K$-theory of the  reduced  $C^*$-algebra $C_{red}^*(\f{X})$.  For a non-effective orbifold,
this becomes a very subtle issue. Nevertheless, the above isomorphism still holds for general orbifolds as established in 
\cite{TaTsW}. 
\end{remark}

\subsection{Kawasaki's index theorem for compact orbifolds}
\label{sec:Kawasaki's index}

The Atiyah-Singer type index theorem for elliptic operators over a compact orbifold was established by Kawasaki in \cite{Kawasaki:1981}.
When the underlying compact orbifold is good, higher index was studied in \cite{Farsi:1992-1} and \cite{Bunke:2007}, and 
twisted $L^2$-index  with trivial Dixmier-Douady invariant was developed   in \cite{MarMat}.  In this subsection, we will review the original  Kawasaki's orbifold index theorem using modern language of twisted sectors and delocalized characteristic classes.

 Let $\f{X}  = (|\f{X}|, \cO) $ be an orbifold. Then the set of pairs
\[
 \{ (x, (g)_{H_x})|  x\in |v|, g\in H_x\},
\]
where $(g)_{H_x}$ is the conjugacy class of $g$ in the local group $H_x$, has a natural 
orbifold structure given by
\begin{equation}
\label{eq:TwistedSectorLocalStructure}
\{ \bigl(\tilde U^{g}, Z(g), \tilde \pi \big) |  g\in H_x\}. 
\end{equation}
Here for each orbifold chart $(\tilde U,  H,  \pi) \in \cO$  around  $x$,  $Z (g)$ is the centralizer of $g$ in $H_x$ and 
$\tilde U^{g}$ is the fixed-point set of $g$ in $\tilde U$.  This orbifold, denoted by $I \f{X} $, is called the {\bf inertia orbifold} of $\f{X}$. The inertia orbifold $I\f{X}$  consists of a  disjoint union of sub-orbifolds of $\f{X}$.  There is a canonical orbifold immersion 
\[
ev:  I\f{X}\longrightarrow   \f{X}
\]
given by the obvious inclusions $\{\bigl(\tilde U^{g}, Z(g), \tilde \pi \big) \to (\tilde U,  H,  \pi)\}$.

To describe the connected components of $I\f{X}$, we need to introduce an equivalence relation on the
set of conjugacy classes in local groups as in \cite{ALR}.  For each $x\in \f{X}$, let 
$(\tilde U_x,  H_x,  \pi_x, U_x)$ be a local orbifold chart at $x$. If $y\in U_x$, then up to conjugation, there is an injective homomorphism of local groups $H_y\to H_x$. Under this injective map,  
the conjugacy class $(g)_{H_x}$  in $H_x$ is well-defined for $g\in H_y$. We define the equivalence to be generated
by the relation $(g)_{H_y}\sim (g)_{H_x}$. Let $\cT_{\f{X}}$ be the set of equivalence classes. Then the inertia orbifold is given by
\[
I\f{X}  = \bigsqcup_{(g) \in \cT_{\f{X}}} \f{X}_{(g)},
\]
where $\f{X}_{(g)} =\{(x, (g')_{H_x})| g'\in H_x, (g')_{H_x} \sim  (g)\}$.
 Note that
$\f{X}_{(e)} =\f{X}$ is called the non-twisted sector and  $\f{X}_{(g)}$ for $g\neq e$ is called a {\bf  twisted sector}
of $\f{X}$.  

Given a complex orbifold vector bundle $E$ over $\f{X}$,     the pull-back bundle $ev^*E$ over $I\f{X}$  has a canonical  automorphism $\Phi$.
With respect to a  Hermitian metric on $E$,  there is an eigen-bundle decomposition
of $ev^*E$
\[
ev^* E = \bigoplus_{\theta \in \Q\cap [0, 1) } E_\theta
\]
where  $E_\theta$ is a complex  orbifold vector bundle over $I\f{X} $,  on which $\Phi$ acts  by multiplying $e^{2\pi \sqrt{-1} \theta}$.   Define the {\bf delocalized Chern character} of $E$ by
\[
\ch_{deloc} (E) = \sum_\theta e^{2\pi \sqrt{-1} \theta} \ch(E_\theta) 
\in     H^{ev}_{orb}(I\f{X}, \C),
 \]
 where $\ch(E_\theta) \in  H^{ev} (I\f{X}, \C)$ is the ordinary Chern character of  $E_\theta$.
 The odd delocalized Chern character
 \[
 \ch_{deloc}: K^1_{orb}(\f{X}) \longrightarrow H^{odd} (I\f{X}, \C)
\]
can be defined in the usual way. The following proposition is established in \cite{HuWa}.  

\begin{proposition}  {\cite[Proposition 2.7]{HuWa}}
For any compact    orbifold $\f{X}$, the delocalized Chern character 
  gives a ring isomorphism 
  \[
  \ch_{deloc}: K^*_{orb}(\f{X}) \otimes_\Z \C  \longrightarrow H^{*} (I\f{X}, \C)
  \]
  over $\C$. 
\end{proposition}

Let $\f{X}$ be a compact Riemannian  orbifold with an orbifold atlas $\cO=\{(\tilde U_i, H_i, \pi)\}$. Then the  bundles    of Clifford algebras over   $\tilde U_i$'s,  
$
\{(\Cliff(T\tilde U_i) \to \tilde U_i,  H_i)\},
$
 whose fiber at $x\in \tilde U_i$ is the real Clifford algebra $\Cliff(T_x \tilde U_i)$,   
define  an {\bf orbifold Clifford bundle}, denoted by $\Cliff(T\f{X})$.   As in the manifold case ({\em cf.} Chapter 3.3 in \cite{BGV}), there is a connection on $\Cliff(T\f{X})$, induced from the Levi-Civita connection $\nabla^{T\f{X}}$,  which is compatible with the Clifford multiplication on $\Gamma(\f{X}, \Cliff(T\f{X}))$.

\begin{definition} 
\label{def:DiracBundle}
 Let  $\f{X}$ be a compact Riemannian even-dimensional  orbifold with an orbifold atlas $\cO=\{\tilde U_i, H_i, \pi)\}$.
\begin{enumerate} \item  A {\bf Dirac}  bundle over $\f{X}$ is a $\Z/2\Z$-graded Hermitian  orbifold  vector bundle $\cE$ 
 with a graded self-adjoint  smooth action of $\Cliff(T\f{X})$ in the sense that for each orbifold chart
$(\tilde U_i, H_i, \pi)$, the action of $v\in \Cliff(T_x \tilde U_i)$ is  of degree one and skew-adjoint.  
\item A connection $\nabla^\cE$ on a Dirac bundle is called a Clifford connection if for any $v\in \Gamma (\f{X}, \Cliff(T\f{X}))$ and
$\xi \in \Gamma (\f{X}, T\f{X})$,
\[
[\nabla_\xi^\cE, c(v)] =  c( \nabla_\xi v).
\]
Here $c$ is denote the Clifford action $\Cliff(T\f{X}) \otimes \cE \to \cE$.
\end{enumerate}
\end{definition}

\begin{remark} \begin{enumerate}
\item By means of a smooth partition of unity on $\f{X}$, any Dirac bundle over $\f{X}$
 admits a Clifford connection. The space of all 
Clifford connections is an affine space modelled on $\Gamma(\f{X}, T^*\f{X}\otimes \End_{\Cliff(T\f{X})}(\cE) )$. Here
$\End_{\Cliff(T\f{X})}(\cE) $ is the bundle of  degree $0$ endomorphisms which commute with the Clifford action. 
\item  For an even dimensional  compact Riemannian  orbifold $\f{X}$,  $\f{X}$ is Spin$^c$ if and only if there exists an  orbifold  complex vector bundle $\cS$, called a complex spinor bundle,  such that 
the complexified version of $\Cliff(T\f{X})$, denoted by $\Cliff_\C(T\f{X})$, is isomorphic to $\End_\C(\cS)$.  A choice of  such an $\cS$ is called a Spin$^c$ structure on $\f{X}$.  Then any Dirac bundle 
$\cE$ can be written as $\cS \otimes   \cW$ for an orbifold complex vector bundle $\cW$ over $\f{X}$. 
\item The Grothendieck group of Dirac bundles over $\f{X}$ forms an abelian group, denoted by $K^0(\f{X}, \Cliff_\C(T\f{X}))$.  The  relative delocalized Chern character
\[
\ch_{deloc}^\cS:  K^0_{orb}(\f{X}, \Cliff_\C(T\f{X})) \longrightarrow  H^{ev} (I\f{X}, \R) 
\]
can be defined by the same construction as in Section 4.1 of \cite{BGV}. When $\f{X}$ is equipped with a
Spin$^c$ structure $\cS$,  then
\[
K^0_{orb}(\f{X}, \Cliff_\C(T\f{X}))  \cong K^0_{orb}(\f{X}),
\]
and  any  Dirac bundle $\cE$ written as $\cS\otimes \cW$,  we have 
\[
\ch_{deloc}^\cS (\cS\otimes \cW ) = \ch_{deloc} (\cW).
\]
\end{enumerate}
\end{remark}
 
Given a Dirac bundle $\cE$ with a Clifford connection $\nabla^\cE$ over a compact Riemannian even-dimensional  orbifold $\f{X}$,  the 
Dirac operator  on $\cE$ is defined to be the following composition
\[
\Dirac^\cE:  \xymatrix{ \Gamma (\f{X}, \cE)  \ar[r]^{\nabla^\cE \qquad   \qquad } & \Gamma  (\f{X}, T^*\f{X}\otimes
\cE ) \cong \Gamma  (\f{X}, T\f{X}\otimes
\cE )   \ar[r]^{\qquad   \qquad c} & \Gamma (\f{X}, \cE).}
\]
When $\f{X}$ is a spin$^c$ orbifold with its spinor bundle $\cS$, then $\cE = \cS\otimes \cW$ for a complex orbifold vector bundle $\cW$, and  $\Dirac^\cE$ will be written as $\Dirac^{\cW}$.  With respect to the $\Z/2\Z$-grading on $\cE = \cE_+ \oplus \cE_-$,
\[
\Dirac^\cE =  \left[ \begin{matrix}
0 &\Dirac^\cE_- \\
\Dirac^\cE_+  &0  \end{matrix}\right]
\]
where $\Dirac^\cE_-$ is the formal adjoint of $\Dirac^\cE_+$. 
Let  $ L^2  (\f{X}, \cE)$ and $ L^2_1  (\f{X}, \cE) $ be   the completions of   $ \Gamma  (\f{X}, \cE)$ with respect to $L^2$-norm and $L_1^2$-norm respectively. 
 Then 
\[
\Dirac^\cE : L^2_1  (\f{X}, \cE) \longrightarrow  L^2  (\f{X}, \cE) 
\]
is a Fredholm operator by showing that the unbounded operator 
\[
\Dirac^\cE:    L^2  (\f{X}, \cE) \longrightarrow  L^2  (\f{X}, \cE)
\]
 has  compact resolvent.  One could also establish the finite dimensionality of $\ker \Dirac^\cE$ by showing
 that the heat operator $e^{-t(\Dirac^\cE)^2}$ is a trace-class operator for any compact Riemannian  orbifold $\f{X}$.  This latter claim follows from the   fact that   $e^{-t(\Dirac^\cE)^2}$ has a smooth kernel. In  next section  we will prove  the  smooth  heat  kernel for any complete  Riemannian  orbifold, so we omit the details here.

The Fredholm index of $\Dirac^\cE$ is given by 
\[
\ind \Dirac^\cE = \dim_\C \ker \Dirac^\cE_+ -\dim_\C  \ker \Dirac^\cE_-. 
\]
To rephrase  the orbifold index theorem  in \cite{Kawasaki:1981} ({\em cf.}  \cite{Bunke:2007}), we introduce the notations for
the delocalized A-hat class and  the delocalized Todd class of   $\f{X}$
 as follows.  For simplicity, we assume that each twisted sector $\f{X}_{(g)}$ is connected, otherwise, these characteristic classes are defined on each connected component.  Considering the orbifold immersion 
 \[
 ev:  I \f{X} \longrightarrow  \f{X},
 \]
 we can decompose the pull-back of the  orbifold tangent bundle  along each twisted sector $\f{X}_{(g)}$ (using the
 Riemannian metric)
 \[
 T\f{X} |_{\f{X}_{(g)}}  \cong T\f{X}_{(g)} \oplus \cN_{(g)}
 \]
 such that the pull-back of the Riemannian curvature of $T\f{X}$ is decomposed as
 \[
ev^*  R^{T\f{X}}  = R^{T\f{X}_{(g)}} \oplus   R^{\cN_{(g)}}
\]
where $ R^{T\f{X}_{(g)}}$ is the Riemannian curvature of $\f{X}_{(g)}$ and $R^{\cN_{(g)}}$ is the curvature of the  induced connection   
on $\cN_{(g)}$.  Then  the delocalized A-hat class of $\f{X}$ is a cohomology class on $I\f{X}$. When  restricted to  each  twisted sector $\f{X}_{(g)}$,   it    is  defined by 
\[
   \hat{A}_{deloc}(\f{X}) = \dfrac{ \hat{A}(\f{X}_{(g)}) }{[ \det \big(1-\Phi_{(g)}  e^{R^{\cN_{(g)}}/2\pi i}  \big) ]^{1/2}} .
 \]
 Here $\hat{A} (\f{X}_{(g)})$ is the  A-hat form of $\f{X}_{(g)}$ defined by
 \[
\left[ \det  \left( \dfrac {R^{T\f{X}_{(g)}}  /  4\pi i } {\sinh (R^{T\f{X}_{(g)}} /  4\pi i  ) }  \right) \right] ^{1/2} 
 \]
 and $\Phi_{(g)}$ is  the automorphism of $\cN_{(g)}$ defined by the inertia orbifold structure.
If  $I\f{X}$ is a  spin$^c$  orbifold,  the delocalized Todd class of   $\f{X}$ is a cohomology class on $I\f{X}$. When   restricted to  each  twisted sector $\f{X}_{(g)}$,  it   is defined by  the closed differential form 
\[
\Td_{deloc}(\f{X})   = \dfrac{  \Td(\f{X}_{(g)})   }{\det \big(1 - \Phi_{(g)}^{-1}  e^{ R^{\cN_{(g)}}/2\pi i}  \big)},
 \]
  where $\Td(\f{X}_{(g)}) =  \det  \left ( \dfrac { iR^{\f{X}_{(g)}}/2\pi }{ 1- e^{-iR^{\f{X}_{(g)}}/2\pi  }} \right)$ is the usual Todd  form  of the spin$^c$ orbifold $\f{X}_{(g)}$.  

\begin{theorem} \label{index:compact}   Let $\f{X}$ be a compact Riemannian even-dimensional  orbifold. Let $\Dirac^\cE$ be a Dirac operator on a Dirac bundle
$\cE$ over $\f{X}$.  Then   index of $\Dirac^\cE$ is given by 
\[
 \ind \Dirac^{\cE} = \int_{I\f{X}}^{orb}  \hat{A}_{deloc}(\f{X})  \ch^\cS_{deloc}(\cE).
  \]
In particular, if  $I \f{X}$  is a  spin$^c$  orbifold and $\cW$ is a  orbifold Hermitian vector bundle    with a Hermitian
  connection.  Then the index of  the spin$^c$ Dirac operator $\Dirac^{\cW}_+$ is given by the formula
    \[
  \ind\Dirac^{\cW} = \int_{I\f{X}}^{orb}  \Td_{deloc}(\f{X})  \ch_{deloc}(\cW).
  \]
  
 \end{theorem}

\begin{remark} One can adapt the proof of  Theorem 4.1 in \cite{BGV}  to compact orbifolds to establish the  local  index version of Theorem  \ref{index:compact}. It amounts to study the heat kernel asymptotics for $e^{-t(\Dirac^\cE)^2}$ on the compact orbifold.  We will apply the same  heat kernel approach to establish a more refined index formula for a discrete group action on a complete Riemannian orbifold, so we will not reproduce a proof here. 
\end{remark}

\section{Invariant Elliptic Operators on Complete Orbifolds}
\label{sec:EllipticPDO}

In this section, after a review of invariant elliptic pseudo-differential operators in Section~\ref{sec:InvElliPDO}, we prepare some analytical tools to study index of Dirac type operators on a complete even dimensional Riemannian  orbifold $\f{X}$, where a discrete group $G$ acts properly, co-compactly and isometrically.
When the orbifold is compact, the index of the Dirac operator $\Dirac^{\cE}$ is known to be calculated by the supertrace of the corresponding heat operator. 
However, when $\f{X}$ is not compact, the convergence of the heat kernel asymptotic expansion as $t\to0^+$ has to be modified accordingly.
This is worked out in Section~\ref{sec:DiracOp&HeatKernelAsym}. 
Also, the operator trace of $e^{-t(\Dirac^{\cE})^2}$ does not make sense. In Section~\ref{sec:cut-off(g)trace}, we introduce $(g)$-trace class operators and in Section~\ref{(g)-traceHeatOp} we compute the
$(g)$-trace  for the heat operator. 
On the one hand, these traces are related to the orbifold index of $\Dirac^{G\backslash \cE}$ on the quotient. 
On the other hand, as we shall see in Section~\ref{sec:Localized indices}, they are topological invariants for $\f{X}$ coming from the higher index for $\Dirac^\cE.$

\subsection{Invariant elliptic pseudo-differential operators}
\label{sec:InvElliPDO}

Let $\f{X}$ be a complete $n$-dimensional oriented Riemannian orbifold   and let $\cE$ be  a Hermitian orbifold vector bundle
with a Hermitian  connection $\nabla$. 
  Denote by $ \Gamma_c (\f{X}, \cE) \subset   \Gamma(\f{X}, \cE)$  the subspaces of  smooth sections of $\cE$
with  compact  support.   
A differential operator 
\[
D:   \Gamma_c (\f{X}, \cE) \longrightarrow   \Gamma(\f{X}, \cE)
\]
of order $m$ is a linear map such that for any orbifold chart $(\tilde U_i, H_i, \pi_i, U_i)$, the operator $D$ is locally represented by
\begin{equation*}
 \tilde D_i:  \Gamma  (\tilde U_i, \tilde\cE|_{\tilde U_i})^{H_i} \longrightarrow   \Gamma(\tilde U_i, \tilde\cE|_{\tilde U_i})^{H_i}, 
\end{equation*}
 an $H_i$-invariant differential operator  of order $m$.  
The completion of $ \Gamma_c (\f{X}, \cE)$ in the Sobolev $k$-norm 
\ba\label{k-norm}
\|\psi\|^2_k =\sum_{i=0}^k \int_{\f{X}}^{orb} | \nabla^i \psi |^2 d\vol_{\f{X}}
\na
is the Sobolev space  denoted by $L^2_k(\f{X}, \cE)$. Here $d\vol_{\f{X}}$ is the Riemannian volume element 
(the unique $n$-form of unit length) on
$\f{X}$  defined by the metric.      Then a differential operator 
\[
D:   \Gamma_c (\f{X}, \cE) \longrightarrow   \Gamma(\f{X}, \cE)
\]
of order $m$ extends to a bounded linear map $D: L^2_{k} (\f{X}, \cE) \to L^2_{k-m} (\f{X}, \cE)$
for all $k\geq m$.   
A  linear map $A:     \Gamma_c (\f{X}, \cE) \longrightarrow   \Gamma(\f{X}, \cE)$
is  called a  {\bf smoothing   operator}   if $A$ extends to a bounded linear map $A: L^2_{k} (\f{X}, \cE) \to L^2_{k+m} (\f{X}, \cE)$  for all $k$ and $m\geq 0$. The Sobolev embedding theorem implies that, for a smoothing 
operator $A$, we have
\[
A( L^2_{k} (\f{X}, \cE)  )\subset  \Gamma(\f{X}, \cE)
\]
for all  $k$.  

\begin{definition}[\bf Pseudo-differential operators  on an orbifold]
\label{def:PsiDO}
A  linear map
\[
D:  \Gamma_c (\f{X}, \cE) \longrightarrow \Gamma(\f{X}, \cE)
\]
is a 
 {\bf pseudo-differential operator }  of order $m$  if modulo smoothing operators,  for any orbifold chart  $\{( \tilde U_i, H_i, \pi_i, U_i)\}$,
 the operator $D$ is represented by  
 \begin{equation}
 \label{eq:OperatroLocal}
\tilde D_i :   \Gamma_c  (\tilde U_i, \tilde\cE|_{\tilde U_i})  \longrightarrow   \Gamma  (\tilde U_i, \tilde\cE|_{\tilde U_i}) 
\end{equation}
which is an  $H_i$-invariant pseudo-differential operator  of  order $m$. 
  A pseudo-differential operator $D$ is {\bf elliptic} 
 if  $\tilde D_i  $   is elliptic  for each
 orbifold chart  $(\tilde U_i, H_i)$.
 Denote by $\pdo_{orb}^m(\f{X}, \cE)$ the linear space  of  all pseudo-differential operators of order $m$ on $(\f{X}, \cE)$. \end{definition}

Due to the local nature of pseudo-differential operators, an operator
$D\in  \pdo_{orb}^m(\f{X}, \cE)$ has a well-defined  principal symbol of order $m$ which is an element in
\[
Sym^m_{orb}(\f{X}, \cE) / Sym^{m-1}_{orb}(\f{X}, \cE)
\]
where $Sym^m_{orb}(\f{X}, \cE)$ is the space of all sections of  order $m$ of the orbifold  vector bundle $\End(p^*\cE)$ over $T^*\f{X}$, here
$p: T^*\f{X} \to \f{X}$ is the obvious projection.  
Associated to a pseudo-differential operator   $D$ of  order  $m$,  the  Schwartz kernel 
\[
k  \in {\mathcal D}' (\f{X}\times \f{X}, \cE\boxtimes \cE^*)
\]
as a distributional-valued  section, is smooth off the diagonal. Here
\[
 \cE\boxtimes \cE^* =\pi_1^* \cE \otimes \pi_2^*\cE^*
 \]
with    the projections  $\f{X}\times \f{X} \to \f{X} $ to the first  factor and the second factor
respectively. As explained in \cite{Ma},  in terms of local orbifold chart $(\tilde U_i, H_i, \pi_i, U)_i$, there is a   distributional-valued  section
 \[
 \tilde k_D (\tilde x, \tilde y) \in \Gamma (\tilde U_i \times \tilde U_i, \tilde \cE_i\boxtimes \tilde \cE^*_i),
 \]
 which is the kernel for the representing operator $\tilde D_i$ in~(\ref{eq:OperatroLocal}), 
 such that
 \ba\label{local:kernel}
k_D (x, y) =  \sum_{h\in H_i} h \tilde k_D( \tilde x \cdot h^{-1}, \tilde y)
 \na
 where $\tilde x \in \pi_i^{-1} (x)$ and $\tilde y \in \pi_i^{-1} (y)$.  Here we assume that $H_i$ acts on  $\tilde U_i $ from the right.

A smoothing  pseudo-differential operator   has a smooth
kernel, that is  a smooth section $k(x, y)$ of the bundle $\cE\boxtimes \cE^* = \pi_1^*(\cE)  \otimes \pi^*_2 (\cE^*)$ over $\f{X} \times  \f{X}$   such that
\[
A\psi  (x) = \int^{orb}_{\f{X}} k(x, y) \psi(y) d\vol_{\f{X}}(y). 
\]

When $\f{X}$ is equipped with a proper,   co-compact and isometric  action of a discrete group
$G$,  we can define  $G$-invariant  pseudo-differential operators  on  $\cE$ in the usual sense.  Let $\cE$  be  an orbifold Hermitian vector bundle with a $G$-invariant   Hermitian  connection $\nabla$.  Consider a collection of  orbifold  charts  of $\f{X}$
\[
\{ (G\times _{G_i} \tilde U_i, H_i,  G\times _{G_i}   U_i)\}  
\]
as provided by  Lemma \ref{lem:quotient}. Here $\{(G\times _{G_i} \tilde U_i, H_i,  G\times _{G_i}   U_i)\}$ is a collection 
of orbifold charts indexed by  the coset space $G/G_i$.    We further assume that $ U_i$ is an open covering of a relatively  compact
open set $C$ in $|\f{X}|$ such that $\f{X} = G \cdot C$. A  $G$-invariant  pseudo-differential operator
\[
D_{\f{X}}: \Gamma_c (\f{X}, \cE) \longrightarrow \Gamma(\f{X}, \cE)
\]
 is represented by a  $G$-invariant  pseudo-differential operator on  the orbifold $ G\times _{G_i}   U_i$. 
Note that  the orbifold structure on  the quotient space $G \backslash |\f{X}|$ is defined by
the collection of  finitely many orbifold charts
\[
\{ ( \tilde U_i,  G_i\times H_i,   U_i)\}.
\]
Then a $G$-invariant pseudo-differential operator    $D_{\f{X}}$ of order $m$ on  $\f{X}$ defines a  pseudo-differential operator   of order $m$ on  the compact orbifold $G\backslash \f{X}$.  

\begin{remark}
\label{rek:GinvKernel}
Let $k(x,y) \in \cE_x \otimes  \cE^*_y$ be the distributional kernel of a $G$-invariant pseudo-differential operator $D_{\f{X}}$. Then we have 
\begin{equation}
\label{eq:GinvKernel}
g k(g^{-1}x, y)=k(x,gy)g\qquad \forall x, y\in\f{X}, \forall g\in G. 
\end{equation}
Here, $g$ stands for the action of $g\in G$ on  the smooth sections in $\Gamma(\f{X}, \cE).$ 
\end{remark}

We shall refer the pseudo-differential operators being studied are properly supported.
Recall that a {\bf properly supported} operator is an operator where the distributional kernel $k(x,y)$ satisfies the following property:  for all compact set $C\subset\f{X}$, the following set is compact in $\f{X}\times\f{X}:$
\begin{equation}
\label{eq:properSuppG-inv}
\{(x,y)\in\f{X}\times\f{X}| k(x,y)\neq0, x \text{ or } y\in C\}.
\end{equation}
If $S$ is a $G$-invariant properly sopported operator with Schwartz kernel $K_S$, then by Remark~\ref{rek:GinvKernel} and the co-compactness of the action, (\ref{eq:properSuppG-inv}) implies the existence of $R>0$ such that
\begin{equation*}
K_S(x,y)=0 \qquad \forall d(x,y)>R.
\end{equation*}
In view of Definition~\ref{def:PsiDO}, any pseudo-differential operator is properly supported up to a smoothing operator.

The following proposition is a key property for elliptic operators. When $\f{X}$ is compact, it implies that elliptic operator has the Fredholm index. In general, this proposition leads to  an elliptic operator model for $K$-homology and higher index.
As a pseudo-differential operator on an orbifold is locally defined on the orbifold charts, the proof are similar to the manifold case (\cite[Proposition~2.7]{W:2012}), there we assumed $D_{\f{X}}$ is properly supported and has order $0$. It is then sufficient to prove the proposition. In fact, any pseudo-differential operator of positive order can be normalized to an order $0$ operator and is properly supported up to a difference of a smoothing operator.

\begin{proposition}\label{prop:EllipticParametrix} Let $(D_{\f{X}})_+: L^2_{k}(\f{X}, \cE_+)\rightarrow L^2_{k-m}(\f{X}, \cE_-),$ where $k>m$, be a $G$-invariant elliptic operator with non-negative order $m$, then there exists a $G$-invariant parametrix $Q_{\f{X}}: L^2_{k-m}(\f{X}, \cE_-)\rightarrow L^2_{k}(\f{X}, \cE_+)$ of $(D_{\f{X}})_+$ so that $1-Q_{\f{X}}(D_{\f{X}})_+=S_0$ and $1-(D_{\f{X}})_+Q_{\f{X}}=S_1$ are smoothing operators.
\end{proposition}

\subsection{Dirac operator and heat kernel asymptotics}
\label{sec:DiracOp&HeatKernelAsym}

In this paper, we focus on Dirac operators, which are first order  elliptic  differential operators. 
Let  $\f{X}$ be a complete  even-dimensional Riemannian orbifold. 
Let $\nabla^{T\f{X}}$ be the Levi-Civita connection on $T\f{X}$.  Given  a
  Dirac bundle   $\cE = \cE_+ \oplus \cE_-$ with a Clifford connection 
$\nabla^\cE$. The Dirac operator
\begin{equation}
\label{eq:FormallySelfAdjointD}
\Dirac^\cE =  \left[ \begin{matrix}
0 &\Dirac^\cE_- \\
\Dirac^\cE_+  &0  \end{matrix}\right]:  L^2(\f{X}, \cE) \longrightarrow  L^2(\f{X}, \cE), 
\end{equation}
with domain  $\Gamma_c(\f{X}, \cE)$  is an essentially self-adjoint  elliptic differential operator.  Moreover, we have the Lichnerowicz formula for $(\Dirac^\cE )^2$
as follows.
\ba\label{D^2}
(\Dirac^\cE )^2 = \Delta^\cE + c(F^{\cE/\cS}) + \dfrac 14 r_{\f{X}} 
\na
where $ \Delta^\cE $ is the Laplace operator on $\cE$, $ c(F^{\cE/\cS})$ is the Clifford action of
the twisted curvature of the  Clifford connection  $\nabla \cE$  and 
$r_{\f{X}}$ is the scalar curvature of $\f{X}$. We remark that the proof of the  essentially self-adjointness of $\Dirac^\cE$ and
the Lichnerowicz formula for $(\Dirac^\cE )^2$ follow from the proof in the case of  smooth manifolds without essential change.
Then the heat operator 
\[
e^{-t(\Dirac^\cE )^2}:    L^2(\f{X}, \cE) \longrightarrow  L^2(\f{X}, \cE)
\]
 is a one -parameter semi-group of operators consisting of positive, self-ajoint operators of norm $\leq 1$,  satisfying  the following properties, for $\psi \in  L^2 (\f{X}, \cE)$, $t\geq 0$,
 \begin{enumerate}
\item $\left( \dfrac{d}{dt}+ (\Dirac^\cE )^2\right)  e^{-t(\Dirac^\cE )^2} \psi = 0$. \\[2mm]
\item $   \lim_{t\to 0}  e^{-t(\Dirac^\cE )^2} \psi = \psi$ in $ L^2 (\f{X}, \cE)$. 
\end{enumerate}
Moreover, $e^{-t(\Dirac^\cE )^2} $ is a smoothing operator. By the Schwartz kernel theorem, there is a kernel 
\[
K_t(x,  y) \in \Gamma(\f{X} \times \f{X}, \cE\boxtimes \cE^*),
\]
called the heat kernel of $\Dirac^\cE $ with respect to $d\vol_{\f{X}}$.   For $\psi \in  L^2(\f{X}, \cE)$,
\[
(e^{-t(\Dirac^\cE )^2}  \psi )  (x) =\int_{\f{X}}^{orb} K_t(x,  y)  (\psi (y)) d\vol_{\f{X}}(y). 
\]
When $\f{X}$ is compact, that is, $|\f{X}|$ is compact, then we have the following  asymptotic expansion: for each Riemannian orbifold chart $(\tilde U_i, H_i, U_i)$, using the notations in (\ref{local:kernel}),  there exists a smooth section $u_j$ of  $\cE\boxtimes \cE^*$ over $\tilde U_i \times \tilde U_i$ such that for
every $l> n=\dim \f{X}$ and $x, y\in U_i$, as $t\to 0^+$,
\ba\label{asym:1}\begin{array}{lll}
\qquad K_t(x, y) &= &  (4\pi t)^{-n/2} d\frac{1}{|H_i|} \sum_{h\in H_i}  e^{-\frac{\tilde d(\tilde x, \tilde y)^2}{4t} }   \sum_{j=0}^l t^j  h u_j (\tilde x \cdot h^{-1},
\tilde y)  \\[2mm]
&& + O(t^{l- n/2}).\end{array}
\na
Here  $\tilde d(\tilde x, \tilde y)$ is the distance function on $\tilde U_i$ defined by the Riemannian metric on $\tilde U_i$. 
Moreover, the off-diagonal estimate is given
by
\[
K_t(x, y) = O( e^{-a^2/4t})
\]
as $t\to 0$,  for  $(x, y) \in | \f{X}|  \times |\f{X}| $ with $d (x, y) >a > 0$. Here $d(x, y)$ is the distance function defined by the
 Riemannian metric on $\f{X}$.  See \cite{Ma} and Chapter 5.4 in \cite{MaMa} for  relevant discussions.  The local index
 technique as \cite{BGV} and \cite{MaMa} gives rise to a local index formula for the Kawasaki orbifold index theorem (Theorem~\ref{index:compact}). In addition, it worth to mention that the Geiner-Hadamard approach to the heat kernel asymptotic in~\cite{Po:CMP, PW:NCGCGI.PartII} gives rise to a concise proof for the heat kernel asymptotic in this setting and leads to an elementary calculation deriving Theorem~\ref{index:compact}.

When $\f{X}$ is a   complete  even-dimensional Riemannian  orbifold equipped with  a proper, co-compact and isometric action of a 
discrete group $G$.    Let $\nabla^{T\f{X}}$ be the $G$-invariant Levi-Civita connection on $T\f{X}$. Then
the scalar curvature $r_{\f{X}}$  is a $G$-invariant function on $\f{X}$.   
Given  a
$G$-equivariant Dirac bundle   $\cE = \cE_+ \oplus \cE_-$ with a $G$-invariant Clifford connection 
$\nabla^\cE$, then the Dirac operator
\[
\Dirac^\cE:   L^2(\f{X}, \cE) \longrightarrow  L^2(\f{X}, \cE)
\]
is a $G$-invariant elliptic differential operator.  The corresponding Dirac operator  on $G\backslash \f{X}$ 
is denoted by
\[
\Dirac^{G\backslash \cE}:   L^2(G\backslash \f{X}, G\backslash \cE) \longrightarrow  L^2(G\backslash \f{X}, G\backslash \cE).
\]

\begin{theorem}
\label{thm:HeatKernelAsympMain}
 Let $\f{X}$ be  a   complete  even-dimensional Riemannian orbifold with  a proper, co-compact and isometric action of a 
discrete group $G$. 
Let $K_t(x, y)$ be the heat kernel of  a $G$-invariant elliptic differential operator $\Dirac^\cE$,  that is, the Schwartz kernel of   the semigroup 
$e^{-t(\Dirac^\cE )^2}$.  Then  $K_t(x, y)$ is a $G$-invariant   smooth kernel with 
 an asymptotic expansion, as $t\to 0^+$,
\ba\label{asymp:exp}
K_t(x, y) \sim (4\pi t)^{-n/2}  e^{-\frac{d(x, y)^2}{4t}} \sum_{j=0}^\infty t^i u_j (x, y)
\na
on a sufficiently small neighbourhood of the diagonal in $\f{X}\times \f{X}$. Here
$u_j(x, y)$ is a smooth section of $\cE\boxtimes \cE^*$ in the sense of (\ref{local:kernel}). 
Moreover,  choose  a relative compact sub-orbifold  $C$ 
in $\f{X}$ with $\f{X} = G\cdot C$.   Let $\pi: C \to G \backslash  \f{X}$ be the natural map  defined by the quotient
map $\f{X} \to G \backslash  \f{X}$. Then the series 
\[
\sum _{g\in G}  g  K_t  (g^{-1}  \cdot x,   y)   =  \sum_{g\in G}  K_t  (x, g\cdot y) g \in \cE_x\otimes \cE^*_y
\]
converges uniformly on $[t_1, t_2] \times C \times C$ to the kernel 
$\overline{K}_t (\bar x, \bar y)$ of $e^{-t(\Dirac^{G\backslash \cE})^2}$ for $\bar{x} = \pi(x)$ and 
$\bar{y} = \pi(y)$. Here $g$ acts on a section $\psi\in L^2(\f{X}, \cE)$ by $(g \cdot \psi)(y) = g \psi(g^{-1}\cdot  y)$. 
\end{theorem}
\begin{proof}  In the case of smooth manifolds, this theorem is due to Donnelly \cite{Do}.    For convenience, we use the same notation $(\Dirac^\cE )^2$ to denote its unique self-adjoint
extension.    By 
 Lichnerowicz formula,  $(\Dirac^\cE )^2$ is a generalized Laplacian operator 
 \[
  \Delta^\cE + c(F^{\cE/\cS}) + \dfrac 14 r_{\f{X}} 
  \]
  with $c(F^{\cE/\cS}) + \dfrac 14 r_{\f{X}} $ is bounded from below due to the isometric co-compact action of $G$ on $\f{X}$.  Then by 
  Theorem D.1.2 in \cite{MaMa} and  the standard functional calculus, 
  \[
  e^{-t(\Dirac^\cE )^2}:  L^2(\f{X}, \cE) \longrightarrow  L^2(\f{X}, \cE),
  \]
   is a smoothing operator and   the heat kernel $K_t(x, y) $ is a smooth section  of $\cE\boxtimes \cE^*$ in $x, y \in \f{X}$ and $t\in (0, \infty)$.  

 The asymptotic expansion follows from the uniqueness  of the fundamental solution to the heat equation 
 on $(0, \infty) \times \f{X} \times \f{X}$, 
\[
\left( \dfrac{\partial }{\partial t} + (\Dirac_y^\cE )^2 \right) K_t(x, y) = 0
\]
 where $(\Dirac_y^\cE )^2 $ acts in the second variable, with the following  initial boundary condition at $t=0$:  if $\psi$ is a smooth section of
$\cE$, then 
\[
\lim_{t\to 0} \int_{\f{X}}^{orb}   K_t(x, y) \psi(y) d\vol_{\f{X}}(y)  =\psi(x) 
\]
in the $L^2$-norm or the uniform norm on any compact set in $\f{X}$.    It is done by building approximate heat kernel in  normal
orbifold coordinate charts as the asymptotic expansion in question  holds on a sufficiently small neighbourhood of the diagonal in $\f{X}\times \f{X}$.  Then the asymptotic expansion (\ref{asymp:exp})  follows from the  construction in \cite[Chapter 2]{BGV}  and \cite{CGT} carried over to the case of orbifolds. Moreover, for any $T>0$ and $0< t\leq T$, we have
the following fiberwise norm estimate
\[
\|( \dfrac{\partial}{\partial t})^i (\nabla^\cE_x)^j ( \nabla^\cE_y)^k    K_t(x, y)   \| \leq c_1  t^{-n/2 -i-j/2-k/2}  e^{-\frac{d(x, y)^2}{4t}}.
\]

To check that the series $\sum_{g\in G}  K_t  (x, g\cdot y) \circ (1, g^{-1})$ converges uniformly on $[t_1, t_2] \times C \times C$, we need to
estimate the number 
\[
n (i)  =  \# \{ g \in G    |  (i-1)r \leq    d(x, g\cdot y)   <  i r \}
\]
  for $i\in \N$,   $x, y \in C$ with $   r >  diam (C)$. As the $G$-action is proper, let 
  \[
m=   \# \{g  \in G| (  g\cdot C ) \cap C \neq \emptyset\}. 
  \]
  
  Suppose  that $B(x, r)$ contains $p_r$ $G$-translates of $y$, that is,
  \[
  p (r) =\#\{ g    |  d(x, g\cdot y) <r\}.
  \]
  Then $B(x, 2r)$ contains  $p(r)$  $G$-translates of $C$ as $r> diam (C)$.  Any point of $\f{X}$ is contained
  in at most $m$ translates of $C$, so we have
  \[
  p(r)  \cdot \vol(C) \leq m \cdot \vol(B(x, 2r)).
  \]
  Note that  $G$ acts on $\f{X}$ 
isometrically and  co-compactly, so the sectional curvatures of $\f{X}$ is bounded from below. By comparing with a space of constant curvature ({\em cf.}  \cite[Proposition 20]{Bo} and \cite[p.256 Corollary 4]{BiCr}), we have 
\[
  \vol(B(x,  r)) \leq c_3 e^{c_4 r}. 
\]
This implies that $ n(i) <  p(ir) < \dfrac{m c_3 e^{2c_4 ir}}{\vol (C)}$ for any $i\in \N$. 
Then by partitioning $G$ into subsets 
\[
G (i)  =   \{
  g \in G    |  (i-1)r \leq    d(x, g\cdot y)   <  i r \}, 
\]
for $i\in \N$, we have the following operator norm estimate 
\[\begin{array}{lll}
&& \disp{\sum_{g\in G} }  \| K_t  (x, g\cdot y) g\| \\[2mm]
 &  \leq &  c_1 t^{-n/2}\disp{ \sum_{g\in G}} e^{-\frac{d^2(x, g\cdot y)} {4t} }
 \\[2mm]
 & =  &  c_1 t^{-n/2} \disp{\sum_{i=1}^\infty  \sum_{g\in G (i)} } e^{-\frac{ -(i-1)^2 r^2} {4t} }\\[2mm]
 &  \leq &  c_1 \disp{t^{-n/2} \sum_{i=1}^\infty n_i  }  e^{-\frac{ -(i-1)^2 r^2} {4t} }\\[2mm]
 &  \leq & \disp{  \dfrac{m c_1c_3  }{\vol (C)} t^{-n/2} \sum_{i=1}^\infty e^{ 2c_4 i r}}   e^{-\frac{-(i-1)^2 r^2} {4t} } 
 \\[3mm]
 & =  &  \disp{ \dfrac{m c_1c_3  }{\vol (C)} t^{-n/2} \sum_{i=1}^\infty }  e^{-\frac{-(i-1)^2 r^2 +2c_4 i r} {4t} },
\end{array}\]
which  converges uniformly on $[t_1, t_2] \times C \times C$. 
By the similar argument, we have uniform 
convergence for all the derivatives of the series. One can check  that 
 \[
\overline{K}_t (\bar x, \bar y)   = \sum_{g\in G}  K_t  (x, g\cdot y) g
 \] 
  is a fundamental solution to the
  heat equation for  $(\Dirac^{G\backslash \cE})^2$ on $G\backslash \f{X}$.  
  By the uniqueness of the  heat kernel, we know that 
the series 
\[
\sum_{g\in G}  K_t  (x, g\cdot y) g
\]
converges uniformly on $[t_1, t_2] \times C \times C$ to the heat kernel of $\Dirac^{G\backslash \cE}$.  
\end{proof}

As $K_t$ is $G$-invariant and $C$ has nonempty intersection with each orbit, the proof of Theorem~\ref{thm:HeatKernelAsympMain}, gives rise to an estimate on the heat kernel $K_t(x,y)$, which is important  for later calculations. 
\begin{corollary}
\label{cor:UniformUpperBdSumK(x,gx)}
There exists $L>0$ such that, for all $x\in\f{X},$ we have
\begin{equation}
\label{eq:UniformUpperBdSumK(x,gx)}
 \sum_{g\in G}   \| K_t  (x, g\cdot x) g\|_{\cE_x}\le L.
\end{equation}
\end{corollary}

\subsection{Cut-off functions and $(g)$-trace class operators}
\label{sec:cut-off(g)trace}
We aim to define some meaningful traces for the heat kernel operator $e^{-t(\Dirac^\cE)^2}$.
We shall first introduce a class of operators on $\f{X}$ ``approximating" the heat kernel operator.

\begin{definition}
\label{def:SchwartzClass}
Let $\mathbf{S}$ be the algebra of bounded operators on $L^2(\f{X}, \cE)$ having the following properties:
\begin{enumerate}
\item The Schwartz kernel $K_S$ of $S\in \mathbf{S}$ is smooth and $G$-invariant, in particular, 
\begin{equation*}
K_S(gx, gy)=g[K_S(x, y)]g^{-1}\qquad \forall x, y\in \f{X}, \forall g\in G.
\end{equation*}
\item The kernel $K_S$ for $S\in\mathbf{S}$ is properly supported, in the sense of~(\ref{eq:properSuppG-inv}).
\end{enumerate}
\end{definition}

Because of condition~(1) in Definition~\ref{def:SchwartzClass}, the usual trace of $S\in \mathbf{S}$ might be infinite. To eliminate the repeated summations caused by the $G$-invariance, we make use of the following ``weight" function over $\f{X}$.

Using Lemma~\ref{lem:quotient} on the local structure of the $G$-orbifold $\f{X}$,  let  $\cG=(\cG_1\rightrightarrows \cG_0)$ be   the corresponding proper \'etale grouppoid which locally looks like ({\em cf.} Remark~\ref{rem:OrbifoldGroupoidStructure})
\begin{equation*}
(G\times_{G_i}\tilde{U_i})\rtimes H_i\rightrightarrows G\times_{G_i}\tilde{U_i}
\end{equation*}
and naturally admits a $G$-action on the left.  Recall a smooth function $f$  on $\f{X}$ corresponds to a smooth
$\cG$-invariant function on $\cG_0$, that is, $s^* f = t^*f $.
  A smooth function $f$  on $\f{X}$ defines a unique continuous function of $|\f{X}|$ which will also be denoted by $f$.   A function on $\f{X}$ is called compactly supported if its support in $|\f{X}|$ is compact.

\begin{definition}[\bf Cut-off function]
\label{def:cut-offFunction}
A nonnegative function $c\in C_c^{\infty}(\f{X}) =  C_c^{\infty}(\cG)$ is called a {\bf cut-off function} of $\f{X}$ associated to the $G$ action if the values over the orbits of $G$-action add up to be $1$, i.e.,  
\begin{equation}
\label{eq:cut-offFunctionOnG_0}
 \sum_{g\in G} c(g^{-1}x)=1, \quad \forall x\in |\f{X}|.
\end{equation}
\end{definition}

Notice that the action of the cut-off function $c\in C^{\infty}_c(\f{X})$ on $f\in \Gamma(\f{X}, \cE)$ is given by point-wise multiplication.
\begin{equation}
\label{eq:Actioncut-offFucntion}
[c\cdot f](x)= c(x) f(x) .
\end{equation}

\begin{remark}
A cut-off function always exists for a proper co-compact $G$-action on $\f{X}.$
In fact, let $h\in C_c^{\infty}(\f{X})$ be a nonnegative function whose support has nonempty intersection with each $G$-orbit, then 
\begin{equation*}
c(x)=\frac{h(x)}{\sum_{g\in G} h(g^{-1}x)} 
\end{equation*}
is a cut-off function.
\end{remark}

We shall construct one particular cut-off function $c$ on $\f{X}$, where we suppose 
\begin{equation*}
| \f{X}| =\bigcup_{i=1}^N G\times_{G_i}U_i,
\end{equation*} 
and  each $U_i$ has an orbifold chart $(\tilde U_i, H_i, \pi_i, U_i)$. 

Let $\{\bar\phi_i\}$ be a partition of unity of $G\backslash \f{X}$ subordinate to the open
cover
\[
\{ V_i = G_i\backslash U_i = G_i \backslash  \tilde U_i /H_i\}_{i=1}^N
\]
such that the lift of $\bar\phi_i$ to $\tilde U_i$ is a smooth $G_i\times H_i$-invariant function, denoted by $\tilde \phi_i$.

We use $[x]$ to denote the equivalence class of an element $x$.
For example, if $g\in G$ and $u\in U_i$, then $[g, u] \in G\times_{G_i} U_i$ and $[u] \in  G_i\backslash U_i\subset |G\backslash \f{X}|$.
Let $\varphi_{i}$ be a function on $G\times_{G_i}U_i\subset\f{X}$ given by 
\begin{equation*}
\varphi_{i}([g,u]) =\begin{cases} \bar\phi_i ([u])  & g\in G_i \\ 0 & g\notin G_i.\end{cases}
\end{equation*}
As $\bar\phi_i$ is $G_i$-invariant, $\varphi_i$ is well-defined. Also, $\varphi_i$ extends to a smooth
$G_i$-invariant  function on $\f{X}$ when setting $\varphi_i(x)=0$ for $x\notin G\times_{G_i}U_i.$

Define a smooth nonnegative function on $\f{X}$ by
\begin{equation}
\label{eq:Particularcut-off}
c(x):=\sum_{i=1}^N\frac{1}{|G_i|}\varphi_i(x).
\end{equation} 
It is compactly supported because each summand is compactly supported.

\begin{lemma}
\label{lem:Specialcut-off}
The non-negative function $c\in C_c^{\infty}(\f{X})$ given by~(\ref{eq:Particularcut-off}) is a cut-off function of $\f{X}.$ 
\end{lemma}
\begin{proof}
Fix $x\in\f{X}$. For each $j\in J:=\{i|x\in G\times_{G_i}U_i\}$, there exists $h_j\in G$ and $u_j\in U_{j}$ so that $x=[h_j,u_j]$. Note also that if $j\notin J$, $\varphi_j(x)=0.$
Hence, by~(\ref{eq:Particularcut-off}) 
\begin{equation*}
\sum_{g\in G}c(g^{-1}x)=\sum_{g\in G}\sum_{i=1}^N\frac{1}{|G_i|}\varphi_i(x) c(g^{-1}[h, u])=\sum_{g\in G}\sum_{j\in J}\frac{1}{|G_j|}\varphi_j[g^{-1}h_j, u_j].
\end{equation*}
As $\varphi_{j}$ vanishes unless $g^{-1}h_j=k\in G_j$, then by the $G_j$-invariance of $\varphi_j$, we obtain
\begin{equation*}
\sum_{g\in G}c(g^{-1}x)=\sum_{k\in G_j}\sum_{j\in J}\frac{1}{|G_j|}\varphi_j([k, u_j])=\sum_{j\in J}\bar\phi_j([u_j]) 
\end{equation*}
Let $\pi: \f{X}\rightarrow G\backslash\f{X}$ be the quotient map, then by definition 
$\pi(x)=[u_j]$ for all $j\in J$ and $\bar\phi_j(\pi(x))=0$ if $j\notin J$.
As $\{\bar\phi_i\}$ is a partition of unity on $G\backslash \f{X}$ we have 
\begin{equation*}
\sum_{j\in J}\bar\phi_j([u_j])=\sum_{j=1}^N \bar\phi_j(\pi(x))=1.
\end{equation*}
The lemma is then proved.
\end{proof}

The cut-off function is designed to deal with $G$-invariant sections and $G$-invariant operators. 
Denote by $\Gamma(\f{X}, \cE)^G$ the subset of $G$-invariant sections in $\Gamma(\f{X}, \cE)$.
We shall present some properties of the cut-off function to be used later.

\begin{lemma}
\label{lem:Propertycut-offFunction}
Let $c\in C_c^{\infty}(\cG_0)$ be the cut-off function given by~(\ref{eq:Particularcut-off}). Then
\begin{enumerate}
\item For  a smooth function $f$ on $G\backslash\f{X}$ and its lift, the $G$-invariant function $\tilde f\in C^\infty(\f{X})$, we have 
\begin{equation*}
 \int_{\f{X}}^{orb}c(x)\tilde f(x)d\vol_{\f{X}}(x) = \int_{G\backslash\f{X}}^{orb}f(x)d\vol_{G\backslash\f{X}}(x).
\end{equation*}
\item If a continuous function $h$ on $\f{X}\times\f{X}$ satisfies that
\begin{equation}
\label{eq:Assumption}
h(gx, gy)=h(x,y), \qquad \forall g\in G, \forall x, y\in \f{X}
\end{equation} 
and that $c(x)h(x, y), c(y)h(x, y)$ are integrable on $\f{X}\times\f{X}$, we have 
\begin{equation*}
\int_{\f{X}\times\f{X}}^{orb}c(x)h(x, y)d\vol_{\f{X}}(x)d\vol_{\f{X}}(y)=\int_{\f{X}\times\f{X}}^{orb}c(y)h(x,y)d\vol_{\f{X}}(x)d\vol_{\f{X}}(y).
\end{equation*}
\end{enumerate} 
\end{lemma}

\begin{proof}
 Let $\{\bar\phi_i\}$ be a partition of unity of $G\backslash \f{X}$ subordinate to the open
cover
\[
\{ V_i = G_i\backslash U_i = G_i \backslash  \tilde U_i /H_i\}_{i=1}^N
\]
such that the lift of $\bar\phi_i$ to $\tilde U_i$ is a smooth $G_i\times H_i$-invariant function, denoted by $\tilde \phi_i$.
Let $\phi_i\in C^{\infty}(\f{X})$ be the $G$-invariant partition of unity of the cover $\{G\times_{G_i}U_i\}_{i=1}^N$ given by $\phi_i[g, u]=\bar\phi_i[u]$ for $g\in G$ and $u\in U_i$ locally.

(1) Using the cut-off function as in (\ref{eq:Particularcut-off}) , we have   
\[
\begin{array}{lll} &&\disp{ \int_{\f{X}}^{orb}}c(x)\tilde f(x)d\vol_{\f{X}}(x)\\[3mm]
&=& \disp{\int_{\f{X}}^{orb} \sum_{i=1}^N} \frac{1}{|G_i|}\varphi_i(x)\tilde f(x)d\vol_{\f{X}}(x)  \\[3mm]
&=& \disp{\sum_{i=1}^N\int_{ U_i}^{orb}} \frac{1}{|G_i|}\varphi_i(x)\tilde f(x)d\vol_{\f{X}}(x)  
\\[3mm]
&=&\disp{\sum_{i=1}^N\frac{1}{|H_i\times G_i|}\int_{ \tilde U_i}} \tilde \phi_i(x) \tilde f(x)d x   \\[2mm]
&=& \disp{ \int_{G\backslash \f{X}}^{orb}}    f(x)d\vol_{G\backslash\f{X}}(x).
\end{array}
\]
 
(2) Using~(\ref{eq:Particularcut-off})  for the cut-off function and the partition of unity $\{\phi_j\}$ for $\f{X}$
subordinated to $\{G\times _{G_j}  U_j\}_{j=1}^N$, we get
\begin{equation}
\label{eq:in Pf1}\begin{array}{lll} 
&&\disp{ \int_{\f{X}\times\f{X}}^{orb}} c(x)h(x, y)d\vol_{\f{X}}(x)d\vol_{\f{X}}(y)\\[3mm]
&=&\disp{  \sum_{i=1}^N \int^{orb}_{ U_i} \int_{\f{X}}^{orb} }  \frac{1}{|G_i|}\varphi_i(x)
 h(x, y ) d\vol_{\f{X}}(x)d\vol_{\f{X}}(y)  \\[3mm]
&=&\disp{  \sum_{i=1}^N \frac{1}{|G_i \times H_i| }\int_{\tilde U_i} \int_{\f{X}}^{orb} }
\tilde \phi_i(x)\ h(x, y )dx d\vol_{\f{X}}(y) \\[3mm]
&=&\disp{ \sum_{i,j=1}^N \frac{1}{|G_i| |H_i| }\int_{\tilde U_i}\int^{orb}_{ G\times _{G_j}  U_j}}
\tilde \phi_i(x) \phi_j(y)h(x,  y)dx d\vol_{\f{X}}(y)
 \\[3mm]
&=&\disp{ \sum_{k\in G}\sum_{i, j=1}^N\frac{1}{|G_i||G_j||H_i||H_j|}\int_{\tilde U_i\times \tilde U_j}}
\tilde \phi_i(x)\tilde\phi_j(y)h(x, k\cdot y)dx dy.
\end{array}
\end{equation}
Here the last equality follows from
\[ 
 \int^{orb}_{ G\times _{G_j}  U_j} 
  \phi_j(y)h(x,  y) d\vol_{\f{X}}(y) = \sum_{k\in G} \frac{1}{|G_j|} \int^{orb}_{U_j} \tilde\phi_j(y)h(x, k\cdot y)  dy.
\]
Similarly,  we have 
\begin{equation}
\label{eq:in Pf2} \begin{array}{lll} 
&&\disp{ \int_{\f{X}\times\f{X}}^{orb}}c(y)h(x,y)d\vol_{\f{X}}(x)d\vol_{\f{X}}(y) \\
 &= &\disp{\sum_{k\in G}\sum_{i, j=1}^N\frac{1}{|G_i||G_j||H_i||H_j|}\int_{\tilde U_i\times \tilde U_j} }\tilde \phi_i(x)\tilde\phi_j(y)h(k\cdot x, y) dx dy.\end{array}
\end{equation}
By~(\ref{eq:Assumption}), we have $h(x,  k\cdot y)=h( k^{-1}\cdot x, y)$ for $k\in G$. Thus, the $k$-term in~(\ref{eq:in Pf1}) is the $k^{-1}$-term in~(\ref{eq:in Pf2}). Then the integrals~(\ref{eq:in Pf1}) and~(\ref{eq:in Pf2}) coincide.
The lemma is then proved.
\end{proof}

\begin{remark}
Condition (2) in Definition~\ref{def:SchwartzClass} is to reduce the complexity in calculating the new traces. However, $e^{-t(\Dirac^\cE)^2}$ does not have proper support. In view of Theorem~\ref{thm:HeatKernelAsympMain}, when $t>0$ is small,  the heat kernel decays exponentially off the diagonal. Therefore, we shall use the following lemma to replace the heat kernel by an operator in the class $\mathbf{S}$ having the same value on the diagonal.
\end{remark}

\begin{lemma}
\label{lem:SmoothingToProperlySupportedSmoothing}
For any $G$-invariant smoothing operator $S$,
the operator 
\begin{equation}
\label{eq:SmoothingToProperlySupportedSmoothing}
S_0:=\sum_{g\in G}g\cdot(c^{\frac12}Sc^{\frac12})
\end{equation} 
belongs to $\mathbf{S}$ ({\em cf.} Definition~\ref{def:SchwartzClass})  and has the same Schwartz kernel as that of $S$ along the diagonal.
\end{lemma}

\begin{proof}
Let $K_S$ and $K_{S_0}$ be the heat kernels for the operators $S$ and $S_0$ respectively. Then by definition~(\ref{eq:SmoothingToProperlySupportedSmoothing}) we have
\begin{equation}
\label{eq:KernelS_0andS}
K_{S_0}(x,y)=\sum_{g\in G}c(g^{-1}x)^{\frac12}K_S(x, y)c(g^{-1}y)^{\frac12}
\end{equation}
and conclude that $S_0$ is a properly supported smoothing operator.
Moreover, if $x=y$, then 
\begin{equation*}
K_{S_0}(x,x)=\sum_{g\in G}c(g^{-1}x)K_S(x,x)=K_S(x,x).
\end{equation*}
The lemma is then proved.
\end{proof}

In the following, we shall define a family of traces, indexed by the conjugacy classes of elements in $G$, for the set of operators slightly larger than the set $\mathbf{S}$ defined in Definition~\ref{def:SchwartzClass}. 
These traces are closely related to the localized indices in Section~\ref{sec:Localized indices}. The expression of these traces follows from an explicit computation in Lemma~\ref{lem:VeryImportantLemma}.

For any $g\in G$, denote by $(g)_G\subset G$ (or $(g)$ when there is no ambiguity) the conjugacy class of $g\in G.$ Also, \emph{whenever} traces of an operator involved in this paper, if the space is $\Z/2\Z$-graded, we shall mean a  graded traces or a supertrace:
\begin{equation*}
 \Tr_s\begin{pmatrix}A & B \\ C & D \end{pmatrix}:=\Tr(A)-\Tr(B).
\end{equation*}

\begin{definition}[\bf $(g)$-trace class operator]
\label{def:(g)trace}
A bounded properly supported $G$-invariant operator $S: L^2(\f{X}, \cE)\rightarrow L^2(\f{X},\cE)$ is of {\bf $(g)$-trace class} if the positive operator 
\[
\sum_{h\in(g)}h^{-1}\phi|S|\psi
\]
 is of trace class for all positive functions $\phi, \psi \in C^{\infty}_c(\cG_0)$. Here $|S|=(SS^*)^{\frac12}.$
 Here $h^{-1}$ stands for the unitary operator on $L^2(\f{X}, \cE)$ given by 
 \[
 (h^{-1}\cdot u)(x):=h^{-1} u(hx)
 \]
  for $u\in L^2(\f{X}, \cE).$  If $S$ is a $(g)$-trace class operator,   the $(g)$-trace is defined  by the formula 
\begin{equation}
\label{def(g)trace}
\tr_s^{(g)}(S):=\sum_{h\in(g)}\tr_s[h^{-1}cS],
\end{equation} where $c$ is a cut-off function on $\f{X}$. 
\end{definition}

\begin{remark}
In case of $G$-manifold, where $g$ is the group identity,  the $(e)$-trace is the same notion as the $G$-trace of bounded $G$-invariant operators on $L^2(\f{X}, \cE)$ introduced in~\cite{W:2012}.  

\end{remark}

Denote by $K$ the subset of $G$ having the following properties: 
\begin{equation}\label{ProofTrace}
\begin{split}
&\{kgk^{-1}| k\in K\}=(g);\\
& \text{If } k_1\neq k_2, \text{ for all } k_j\in K, \text{ then } k_1gk_1^{-1}\neq k_2gk_2^{-1}.
\end{split} 
\end{equation}
We shall sometime to use $\{kgk^{-1}\}_{k\in K}$ to denote the conjugacy class $(g)$ of $g\in G.$
In view of the following lemma, we shall identify $K$ with the quotient $G/Z_G(g)$ where  $Z_{G}(g)=\{h\in G| hg=gh\}$ is the centralizer of $g$ in $G.$
Note that, if finite, $K$ has the same cardinality as that of $(g).$

\begin{lemma}
\label{lem:K^GZ_G=G}
Let $K$ be a subset of $G$ having the property~(\ref{ProofTrace}). Then
\begin{equation*}
K\cdot Z_{G}(g)=G.
\end{equation*}
\end{lemma}

\begin{proof}
If $k_1, k_2\in K$ and $h_1, h_2\in Z_{G}(g)$ satisfy $k_1h_1=k_2h_2,$ then $$k_1h_1gh_1^{-1}k_1^{-1}=k_2h_2gh_2^{-1}k_2^{-1}$$ and it is the same as $k_1gk_1^{-1}=k_2gk_2^{-1}$. By definition of $K$, we have $k_1=k_2$, hence $h_1=h_2.$ 
Then the map
\begin{equation}
\label{eq:KZ=G}
m: K\times Z_{G}(g)\longrightarrow G \qquad (k, h)\mapsto kh
\end{equation} 
is injective.
For each $l\in G$, then by definition there is an $k\in K$ so that $lgl^{-1}=kgk^{-1}$. Thus, $k^{-1}l\in Z_G(g)$ and~(\ref{eq:KZ=G}) is surjective. 
The lemma is then proved.
\end{proof}

\begin{proposition}
\label{prop:FormulaFor(g)Trace}
Let $S$ be a $(g)$-trace class operator having smoothing Schwartz kernel and let $c$ be a cut-off function on $\f{X}$. Then 
\begin{equation}
\label{eq:tr(g)FormulaInProposition}
\tr_s^{(g)}S=\sum_{h\in(g)}\int_{\f{X}}^{orb}c(x)\Tr_s [h^{-1}K_S(hx, x)] d\vol_{\f{X}}(x).
\end{equation}
Here, $\Tr_s$ is the matrix supertrace of $\End(\cE_x, \cE_x)$.
Alternatively, 
 \begin{equation}\label{eq:(g)TraceKernel}
 \tr_s^{(g)}S= \sum_{k\in G/Z_G(g)}\disp{\int}_{\f{X}}^{orb}c(kx)\Tr_s [g^{-1}K_S(gx, x)]d\vol_{\f{X}}(x),
\end{equation}
where $G/Z_G(g)$ is identified as a subset $K$ of $G$ having property~(\ref{ProofTrace}).
In particular, operators from $\mathbf{S}$ given by Definition~\ref{def:SchwartzClass} are of $(g)$-trace classes.
\end{proposition}

\begin{proof}
As $(h^{-1}\cdot u)(x)=h^{-1}u(hx)$ for $u\in L^2(\f{X}, \cE)$, we have 
\begin{align*}
[h^{-1}cS]u(x) =&c(hx)h^{-1}Su(hx)=\int_{\f{X}}^{orb} c(hx)h^{-1}[K_S(hx,y)]u(y)d\vol_{\f{X}}(y).
\end{align*}
Then by the change of variable $x\to h^{-1}x$ and the invariance of the measure $d\vol_{\f{X}}(h^{-1}x)=d\vol_{\f{X}}(x)$, 
we obtain
\begin{align*}
\tr_s[h^{-1}c S]
=\int_{\f{X}}^{orb} c(x)\Tr_s [h^{-1}K_S(hx, x)]d\vol_{\f{X}}(x)
\end{align*}
If $h=kgk^{-1}$, then as $S$ is $G$-invariant, 
by~(\ref{eq:GinvKernel}) we have 
$h^{-1}[K_S(hx', x')]=g^{-1}[K_S(gx, x)]$ where $x'=kx.$
Therefore, 
\begin{align*}
\tr_s^{(g)}S&=\sum_{h\in(g)}\tr_s[h^{-1}cS]\\
&=\sum_{h\in(g)}\int_{\f{X}}^{orb}c(x)\Tr_s [h^{-1}K_S(hx, x)] d\vol_{\f{X}}(x)\\
&=\sum_{k\in G/Z_G(g)}\int^{orb}_{\f{X}}c(kx)\Tr_s [g^{-1}K_S(gx, x)]d\vol_{\f{X}}(x).
\end{align*}

If $S\in\mathbf{S}$, then as the action of $G$ on $\f{X}$ is proper, there are finitely many nonvanishing terms in the sum in~(\ref{eq:tr(g)FormulaInProposition}).
Hence operators from $\mathbf{S}$ are of $(g)$-trace classes.
The proposition is proved.
\end{proof}

\begin{remark}
\label{rem:IndependentOfcut-offFunction}
The trace $\tr_s^{(g)}$ given by Definition~\ref{def:(g)trace} does not depend on the choice of cut-off function $c$.
In fact, set
$$m(x):=\sum_{h\in(g)}\Tr_s [h^{-1}K_S(hx, x)]$$ and observe that for any $k\in G$, we have
\[
\begin{array}{lll}
m(kx) &= & \sum_{h\in(g)}\Tr_s h^{-1}[K_S(hkx, kx)]\\[2mm]
&=& \sum_{h\in(g)}\Tr_s[(k^{-1}hk)^{-1}K_S(k^{-1}hkx, x)]=m(x).
\end{array}
\]
Then by Lemma~\ref{lem:Propertycut-offFunction} (1) the integral over the orbifold $\f{X}$ does not depend on choice of a cut-off function $c.$
\end{remark}

The $(g)$-trace defined in~(\ref{def(g)trace}) has the tracial property as is shown in the following proposition.
As we shall see later in the proof of Proposition~\ref{prop:LocalizedIndex=SupertraceSmoothing}, this trace is the composition of a matrix trace and the localized $(g)$-trace given by Definition~\ref{def:Localized(g)Trace}.

\begin{proposition}\label{tracia:property}
Let $S, T$ be $G$-invariant operators and let $ST$ and $TS$ be of $(g)$-trace class, then 
\begin{equation*}
\tr_s^{(g)}(ST)=\tr_s^{(g)}(TS).
\end{equation*}
\end{proposition}

\begin{proof}
By Proposition~\ref{prop:FormulaFor(g)Trace}, we have 
\begin{equation}
\label{eq:1}
\tr_s^{(g)}(ST)=\sum_{h\in(g)}\int_{\f{X}\times\f{X}}^{orb}c(x)\Tr_s[h^{-1}(K_S(hx,y)K_T(y, x))]d\vol_{\f{X}}(x)d\vol_{\f{X}}(y).
\end{equation}
For the similar reason we arrive at the $(g)$-trace for $TS:$ 
\begin{equation}
\label{eq:2}\begin{array}{lll}
\tr_s^{(g)}(TS) &= &\disp{ \sum_{h\in(g)}\int_{\f{X}\times\f{X}}^{orb}} c(y)\Tr_s[h^{-1}(K_T(hy,x)K_S(x, y))]d x d y \\[2mm]
&=& \disp{\sum_{h\in(g)}\int_{\f{X}\times\f{X}}^{orb}} c(y)\Tr_s[h^{-1}(K_T(y,x)K_S(x, h^{-1}y))]d x d y,
\end{array}
\end{equation}
where the latter equality comes from the change of variables $y\mapsto h^{-1} y$.  
Set 
\[
m(x, y) =\sum_{h\in(g)}\Tr_s[h^{-1}(K_S(hx,y)K_T(y, x))].
\]
Then the $G$-invariance property of $S$ and $T$ implies that $m$ satisfies~(\ref{eq:Assumption}):
\[\begin{array}{lll}
m(kx, ky) &=&\disp{\sum_{h\in(g)} } \Tr_s[h^{-1}K_S(hkx, ky)K_T(ky, kx)]\\[2mm]
& =&\disp{\sum_{h\in(g)} } \Tr_s[(k^{-1}hk)^{-1}K_S(k^{-1}hkx, y)K_T(y, x)]\\[2mm]
&=& m(x,y), \qquad \forall k\in G.
\end{array} 
\]
Notice that 
\[\begin{array}{lll}
 \Tr_s[h^{-1}(K_S(hx,y)K_T(y, x))]  
&=&  \Tr_s[K_T(y, x) h^{-1}K_S(hx,y)] \\[2mm]
&=& \Tr_s[h^{-1}K_T(y, x)  K_S( x,h^{-1}y)] .  
\end{array}
\]
 Then we have 
 \[
 \begin{array}{lll}
  \tr_s^{(g)}(ST) &=&\disp{ \int_{\f{X}\times\f{X}}^{orb}}  c(x) m(x, y) dx dy \\[3mm]
  &=&  \disp{ \int_{\f{X}\times\f{X}}^{orb}}  c(y) m(x, y) dx dy = \tr_s^{(g)}(TS) , 
  \end{array}
  \]
 by  Lemma~\ref{lem:Propertycut-offFunction}~(2). \end{proof}

Finally we look at the $G$-invariant heat operator $S_t:=e^{-t(\Dirac^\cE)^2}$ for $t>0$ where $\Dirac^\cE$ is the $G$-invariant Dirac operator on $\f{X}$ given by~(\ref{eq:FormallySelfAdjointD}).  
Denote by $K_t(x, y)$ the Schwartz kernel of $e^{-t\Dirac^2}$, which can be ``approximated" by elements of $\mathbf{S}$, i.e., the norm of $K_t(x, y)$ decays rapidly when $d(x,y)\to\infty$ for a fixed positive number $t$ ({\em cf.} Theorem~\ref{thm:HeatKernelAsympMain}). 

Following from Theorem~\ref{thm:HeatKernelAsympMain}, when $g\in G$ does not have a fixed point on $\f{X}$, we have 
\begin{equation}
\label{eq:VanishingHeatKernel}
\lim_{t\to0^+}g^{-1}K_t(gx, x)=0.
\end{equation}
If $g\in G$ has a fixed point $(h,x)\in\f{X}$ where $h\in G$ and $x\in U_i,$ then 
\[
g(h,x)=(gh, x)=(h(h^{-1}gh), x) 
\]
 implies that 
$
h^{-1}gh\in G_i$ and fixes  $x\in U_i.
$
Then the following lemma implies that we may assume $g\in G_i$ and $g$ fixes the point $(e, x).$

\begin{lemma}
\label{lem:(g)capG_iNonempty}
If $g\in G$ has a fixed point on $\f{X}$, then there exists an $i\in\{1, \ldots, N\}$, so that $(g)\cap G_i$ is not empty.
\end{lemma}

\begin{proposition}
\label{prop:HeatOp(g)TraceClass}
Let $S_t$ be the heat operator $e^{-t(\Dirac^\cE)^2}$ and $K_t$ be its Schwartz kernel. 
Then it is of $(g)$-trace class and the limit of $(g)$-trace is given by the finite sum 
\begin{equation}
\label{eq:(g)traceIntegralHeatKernel}
 \tr_s^{(g)} S_{t}=\sum_{h\in(g)\cap(\cup_{i=1}^N G_i)}\lim_{t\to0^+}\int_{\f{X}}^{orb}c(x)\Tr_s [h^{-1}K_t(hx, x)] d\vol_{\f{X}}(x).
\end{equation}
\end{proposition}

\begin{proof}
Following from Definition~\ref{def:(g)trace}, let us estimate the sum $\sum_{h\in(g)}|\Tr_s[h^{-1}\phi S_t\psi]|$ where $\phi, \psi\in C^{\infty}_c(\f{X})$ as follows:
\begin{equation}
\label{eq:AbsoluteConvergent}
\begin{split}
\sum_{h\in(g)}|\Tr_s[h^{-1}\phi S_t\psi ]|\le &\sum_{h\in(g)}\int_{\f{X}}^{orb}\phi(hx)\psi(x)\dim(\cE)\|h^{-1}K_t(hx,x)\|_{\cE_x}d\vol_{\f{X}}(x)\\
\le &\dim\cE\|\phi\|_{L^\infty}\int_{\f{X}}^{orb} \psi(x)\sum_{h\in G}\|h^{-1}K_t(hx, x)\|_{\cE_x}d\vol_{\f{X}}(x) \\
\le & L\dim\cE\|\phi\|_{L^\infty} \|\psi\|_{L^1}  <\infty.
\end{split}
\end{equation}
Here, in the last inequality, we have used the uniform upper estimate in Corollary~\ref{cor:UniformUpperBdSumK(x,gx)}.
Therefore, the heat operator $S_t=e^{-t(\Dirac^\cE)^2}$  on $\f{X}$ is of $(g)$-trace class for all $g\in G$ and $t>0$.

Let $a(t) =  \tr_s^{(g)} S_{t}$, then 
\[\begin{array}{lll}
\dfrac{d a(t)}{dt} &  = & -   \tr_s^{(g)}  ((\Dirac^\cE)^2 e^{-t(\Dirac^\cE)^2})\\[2mm]
&=&  - \frac 12  \tr_s^{(g)}  ([ \Dirac^\cE ,  \Dirac^\cE e^{-t(\Dirac^\cE)^2}])
\end{array}
\]
which vanishes by  Proposition \ref{tracia:property}. So the function $a(t) =  \tr_s^{(g)} S_{t}$ is constant in $t$. 
Then, by Proposition~\ref{prop:FormulaFor(g)Trace}, we 
have 
\begin{equation}
\label{eq:InProof}
 \tr_s^{(g)} S_{t}=\lim_{t\to0^+}\sum_{h\in(g)}\int_{\f{X}}^{orb}c(x)\Tr_s [h^{-1}K_t(hx, x)] d\vol_{\f{X}}(x).
\end{equation}
Using the same argument as~(\ref{eq:AbsoluteConvergent}), the sum in~(\ref{eq:InProof}) is absolute convergent.
Hence, when we take the limit of the sum, the limit commute with the infinite sum as well as the integral $\int_{\f{X}}^{orb}$.
Further, the limit commutes with the integral.
As $t\to0^+$ then from~(\ref{eq:VanishingHeatKernel}) the summand in~(\ref{eq:(g)traceIntegralHeatKernel}) tends to $0$ if $h$ does not fix any point in $\f{X}.$
The nonzero ones are the ones where $hx=x$ for some $x\in\f{X}$.
By Lemma~\ref{lem:(g)capG_iNonempty} we see it happens only when $h\in(g)\cap G_i$ for some $i$. 
Thus, there are only finitely many summand and
\begin{equation*}
\tr_s^{(g)} S_{t}
=\sum_{h\in(g)\cap (\cup_{i=1}^N G_i)}\lim_{t\to0^+}\int_{\f{X}}^{orb}c(x)\Tr_s [h^{-1}K_{S_{t}}(hx, x)] d\vol_{\f{X}}(x).
\end{equation*}
The proposition is then proved.
\end{proof}

\begin{corollary}
Let $S_{0,t}:=\sum_{g\in G}g\cdot(c^{\frac12}S_tc^{\frac12})$ where $S_t=e^{-t(\Dirac^\cE)^2}$. 
Then 
\begin{equation*}
  \tr_s^{(g)}(S_t) = \lim_{t\to0^+}   \tr_s^{(g)}(S_{0,t}) .
\end{equation*}
\end{corollary}

\begin{proof}
 From Proposition~\ref{prop:FormulaFor(g)Trace} and Lemma~\ref{lem:SmoothingToProperlySupportedSmoothing}, we see that $S_{0,t}$ is of $(g)$-trace class for all $g\in G$. 
As the sum $\sum_{h\in(g)}K_{S_{0,t}}(hx, x)$ has finitely many non-vanishing terms,  
we have
\begin{equation}
\label{eq:(g)traceIntegralHeatKer}
 \lim_{t\to0^+}  \tr_s^{(g)} S_{0,t}=  \sum_{h\in(g)}\int_{\f{X}}^{orb}c(x)\Tr_s [\lim_{t\to0^+}h^{-1}K_{S_{0,t}}(hx, x)] d\vol_{\f{X}}(x).
\end{equation}
Observe that $S_{0,t}$ and $S_t$ have the same diagonal value as $S_t=e^{-t(\Dirac^\cE)^2}$ by Lemma~\ref{lem:SmoothingToProperlySupportedSmoothing}.
In view of Proposition~\ref{prop:HeatOp(g)TraceClass}, we have the same formula for $\tr_s^{(g)} S_t$ as~(\ref{eq:(g)traceIntegralHeatKer}).
Hence, the corollary is then proved. 
\end{proof}

\subsection{Calculation of $(g)$-trace for the heat operator}
\label{(g)-traceHeatOp}
We shall derive the formula for the $(g)$-trace of the heat operator $e^{-t(\Dirac^\cE)^2}$ in terms of the local data on the quotient orbifold $G\backslash\f{X}$.

To calculate the  $(g)$-trace  of the heat operator $S_t=e^{-t(\Dirac^\cE)^2}$, we need to describe the twisted sectors for the quotient
orbfiold $G\backslash \f{X}$.  Let $\cT_{\f{X}}$ be the  set for the twisted sectors of $\f{X}$ consisting of
equivalence classes of conjugacy classes in the local isotropy groups of $\f{X}$.  Then 
the  set for the twisted sectors of $G\backslash \f{X}$ consists of pairs
\[
((g), (h))
\]
where $(g)$ is the conjugacy class of $g$ in $G$ such that $g$ has non-empty fixed points, and
$  (h)  \in  \cT_{\f{X}}$.   We now  describe the twisted sector $(G\backslash \f{X})_{((g), (h))}$. 
  Let  
  \[
  \{( G\times_{G_i} \tilde U_i, H_i, G\times_{G_i} \tilde U_i)\}
\]
 be
   orbifold charts as in Lemma \ref{lem:quotient} such that 
   $ G\times_{G_i} U_i$ consists of a disjoint union of $G$-translate of $U_i$. Suppose  that 
   $g\in G$ has non-empty fixed points in $U_i$, then 
   \[
(g) \cap G_i = (g_i)_{G_i}
\]
for a conjugacy class of $g_i$ in $G_i$. Let $(h) \in \cT_{\f{X}}$ 
have  a representative $(h_i)$,  a conjugacy class in $H_i$.   
  Then 
\ba\label{orbi:chart/G}
\{ (\tilde U_{i}^{g_i, h_i}, Z_{G_i}(g_i) \times Z_{H_i}(h_i))  \}
\na
are  orbifold charts for the twisted sector  $(G\backslash \f{X})_{((g), (h))}$.  

\begin{definition} The {\bf $(g)$-twisted sector} of $G\backslash \f{X}$ is defined to be 
\ba\label{(g)-sector}
(G\backslash \f{X})_{(g)} = \bigsqcup_{(h) \in  \cT_{\f{X}}} (G\backslash \f{X})_{((g), (h))} 
\na
as a sub-orbifold of the inertia orbifold  $I(G\backslash\f{X})$.  Let
\begin{equation}
\label{eq:RestrictionDelocChar}
\hat{A}_{(g)} (\f{X})  \ \text{and} \  \ch^{\cS}_{(g)} (G\backslash \cE) 
\end{equation}
be the restrictions of $\hat{A}_{deloc}(G\backslash \f{X}) $ and $\ch^{\cS}_{deloc} (G\backslash\cE)$ to the 
$(g)$-twisted sector $(G\backslash \f{X})_{(g)} $. 
\end{definition}

\begin{theorem}
\label{thm:HeatKernelAsymptotics}
For all $g\in G,$ the heat operator $S_t=e^{-t(\Dirac^\cE)^2}$ is a $(g)$-trace class operator, and the $(g)$-trace is  given  by
\begin{equation}
\label{eq:LocalizedIndexStrata}
\begin{split}
\tr_s^{(g)} S_{t}=&\lim_{t\to0^+}\sum_{h\in (g)}\disp{\int}_{\f{X}}^{orb}c(x)\Tr_s[h^{-1}K_t(hx, x)]d\vol_{\f{X}}(x) \\
=&\int^{orb}_{(G\backslash \f{X})_{(g)}}\hat{A}_{(g)} (\f{X})   \ch^{\cS}_{(g)} (G\backslash\cE).
\end{split}
\end{equation}
\end{theorem}

\begin{proof}
If $g\in G$ does not fix any point in $\f{X}$, then by~(\ref{eq:VanishingHeatKernel}), we have $\tr_s^{(g)} S_{t} =0$. On the other hand, $(G\backslash \f{X})_{(g)}$ is an empty set. 
So the integral on the right hand side also vanishes.

If $g\in G$ has a fixed point in $\f{X}$, then by Lemma~\ref{lem:(g)capG_iNonempty} we may assume $g\in G_i$ for some $i.$ 
Let $c$ be the cut-off function of $\f{X}$ given by~(\ref{eq:Particularcut-off}) (See also~Lemma~\ref{lem:Specialcut-off}). 
Let $K_t$ (resp. $\bar K_t$) be the heat kernel of $\Dirac^\cE$ (resp. $\Dirac^{G\backslash \cE}$) and let $\tilde K_t$ be the lift of $\bar K_t$ to $\tilde U_i\times \tilde U_i.$
Then we have 
\begin{equation}
\label{eq:pfSupertrace}
\begin{split}
\tr_s^{(g)} S_{t}=&\sum_{k\in(g)}\disp{\int}_{\f{X}}^{orb}c(x)\Tr_s[k^{-1}K_t(kx, x)]d\vol_{\f{X}}(x)\\
=&\sum_{k\in (g)}\sum_{i=1}^N\int^{orb}_{G_i\times_{G_i}U_i} 
\frac{1}{|G_i|}\varphi_i(x)\Tr_s[k^{-1}K_t(kx, x)]d\vol_{\f{X}}(x).
\end{split}
\end{equation}
By Corollary~\ref{cor:UniformUpperBdSumK(x,gx)}, the sum is absolute convergent.
Hence, when we take the limit of the sum, the limit commutes with the infinite sum.
Then, notice that by~(\ref{eq:VanishingHeatKernel}),  as $t\to0^+$, $\int_{G_i\times_{G_i}U_i}\varphi(x)\Tr_s[k^{-1}K_t(kx, x)]\to 0$ unless there exists a point $x=[m,u]\in G_i\times_{G_i}U_i$ fixed by $k\in G$, that is, 
\[
[m,u]=k[m,u]=[m(m^{-1}km),u],
\]
 which means that $k':=m^{-1}km\in G_i.$ 
In this situation, we have 
\begin{equation*}
k^{-1}K_t(kx, x)=K_t([m,u], [m, m^{-1}kmu])k=\bar K_t([u], k'[u])k'.
\end{equation*}
To simplify the notation, let us still use $k$ for $k'\in(g)_{G_i}.$
Hence, as we take the limit of each summand of~(\ref{eq:pfSupertrace}), it vanishes excepted for the terms $k\in(g)_{G_i}$.
Thus, $\lim_{t\to0^+}\tr_s^{(g)} S_{t}$ is equal to
\begin{equation*}
\sum_{i=1}^N\frac{1}{|G_i|}\frac{1}{|H_i|}\lim_{t\to0^+}\int_{\tilde U_i} \tilde\phi_i(u)\sum_{k\in (g)_{G_i}}\sum_{(h)\in\cT_{\f{X}}\cap H_i}\sum_{l\in(h)}\Tr_s[\tilde K_t(u, kul)kl]du.
\end{equation*}
As $\tilde K_t$ is $G_i$ and $H_i$ invariant, $\tilde K_t$ remains constant on the conjugacy classes in $G_i$ and in $H_i$.  Also, by Lemma~\ref{lem:K^GZ_G=G} we have $\frac{|(g)_{G_i}|}{|G_i|}=\frac{1}{|Z_{G_i}(g)|}$ and $\frac{|(h)_{H_i}|}{|H_i|}=\frac{1}{|Z_{H_i}(h)|}$.
Therefore, we conclude that 
\begin{equation*}
\label{eq:4}
\lim_{t\to0^+}\tr_s^{(g)} S_{t}
=\sum_{i=1}^N\frac{1}{|Z_{G_i}(g)|}\lim_{t\to0^+}\int_{\tilde U_i} \tilde\phi_i(u)\sum_{(h)\in\cT_{\f{X}}\cap H_i}\frac{1}{|Z_{H_i}(h)|}\Tr_s[\tilde K_t(u, guh)gh]du.
\end{equation*}

Applying the standard local index techniques as in \cite{BGV, Bis, LYZ, Ma, PW:NCGCGI.PartII}, we get
\begin{equation*} \begin{array}{lll}&& 
 \disp\sum_{(h)\in\cT_{\f{X}}\cap H_i} \dfrac{1}{|Z_{G_i}(g)|} \dfrac{1}{|Z_{H_i}(h)|}\disp\lim_{t\to0^+}\disp{\int_{\tilde U_i}}  \tilde\phi_i(u) \Tr_s[K_t(u, guh)gh]du\\[3mm] 
& =& \disp{\sum_{(h)\in \cT_{\f{X}} \cap H_i} }  \dfrac{1}{|Z_{G_i}(g)|} \dfrac{1}{|Z_{H_i}(h) |}   \disp{\int_{\tilde U_i^{g, h}}} \tilde\phi_i(u) \hat{A}_{deloc} (\f{X})   \ch^{\cS}_{deloc} (G\backslash\cE) \\[3mm] 
& =& \disp{\int^{orb}_ {(G\backslash\f{X} )_{(g)}}}\bar\phi_i   \hat{A}_{(g)} (\f{X})   \ch^{\cS}_{(g)} (G\backslash\cE),
\end{array}
\end{equation*} 
for $g\in G_i$. 
Therefore,   
\begin{equation*}\begin{array}{lll}
\tr_s^{(g)} S_{t} &=& \disp\lim_{t\to0^+}\tr_s^{(g)} S_{t} \\[2mm]
&= & \disp{\sum_{i=1}^N  \int^{orb}_{(G\backslash \f{X})_{(g)}}} 
\bar\phi_i   \hat{A}_{(g)} (\f{X})   \ch^{\cS}_{(g)} (G\backslash\cE)\\[3mm]
&=& \disp{   \int^{orb}_{(G\backslash \f{X})_{(g)}}} 
    \hat{A}_{(g)} (\f{X})   \ch^{\cS}_{(g)} (G\backslash\cE).\end{array}
\end{equation*}
This completes the proof of the 
 theorem. 
\end{proof}

\begin{remark}
Combining  with the Kawasaki' s orbifold  index theorem,  Theorem~\ref{thm:HeatKernelAsymptotics} implies that
\ba\label{combined:index}
\begin{array}{lll}
\quad\disp{  \sum_{(g) }} \tr_s^{(g)}S_t  &= &    \disp{\sum_{(g) }     \int^{orb}_{(G\backslash \f{X})_{(g)}} }
    \hat{A}_{(g)} (\f{X})   \ch^{\cS}_{(g)} (G\backslash\cE)  \\[2mm]
&=& \disp{\int^{orb}_{I(G\backslash \f{X})} }\hat{A}_{deloc} (G\backslash \f{X})   \ch^{\cS}_{deloc} (G\backslash\cE) 
 = \ind (\Dirac^{G\backslash\cE}).
 \end{array}
 \na
 In Sections \ref{sec:K-theoretic index} and  \ref{sec:Localized indices}, we shall give a $K$-theoretic interpretation of $ \tr_s^{(g)}e^{-t(\Dirac^\cE)^2}$, which is called the localized index of $\Dirac^\cE$ at the conjugacy class $(g)$ of $g$ in $G$. 
\end{remark}

\section{Higher Index Theory for Orbifolds}
\label{sec:K-theoretic index}

In this section,  we propose a   higher  index of a $G$-invaraint Dirac operator 
\[
\Dirac^\cE: L^2(\f{X}, \cE)\rightarrow L^2(\f{X}, \cE)
\]
on  
a complete Riemannian  even dimensional orbifold with a proper, co-compact  and isometric action of a discrete  group $G$. This higher index generalizes the notion of higher index for smooth noncompact  manifolds with
a proper and free co-compact action of a discrete group $G$ as in
\cite{Kasparov:1983} and  \cite{Kasparov:1988dw}.  This  higher index  for smooth noncompact  manifolds 
plays an important role in the  study of the Novikov conjecture and the Baum-Connes conjecture. See \cite{BC}, \cite{BCH} and \cite{Yu}.  

When $X $ is a proper co-compact Riemannian  $G$-manifold with a $G$-equivaraint Clifford module $\cE$,  the Dirac operator
 $\Dirac^\cE_X$ 
gives rise to a $K$-homology class 
\ba\label{K-homology:class}
[(L^2(X, \cE), F=\Dirac_X^\cE\left[(\Dirac_X^\cE)^2+1\right]^{-\frac12}]\in K^0_G(C_0(X)).
\na
The higher index  of $\Dirac^\cE_X$ is defined to be the image of  the homology class (\ref{K-homology:class}) 
under the  higher index map (See \cite{Kasparov:1983}) 
\begin{equation}
\label{eq:AssemblyMfd}
\mu: K^0_G(C_0(X))\longrightarrow K_0(C^*(G)).
\end{equation}
We refer to~\cite{Kasparov:1988dw, Ba:KTOA} for the definition and basic properties of $KK$-group and its relation to $K$-theory and $K$-homology. 
In Section~\ref{sec:AnalyticKIndex}, we shall construct the higher index  map 
  for a complete Riemannian orbifold $\f{X}$  with a proper, co-compact  and isometric action of a discrete  group $G$. Then we calculate  the  higher index of $\Dirac^\cE$ as a $KK$-cycle in Section~\ref{sec:K-IndexD}. After that we shall show that the  higher index  of $\Dirac^\cE$ is related to the orbifold index of $\Dirac^{G\backslash\cE}$ via the trivial representation of $G$ ({\em cf.} Theorem \ref{Thm:TrivialRepK-theoreticIndex=OrbifoldIndex}).  

\subsection{Higher analytic index map}
\label{sec:AnalyticKIndex}
We formulate the higher index in the context of proper co-compact $G$-orbifold following the philosophy of \cite{Kasparov:1983}. 
In view of Lemma~\ref{lem:quotient} on the local structure of the $G$-orbifold $\f{X}$, choose the corresponding proper \'etale grouppoid $\cG=(\cG_1\rightrightarrows \cG_0)$, i.e. $|\f{X}|\cong \cG_0/\cG_1$, which locally looks like ({\em cf.} Remark~\ref{rem:OrbifoldGroupoidStructure})
\begin{equation*}
(G\times_{G_i}\tilde{U_i})\rtimes H_i\rightrightarrows G\times_{G_i}\tilde{U_i}
\end{equation*}
and naturally admits a $G$-action on the left. 
As $G$ acts properly and co-compactly on $\f{X}$, $G$ acts properly and co-compactly on $\cG.$
Let $C_{red}^*(\f{X})$ be the reduced $C^*$-algebra given by Definition \ref{GroupCstarAlgebra}. 
As the reduced $C^*$-norm is preserved under the left action of $G$ on $C_c^{\infty}(\cG_1)$
\begin{equation*}
(h\cdot f)(g)=f(h^{-1}\cdot g)\qquad\forall h\in G, \forall f\in C_c^{\infty}(\cG_1),
\end{equation*}
the $G$-action extends to $C_{red}^*(\f{X})$.
The convolution algebra $C_c(G, C_{red}^*(\f{X}))$ is represented as a set of bounded operators on $L^2(G, C_{red}^*(\f{X}))$ given by the integration of the left regular representation of $G$ on $L^2(G, C_{red}^*(\f{X}))$
\[
C_c(G, C_{red}^*(\f{X}))\longrightarrow\cB(L^2(G, C_{red}^*(\f{X}))).
\]
Denote by $C_{red}^*(\f{X})\rtimes_r G$ the closure of $C_c(G, C_{red}^*(\f{X}))$  
  under the operator norm of $\cB(L^2(G, C_{red}^*(\f{X})))$. 
Denote by $C_{red}^*(\f{X})\rtimes G$ the closure of the maximal operator norm of all the covariant representations of the convolution algebra $C_c(G, C_{red}^*(\f{X})).$

A first ingredient is a projection in $C_{red}^*(\f{X})\rtimes G$ constructed from a cut-off function $c$.
Let $\cP$ be the closure of $C_c^{\infty}(\cG_1)$ under the norm of $C_{red}^*(\f{X})\rtimes G$. 
Then $\cP$ is a Hilbert $C_{red}^*(\f{X})\rtimes G$-module.  
Let $c\in C_c^{\infty}(\cG)$ be a cut-off function of $\f{X}$ associated to the $G$ action ({\em cf.} Definition~\ref{def:cut-offFunction}) and $s: \cG_1\rightarrow\cG_0$ be the source map. 
Then the pullback function 
\begin{equation*}
s^*c(\gamma):=c(s(\gamma)), \quad \forall \gamma\in \cG_1
\end{equation*}
gives rise to a ``cut-off" function in $C^{\infty}_c(\cG_1),$ which satisfies 
\begin{equation*}
\sum_{g\in G}s^*c(g^{-1}\gamma)=\sum_{g\in G}c(s(g^{-1}\gamma))=\sum_{g\in G}c(g\cdot s(\gamma))=1.
\end{equation*}
Define a projection $p\in C_c(G, C_c^{\infty}(\cG_1))\subset C_{red}^*(\f{X})\rtimes G$ by 
\begin{equation}\label{eq:Projection}
p(g)(x)=[s^*c(g^{-1}x)s^*c(x)]^{\frac12}, \quad \forall g\in G, \forall x\in\cG_1.
\end{equation}
The Hilbert module $\cP$ is related to $p$. In fact, the image of the convoluting operation of $p$ on $C_{red}^*(\f{X})\rtimes G$ is $\cP$, i.e., 
\begin{equation}
\label{eq:p&P}
p\cdot [C_{red}^*(\f{X})\rtimes G]=\cP.
\end{equation}
They represent the same element in the following identifications.
\begin{equation}
\label{eq:ProjCuttoffKTheory}
\begin{split}
KK(\C, C_{red}^*(\f{X})\rtimes G)&  \cong K_0(C_{red}^*(\f{X})\rtimes G)\cong K_0(C_{red}^*(G\backslash \f{X}))\\
[(\cP, 1_{\C}, 0)]&  \mapsto[p]\mapsto[1].
\end{split}
\end{equation}

The second ingredient is the analytic $K$-homology of $C^*_{red}(\f{X}).$
Let $H$ be a $\Z/2\Z$-graded Hilbert space equipped with a unitary representation $\pi$ of the group $G$ and with a $*$-homomorphism $\phi: C^{*}_{red}(\f{X})\rightarrow\cB(H)$, where the two representations respect the $\Z/2\Z$-grading and are compatible with the action of $G$ on $C^*_{red}(\f{X})$:
\begin{equation*}
\phi(g\cdot a)=\pi(g)\phi(a)\pi(g)^{-1}\qquad \forall g\in G, \forall a\in C^*_{red}(\f{X}).
\end{equation*}
Let $F=\begin{pmatrix}0 & F_- \\ F_+ & 0 \end{pmatrix}$, where $F_+^*=F_-$  be a bounded operator on $H$ such that
\begin{equation}
\label{eq:KHomologyConditions}
\phi(a)(F^2-1), [\phi(a), F], [\pi(g), F]\in\cK(H)\qquad \forall a\in C^*_{red}(\f{X}), \forall g\in G.
\end{equation}
Here $\cK(H)$ is the set of compact operators on $H$ and $[\cdot, \cdot]$ is the graded commutator.
Then the triple $(H, \pi, F)$ gives rise to a $K$-homology element in $K^0_G(C^*_{red}(\f{X}))$. 
This $K$-homology cycle is the abstract model for $G$-invariant elliptic pseudo-differential  operators of order $0$ on $\f{X}$. 
Denote by $[F]$ the equivalence class $[(H, \pi, F)]\in K^0_G(C^*_{red}(\f{X})).$

\begin{definition}[\bf Analytic index \cite{Kasparov:2008}]
The {\bf analytic $K$-theoretic index map} is the homomorphism $\mu$
\begin{equation}\label{eq:K-theoreticIndex}
\mu: K_G^0(C_{red}^*(\f{X}))\longrightarrow K_0(C^*(G))
\end{equation}
from a $K$-homology element $[F]\in K_G^0(C_{red}^*(\f{X}))$ to the $KK$-product of the following elements 
 \begin{equation*}
 [\cP]\in KK(\C, C^*_{red}(\f{X})\rtimes G) \quad \text{and} \quad  j^G([F])\in KK_0(C_{red}^*(\f{X})\rtimes G, C^*(G)).
 \end{equation*}
 Here, $j^G$ given by
\begin{equation}
j^G: KK^G(C_{red}^*(\f{X}), \C)\longrightarrow KK_0(C_{red}^*(\f{X})\rtimes G, C^*(G))
\end{equation}
the descent homomorphism ({\em cf.} \cite{Kasparov:1988dw} 3.11).
\end{definition}

In order to accommodate the localized indices to be introduced in the next section, we introduce some variations of the analytic index map taking values in $K$-theory of some completions of $\C G$ in some other norms.

First of all, recall that the left regular representation of $L^1(G)$ on $L^2(G)$ extends to a natural surjective $*$-homomorphism $r: C^*(G)\rightarrow C^*_r(G)$, which gives rise to a $K$-theory homomorphism 
\begin{equation*}
r_*: K_0(C^*(G))\longrightarrow K_0(C^*_r(G)).
\end{equation*}
Composing $r_*$ with~(\ref{eq:K-theoreticIndex}) gives rise to the reduced version of the analytic index map.

\begin{definition}[\bf Analytic index (reduced version)]
The {\bf reduced} analytic $K$-theoretic index map is the homomorphism $\mu_{red}:=r_*\circ\mu$
\begin{equation}\label{eq:K-theoreticIndexred}
\mu_{red}: K_G^0(C_{red}^*(\f{X}))\longrightarrow K_0(C^*_r(G))
\end{equation}
from a $K$-homology element $[F]\in K_G^0(C_{red}^*(\f{X}))$ to the $KK$-product of $r_*[\cP]\in KK(\C, C^*_{red}(\f{X})\rtimes_r G)$ with
 \begin{equation*}
    j_r^G([F]):=r_*\circ j^G[F]\in KK_0(C_{red}^*(\f{X})\rtimes_r G, C^*_r(G)).
 \end{equation*}
\end{definition}

\begin{remark}
Observe that the $K$-theory element $r_*[\cP]=[r\circ\cP]$ is represented by $\cP_{red}:= r(\cP)$, which is the $C_{red}^*(\f{X})\rtimes_r G$-module given by $r(\cP)=p[C^*_{red}(\f{X})\rtimes_r G]$ in view of~(\ref{eq:p&P}).
\end{remark}

Moreover, assume that $\cS(G)$ is a Banach algebra containing $\C G$ as a dense subalgebra. Then there is a condition for $\cS(G)$ due to Lafforgue~\cite{LafforgueThesis}, which gives rise to the Banach algebra version of the analytic index map.

\begin{definition}[{\bf Unconditional completion} \cite{LafforgueThesis}] Let $\cS(G)$ be a Banach algebra equipped with a norm $\|\cdot\|_{\cS(G)}$ and containing $\C G$ as a dense subalgebra. Then $\cS(G)$ is called an {\bf unconditional completion} of $\C G$ if for any $f_1, f_2\in\C G$ satisfying $|f_1(g)|\le|f_2(g)|, \forall g\in G$, we have $\|f_1\|_{\cS(G)}\le\|f_2\|_{\cS(G)}.$ 
\end{definition}

Let $B$ is a $G$-Banach-algebra with norm $\|\cdot\|_B$. Denote by $\cS(G, B)$ the completion of $C_c(G, B)$ with respect to the norm 
\begin{equation*}
\|\sum_{g\in G}a_gg\|:=\|\sum_{g\in G}\|a_g\|_{B}g\|_{\cS(G)}, \quad  \sum_{g\in G}a_g g\in C_c(G, B).
\end{equation*}
If $\cS(G)$ is an unconditional completion, then by~\cite{LafforgueThesis} the descent map formulated using Banach $KK$-theory is well-defined 
\begin{equation}
\label{eq:descentBanKK}
j^{G}_{\cS(G)}: KK(C^*_{red}(\f{X}), \C)\rightarrow KK^{ban}(\cS(G, C^*_{red}(\f{X})), \cS(G)).
\end{equation}
In~(\ref{eq:descentBanKK}) we used the $KK$-group $KK^{ban}_G(A, B)$ associated to two $G$-Banach algebras $A$ and $B$. The group is defined by generalized Kasparov cycles of the form $(E, \phi, F)$ modulo suitable equivalence relations~\cite{LafforgueThesis}. 
Here, $E$ is a $\Z/2\Z$-graded Banach $B$-module and $F\in\cB(E)$ is an odd selfadjoint operator. In addition, $G$ represents in $\cB(E)$ as a grading preserving unitary representation and $\phi: A\rightarrow\cB(E)$ is a grading preserving homomorphism.

 Denote by $\cK(E)$ the compact operator over $E$, then the triple $(E, \phi, F)$ is a {\bf generalized Kasparov cycle} if 
\begin{equation*}
[\phi(a), F], \phi(a)(F^2-1), [\pi(g), F]\in\cK(E) \qquad \forall a\in A, \forall g\in G.
\end{equation*} 
Analogously, the Banach  algebra version of analytic index map can be defined as follows.

\begin{definition}[{\bf Analytic index (Banach algebra version)} \cite{LafforgueThesis}]
\label{def:BanachVHigherIndex}
Let $\cS(G)$ be a unconditional completion of $\C G$. 
The {\bf Banach algebra  version} of the analytic $K$-theoretic index map is the homomorphism $\mu_{\cS(G)}$
\begin{equation}\label{eq:K-theoreticIndexBan}
\mu_{\cS(G)}: K_G^0(C_{red}^*(\f{X}))\longrightarrow K_0(\cS(G))
\end{equation}
from a $K$-homology element $[F]\in K_G^0(C_{red}^*(\f{X}))$ to the $KK$-product of the element 
\begin{equation*}
[\cP_{\cS(G)}]\in KK^{ban}(\C, \cS(G, C^*_{red}(\f{X}))),
\end{equation*} 
represented by the Banach $\cS(G, C^*_{red}(\f{X}))$-module $\cP_{\cS(G)}:=p[\cS(G, C^*_{red}(\f{X}))]$,
with the element 
\begin{equation*}
j^G_{\cS(G)}([F])\in KK^{ban}(\cS(G, C^*_{red}(\f{X})), \cS(G))
\end{equation*}
given by~(\ref{eq:descentBanKK}).
\end{definition}

\begin{remark}
The Banach algebra $L^1(G)\subset C^*_r(G)$ is an unconditional completion of $\C G$. 
In general, if we have an inclusion $i^{\cS, C_r^*}: \cS(G)\rightarrow C^*_{r}(G)$ for $\cS(G)$, then the (reduced) higher index and the Banach algebra version higher index are related as follows 
\begin{equation*}
r_*\mu[F]=\mu_{red}[F]=i^{\cS, C_r^*}_{*}\mu_{\cS(G)}[F], \qquad \forall [F]\in K_G^0(C^*_{red}(\f{X})).
\end{equation*}  
\end{remark}

\subsection{Higher  index for a discrete group action on a non-compact orbifold}
\label{sec:K-IndexD}
Given the abstract setting of the higher index maps in Section~\ref{sec:AnalyticKIndex}, we shall describe the $K$-homological cycles in $K^0_G(C_{red}^*(\f{X}))$ of the Dirac operator $\Dirac^{\cE}$ and its analytic index in $K_0(C^*(G))$ and in $K_0(\cS(G))$ in terms of (generalized) Kasparov cycles.

Let $G$ be a discrete group acting properly, co-compactly and isometrically on a complete Riemannian orbifold $\f{X}$. 
Let $\cE$ be a Hermitian orbifold vector bundle over $\f{X}$. 
The Hilbert space $L^2(\f{X}, \cE)$ is the completion of the compactly supported smooth sections $\Gamma_c(\f{X}, \cE)$ with respect to the following inner product
\begin{equation}
\label{eq:GammaXEHilbertSpIP}
\begin{split}
&\langle f, g\rangle_{L^2}=\int_{\f{X}}^{orb}\langle f(x), g(x) \rangle_{\cE_x} d\vol_{\f{X}}(x) \qquad \forall f, g\in\Gamma_c(\f{X}, \cE).
\end{split}
\end{equation}
It is a $G$-algebra with the action given by
\begin{equation*}
[g\cdot f](x):=gf(g^{-1}x), \qquad \forall g\in G, \forall f\in L^2(\f{X}, \cE). 
\end{equation*}
A natural representation $\phi: C_{red}^*(\f{X})\rightarrow\cB(L^2(\f{X}, \cE))$ is given by the following natural action: $\forall m\in C_c^{\infty}(\cG_1), \forall x\in\cG_0$ and $\forall f\in L^2(\f{X}, \cE),$
\begin{equation}\label{eq:RepresentationOrbifoldCStarAlgebra}
[m\cdot f](x)=\sum_{g\in s^{-1}(x)}m(g)(g\cdot f)(x)=\sum_{g\in s^{-1}(x)}m(g)[g(f(g^{-1}x))]\in\tilde\cE_x.
\end{equation}

As we have discussed in Section~\ref{sec:EllipticPDO}, a properly supported $G$-invariant pseudo-differential operator $D_{\f{X}}$ on $\Gamma_c(\f{X}, \cE)\rightarrow \Gamma_c(\f{X}, \cE)$ of order $m$ extends to a bounded linear operator between Sobolev spaces
\begin{equation*}
D_{\f{X}}: L^2_k(\f{X}, \cE)\rightarrow L^2_{k-m}(\f{X}, \cE)
\end{equation*} 
for all $k\ge m.$
By the compact embedding theorem, for any $f\in C_c(\cG_1)$ and for any pseudodifferntial operator $D_{\f{X}}$ of negative order, the operator $\phi(f)D_{\f{X}}$ is compact.

\begin{lemma}\label{le:KHomologyElement}
Let $\Dirac^{\cE}$ be a Dirac operator on $\f{X}$.
Then the Hilbert space $L^2(\f{X}, \cE)$, the representation ~(\ref{eq:RepresentationOrbifoldCStarAlgebra}) and the operator $F=\frac{\Dirac^{\cE}}{\sqrt{1+(\Dirac^{\cE})^2}}$ form a cycle in the $K$-homology group 
\begin{equation*}
[\Dirac^{\cE}]:=[(L^2(\f{X}, \cE), \phi, F)]\in K^0_G(C_{red}^*(\f{X})).
\end{equation*}
\end{lemma}
\begin{proof}
Note that $\Dirac^{\cE}=\begin{pmatrix}0 & \Dirac^{\cE}_- \\ \Dirac^{\cE}_+ & 0 \end{pmatrix}$ is   a $G$-invariant   Dirac operator on $L^2(\f{X}, \cE).$ Then $F=\frac{\Dirac^{\cE}}{\sqrt{1+(\Dirac^{\cE})^2}}$ is an order $0$ pseudo-differential operator on $L^2(\f{X}, \cE).$
 $F$  can be extended  to a bounded operator on $L^2(\f{X}, \cE)$. 
Observe that $1-F^2=(1+(\Dirac^{\cE})^2)^{-1}$ is a pseudo-differential operator of order $-2$.
Thus, 
\begin{equation*}
\phi(f)(1-F^2)\in\cK(L^2(\f{X}, \cE)) \qquad \forall f\in C_0(\f{X}) =  C_0(\cG).
\end{equation*}
Note that by the $G$-invariance of $F$, we have $[F, \pi(g)]=0, \forall g\in G$. 
 It is straightforward to check $[F, \phi(f)]\in\cK(L^2(\f{X}, \cE)), \forall f\in C_c(\cG_1)$. 
For  the conditions in~(\ref{eq:KHomologyConditions}), we only need to check  for $f\in C_c(\cG_1)$ the dense subalgebra of $C^*_{red}(\f{X})$, which can be done as in  \cite{Ba:KTOA} for  the manifold cases.
\end{proof}

\begin{definition} 
The {\bf higher index} of $\Dirac^{\cE}$,denoted by $\Ind \Dirac^{\cE}$,  is defined to be
\[
 \Ind \Dirac^{\cE} = \mu ([\Dirac^{\cE}]) \in K_0(C^*(G)), 
 \]
 the image in $K_0(C^*(G))$ of $[\Dirac^{\cE}]$ under  the analytic index map (\ref{eq:K-theoreticIndex}).
 \end{definition}

We now  calculate the  higher  index of $\Dirac^{\cE}$, as a  $KK$-cycle in $K_0(C^*(G))\cong KK(\C, C^*(G))$,  following from the work of Kasparov~\cite{Kasparov:2008}.

Recall that in~(\ref{eq:GammaXEHilbertSpIP}) the  algebra $\Gamma_c(\f{X}, \cE)$ is equipped with a pre-Hilbert space inner product. However, this inner product is not sufficient for us to derive an analogous Fredholm property for $\Dirac^{\cE}$. 
Instead, we equip $\Gamma_c(\f{X}, \cE)$ a $\C G$-valued inner product given by 
\begin{equation}
\label{eq:C_cGInnerProduct}
\begin{split}
\langle f_1, f_2\rangle_{\C G}(g):=&\int_{\f{X}}^{orb}(f_1(x), g(f_2(g^{-1}x)))_{\cE_x} d\vol_{\f{X}}(x) \\
(f\cdot b)(x):=&\sum_{g\in G} g(f(g^{-1}x))b(g^{-1})  
\end{split}
\end{equation}
for all $f, f_1, f_2\in \Gamma_c(\f{X}, \cE), g\in G,$ and $b\in \C G, x\in\cG_0.$
It is routine to check that~(\ref{eq:C_cGInnerProduct}) gives rise to a pre-Hilbert $\C G$-module. For example, by~(\ref{eq:C_cGInnerProduct}) we have 
\begin{equation*}
\langle f_1, f_2 b\rangle_{\C G}=\langle f_1, f_2\rangle_{\C G}b,\qquad  \forall f_1, f_2\in\Gamma_c(\f{X}, \cE), b\in\C G.
\end{equation*}

Denote by $\cA$ the completion of $\Gamma_c(\f{X}, \cE)$ in the Hilbert $C^*(G)$-norm given by the inner product~(\ref{eq:C_cGInnerProduct}).
Let $c\in C^{\infty}_c(\cG)$ be a cut-off function of $\f{X}$ with respect to the $G$-action, where the pullback $s^*c^{\frac12}\in C_c^{\infty}(\cG_1)$ acts on $g\cdot e\in L^2(\f{X}, \cE)$ in the sense of (\ref{eq:RepresentationOrbifoldCStarAlgebra}).
Then there is an inclusion $\iota: \Gamma_c(\f{X}, \cE)\hookrightarrow C_c(G, L^2(\f{X}, \cE))$ given by
\begin{equation}
\label{eq:injection}
\iota(e)(g)=c^{\frac12}\cdot [g\cdot e] \qquad \forall e\in \Gamma_c(\f{X}, \cE), \forall g\in G.
\end{equation}
Denote by $L^2(\f{X}, \cE)\rtimes G$ the maximal crossed product as the completion of $C_c(G, L^2(\f{X}, \cE))$ under the norm $\|\sum_{g\in G}\|a_g\|_{L^2}g\|_{max}$ for $f=\sum_{g\in G}a_g g\in C_c(G, L^2(\f{X}, \cE))$, where $\|\cdot\|_{max}$ is the norm of $C^*(G)$.
It is a Hilbert $C^*(G)$-module.

\begin{proposition}[\cite{Kasparov:2008}]
The inclusion map $\iota$ given by (\ref{eq:injection}) extends to a injective homomorphism between Hilbert $C^*(G)$-modules 
\begin{equation*}
\iota: \cA\hookrightarrow L^2(\f{X}, \cE)\rtimes G.
\end{equation*}
Moreover, $\cA$ is a direct summand of $L^2(\f{X}, \cE)\rtimes G$ as a Hilbert $C^*(G)$-submodule.
\end{proposition}

\begin{proof}
Notice that $C_c^{\infty}(G, L^2(\f{X}, \cE))$ carries a natural $\C G$-module structure given by convolution.  
It is easy to check that the inclusion $\iota$ is compatible with the pre-Hilbert $\C G$-module structures.
By taking the completion, $\cA$ is regarded as a Hilbert $C^*(G)$-submodule for $L^2(\f{X}, \cE)\rtimes G$.
To show the second claim, note that $\iota$ admits an adjoint $q:=\iota^*$ with respect to the $\C G$-valued inner product on $C_c^{\infty}(G, L^2(\f{X}, \cE))$ given by
\begin{equation}
\label{eq:AdjointOfInjection}
q: C_c^{\infty}(G, L^2(\f{X}, \cE))\rightarrow \Gamma_c(\f{X}, \cE) 
\qquad 
 f\mapsto \sum_{g\in G} c(g^{-1}\cdot)^{\frac12}\cdot \{g\cdot [f(g^{-1})]\}.
\end{equation} 
As a consequence $q\circ \iota$ is identity on $\Gamma_c(\f{X}, \cE)$ and $\iota \circ q$ is the projection from $C_c^{\infty}(G, L^2(\f{X}, \cE))$ to $\Gamma_c(\f{X}, \cE)$.
The projection $\iota\circ q$ extends to $L^2(\f{X}, \cE)\rtimes G$ and has $\cA$ as its image.
The proposition is then proved.
\end{proof}

\begin{remark}
\label{rem:AandP}
The projection map $\iota\circ q$ applied to $L^2(\f{X}, \cE)\rtimes G$ agrees with the left convolution of the projection $p$ (defined in (\ref{eq:Projection})) with $L^2(\f{X}, \cE)\rtimes G.$ 
In other words, we have 
\begin{equation}
\label{eq:AinRem}
p[L^2(\f{X}, \cE)\rtimes G]=\cA.
\end{equation}
\end{remark}
 
\begin{proposition}[\cite{Kasparov:2008}]
\label{prop:HigherIndexKK}
Let $\Dirac^{\cE}$ be a Dirac operator on $\f{X}$ and let $[\Dirac^{\cE}]$ be the $K$-homology element in Lemma~\ref{le:KHomologyElement}. Then the $KK$-cycle of the higher index $\Ind \Dirac^{\cE}\in K_0(C^*(G))$ is given by
\begin{equation}
\label{eq:HigherIndexKK}
[(\cA, 1_{\C}, p\tilde F p)]=[(q(L^2(\f{X}, \cE)\rtimes G), 1_{\C}, q\circ\tilde F\circ i)]\in KK(\C, C^*(G)).
\end{equation}
Here, $\tilde F$ is the lift of the operator $F=\Dirac^{\cE}(1+(\Dirac^{\cE})^2)^{-\frac12}$ to $L^2(\f{X}, \cE)\rtimes G$ given by
\begin{equation}
\label{eq:LiftedOperatorF}
[\tilde F(h)]g=F(h(g)), \qquad \forall h\in C_c(G, L^2(\f{X}, \cE)), \forall g\in G.
\end{equation}
\end{proposition} 
 
\begin{proof} 
By definition of the higher index, we have 
\begin{equation*}
\Ind  \Dirac^{\cE} =\mu[F]=[\cP]\otimes_{C_{red}^*(\f{X})\rtimes G}j^G([L^2(\f{X}, \cE), \phi, F]),
\end{equation*}
where $F=\frac{\Dirac^{\cE}}{\sqrt{1+(\Dirac^{\cE})^2}}$ is the bounded operator on $\Gamma_c(\f{X}, \cE)$.
Denote by $\tilde F$ the lift of the operator $F$ to $L^2(\f{X}, \cE)\rtimes G$ given by~(\ref{eq:LiftedOperatorF}).
Then the image of $[F]$ under the descent map $j^G$ is given by
\begin{equation*}
j^G([L^2(\f{X}, \cE), \phi, F])=[L^2(\f{X}, \cE)\rtimes G, \tilde\phi, \tilde F].
\end{equation*}
Thus, we have
\begin{align*}
\Ind  \Dirac^{\cE}=&[\cP, 1_{\C}, 0]\otimes_{C_{red}^*(\f{X})\rtimes G} [L^2(\f{X}, \cE)\rtimes G, \tilde\phi, \tilde F]\\
=&[\cP\otimes_{C_{red}^*(\f{X})\rtimes G}[L^2(\f{X}, \cE)\rtimes G], 1_{\C}, \cP\tilde F\cP].
\end{align*}

The statement is proved by noting that $\cA$ is the product of the Hilbert $C_{red}^*(\f{X})\rtimes G$-module $\cP$ and the Hilbert $C^*(G)$-module $L^2(\f{X}, \cE)\rtimes G$. 
In fact, by Remark~\ref{rem:AandP}, we have
\begin{equation*}
\label{eq:ProductAlgebra}
\begin{split}
\cP\otimes_{C_{red}^*(\f{X})\rtimes G}[L^2(\f{X}, \cE)\rtimes G]
=&p\cdot[C_{red}^*(\f{X})\rtimes G]\otimes_{C_{red}^*(\f{X})\rtimes G} [L^2(\f{X}, \cE)\rtimes G]\\
=&p\cdot[L^2(\f{X}, \cE)\rtimes G]=\cA.
\end{split}
\end{equation*}

In addition, the compression $\cP\tilde F\cP$ of $\tilde F$ and $\cP$ on $p\cdot[L^2(\f{X}, \cE)\rtimes G]=\cA$ is alternatively written as $q\circ \tilde F\circ \iota$ on $q(L^2(\f{X}, \cE)\rtimes G)\cong\cA.$ This follows from 
$\iota\circ q=p\cdot$ on $L^2(\f{X}, \cE)\rtimes G.$ 
Therefore, the  higher  index of $\Dirac^{\cE}$ is represented by the following $KK$-cycle.
\begin{equation}
\label{eq:HigherIndexKKBan}
[(\cA, 1_{\C}, p\tilde F p)]=[(q(L^2(\f{X}, \cE)\rtimes G), 1_{\C}, q\circ\tilde F\circ i)]\in KK(\C, C^*(G)).
\end{equation}
The proposition is then proved.
\end{proof}

Let $\cS(G)$ be an unconditional completion of $\C G$, under a Banach norm $\|\cdot\|_{\cS(G)}$. Similarly we have the following norm-preserving inclusion between Banach $\cS(G)$-modules extending the map~(\ref{eq:injection})
\begin{equation*}
\iota_{\cS(G)}: \cA_{\cS(G)}\longrightarrow \cS(G, L^2(\f{X}, \cE)).
\end{equation*}
Here, $\cA_{\cS(G)}$ is the Banach $\cS(G)$-module in the same fashion as~(\ref{eq:AinRem}) given by
 \begin{equation}
 \label{eq:A_S(G)}
 \cA_{\cS(G)}:= p[\cS(G, L^2(\f{X}, \cE))].
 \end{equation} 
It is a direct summand of $\cS(G, L^2(\f{X}, \cE)).$
Denote by $\tilde F$ the lift of $F$ from $\Gamma_c(\f{X}, \cE)$ to $\cS(G, L^2(\f{X}, \cE))$ given by~(\ref{eq:LiftedOperatorF}).
Then the Banach algebra version of the higher index is then stated as follows. 

\begin{proposition}
\label{prop:HigherIndexKKBan}
Let $\Dirac^{\cE}$ be a Dirac operator on $\f{X}$ and let $[\Dirac^{\cE}]$ be the $K$-homology element in Lemma~\ref{le:KHomologyElement}. Then the $KK$-cycle of the higher index $\mu_{\cS(G)}[\Dirac^{\cE}]\in K_0(\cS(G))$ is given by the following $KK$-cycle:
\begin{equation}
\label{eq:HigherIndexKKban}
[(\cA_{\cS(G)}, 1_{\C}, p\tilde F p)]=[(q( \cS(G, L^2(\f{X}, \cE))), 1_{\C}, q\circ\tilde F\circ i)].
\end{equation}
Here, $\tilde F$ is the lift of the operator $F=\Dirac^{\cE}[1+(\Dirac^{\cE})^2]^{-\frac12}$ to $\cS(G, L^2(\f{X}, \cE)$ given by~(\ref{eq:LiftedOperatorF}).
\end{proposition} 

\subsection{Orbifold index and  the  higher  index}
We relate the higher index of $\Dirac^{\cE}$ to the Kawasaki index of $\Dirac^{G\backslash \cE}$ on $G\backslash \f{X}.$ This is essentially a result of Theorem~\ref{thm:HeatKernelAsympMain} and Theorem~\ref{eq:LocalIndexFormula}. Here we present an alternative proof which uses only $KK$-theory.

\begin{theorem}\label{Thm:TrivialRepK-theoreticIndex=OrbifoldIndex}
The Kawasaki's index for closed orbifold $G\backslash \f{X}$ is equal to the trivial representation of $G$ induced on the higher index of $[\Dirac^{\cE}]$: 
\begin{equation}\label{eq:HigherIndexToOrbifoldIndex}
\ind \Dirac^{G\backslash \cE}=\rho_*(\mu[\Dirac^{\cE}])
\end{equation}
where $\rho_*: K_0(C^*(G))\rightarrow\Z$ is induced by 
\begin{equation}\label{eq:Homomorphism and Trace}
\rho: C^{\ast}(G)\rightarrow\C: \sum \alpha_g g\mapsto \sum\alpha_g
\end{equation} 
and $\mu$ is the higher index map (\ref{eq:K-theoreticIndex}).
\end{theorem} 

\begin{proof}
Let $F=\frac{\Dirac^{\cE}}{\sqrt{1+(\Dirac^{\cE})^2}}$ and $F_{G\backslash \f{X}}=\frac{\Dirac^{G\backslash \cE}}{\sqrt{1+(\Dirac^{G\backslash \cE})^2}}.$
Then the action of $F_{G\backslash \f{X}}$ on $\Gamma(G\backslash \f{X}, G\backslash \cE)$ can be identified to that of $F$ on the $G$-equivariant sections $\Gamma(\f{X}, \cE)^G.$
The functorial map $\rho_*$ applied to the higher index $\mu[\Dirac^{\cE}]$ is essentially the $KK$-product of $\mu[F]$ by $[(\C, \rho, 0)]$ over $C^*(G)$ in the following map (See~\cite{Ba:KTOA})
\[
KK(\C, C^*(G))\times KK(C^*(G), \C)\rightarrow KK(\C,\C)\cong\Z.
\]
Then the right hand side of (\ref{eq:HigherIndexToOrbifoldIndex}) is
\begin{equation}
\label{eq:OrbifoldKIndexRHS}
\begin{split}
\rho_*(\mu[F])=&[(\cA, 1_{\C}, q\circ\tilde F\circ \iota)]\otimes_{C^*(G)}[(\C, \rho, 0)]\\
=&[(\cA\otimes_{C^*(G)}\C, 1_{\C}, (q\circ\tilde F\circ \iota )\otimes 1)]
\end{split}
\end{equation}
where $\cA=q(L^2(\f{X},\cE)\rtimes G)$ and by (\ref{eq:injection}), (\ref{eq:AdjointOfInjection}) and (\ref{eq:LiftedOperatorF}) we have
\begin{equation}
\label{eq:OrbioldKIndexCompression}
q\circ\tilde F\circ \iota =\sum_{g\in G}g(c^{\frac12}Fc^{\frac12}).
\end{equation}

In the following we shall identify $L^2(G\backslash \f{X}, G\backslash\cE)$ and $\cA\otimes_{C^*(G)}\C$.
For a fixed cut-off function $c$, we have the map
\begin{equation*}
j: \Gamma(G\backslash \f{X}, G\backslash\cE)\longrightarrow \Gamma_c(\f{X}, \cE)\otimes_{\C G}\C \qquad f\mapsto c\cdot\tilde f\otimes1.
\end{equation*}
Here, $\Gamma_c(\f{X}, \cE)$ is a pre-Hilbert $\C G$-module and $\tilde f$ is the image of $f$ under the natural map 
$\Gamma(G\backslash \f{X}, G\backslash\cE)\cong \Gamma(\f{X}, \cE)^G.$ 
We show that $j(f)$ does not depend on the choice of $c.$ 
In fact, for another cut-off function $d$ and any $h\otimes1\in \Gamma_c(\f{X}, \cE)\otimes_{\C G}\C$, as $\tilde f$ is $G$-invariant, we have 
\begin{equation*}
\langle (d-c)\tilde f\otimes 1, h\otimes 1 \rangle=\sum_{g\in G}\langle [(d-c)\tilde f](gx), h(x)\rangle=0.
\end{equation*}
Hence the nondegeneracy of the inner product implies that $d\tilde f\otimes1=c\tilde f\otimes1.$ 

We claim that $j$ preserves the inner products. 
In fact, the claim follows from
\begin{align*}
\langle f, h \rangle=&\int_{G\backslash \f{X}}^{orb} (f(x), h(x))_{(G\backslash\cE)_x}d\vol_{G\backslash\f{X}}(x)\\
=&\disp{\int}_{\f{X}}^{orb}c(x) (f(x), h(x))_{\tilde{\cE}_x}d\vol_{\f{X}}(x)\quad \text{and}
\end{align*}
\begin{align*}
\langle j(f), j(h) \rangle=&\sum_{g\in G}\disp{\int}_{\f{X}}^{orb}(c(x)\tilde f(x), c(g^{-1}x)\tilde h(gx))_{\tilde{\cE}_x}d\vol_{\f{X}}(x)\\
=&\disp{\int}_{\f{X}}^{orb}(c(x)\tilde f(x), \sum_{g\in G}c(g^{-1}x)\tilde h(x))_{\tilde{\cE}_x}d\vol_{\f{X}}(x)
\end{align*}
for all $ f, h\in C(G\backslash \f{X}, G\backslash\cE).$
Hence, the map $j$ extends to an isomorphism of two Hilbert spaces
\begin{equation*}
j: L^2(G\backslash \f{X}, G\backslash\cE)\rightarrow\cA\otimes_{\C G}\C.
\end{equation*}
It is straight forward to check that the inverse of $j$ is given by 
\begin{equation*}
j^{-1}: \cA\otimes_{C^{*}(G)}\C\rightarrow L^2(G\backslash \f{X}, G\backslash\cE) \quad h\otimes1\mapsto \sum_{g\in G} h(g^{-1}\cdot).
\end{equation*}
Then together with (\ref{eq:OrbifoldKIndexRHS}) and (\ref{eq:OrbioldKIndexCompression}) we have
\begin{equation*}
\rho_*(\mu[F])=[(\cA\otimes_{C^*(G)}\C, 1_{\C}, \sum_{g\in G}g(c^{\frac12}Fc^{\frac12})\otimes 1)]
=[(L^2(G\backslash \f{X}, G\backslash\cE), F_0)],
\end{equation*}
where 
\begin{equation*}
F_0=j^{-1}\circ[\sum_{g\in G}g(c^{\frac12}Fc^{\frac12})\otimes 1]\circ j=\sum_{l,g\in G}c(g^{-1}l^{-1}x)^{\frac12}Fc(g^{-1}l^{-1}x)^{\frac12}c(l^{-1}x).
\end{equation*}

Finally, observe that the left hand side of (\ref{eq:HigherIndexToOrbifoldIndex}) is 
\[
\ind \Dirac^{G\backslash \cE}=[(L^2(G\backslash \f{X}, G\backslash\cE), 1_{\C}, F_{G\backslash \f{X}})].
\]
One need only to show that $F_0$ and $F$ coincide up to compact operators on $\Gamma(\f{X}, \cE)^G$ (denoted $F_0\equiv F$), i.e., they have the same Fredholm index. 
As we have
 $\sum_{g\in G}g(c^{\frac12}Fc^{\frac12})\equiv\sum_{g\in G}g(cF)=F.$
Then 
\begin{equation*}
F_0=\sum_{l\in G}l[\sum_{g\in G}g(c^{\frac12}Fc^{\frac12})c]\equiv\sum_{l\in G}l(Fc)=F=F_{G\backslash \f{X}}.
\end{equation*}
The theorem is proved.
 \end{proof}

\begin{remark}
It is important to emphasis $G$ being discrete to ensure 
\[
\Gamma(\f{X},\cE)^G\cong \Gamma(G\backslash \f{X}, G\backslash\cE)
\]
 and so that $\Dirac^{G\backslash \cE}$ is a restriction of $\Dirac^{\cE}$ to the invariant sections. We used this identification in the proof of Theorem \ref{Thm:TrivialRepK-theoreticIndex=OrbifoldIndex}.  
If $G$ is a locally compact group acting on $\f{X}$ properly, co-compactly and isometrically, 
then an elliptic operator $\Dirac^{G\backslash \cE}$ on $G\backslash \f{X}$ lifts to a transversally elliptic operator on $\f{X}$, which is elliptic when $G$ is discrete.
When a group $G$ is continuous, the restriction $\Dirac^{G\backslash \cE}$ loses informations on the longitudinal part of a $G$-invariant operator $\Dirac^{\cE}$ on $\f{X}.$ However, a result similar to Theorem~\ref{Thm:TrivialRepK-theoreticIndex=OrbifoldIndex}, when a locally compact group acts on a manifold properly and co-compactly, can be found in~\cite{Mathai-Zhang}.
\end{remark}

\section{Localized Indices}
\label{sec:Localized indices}

In this section, we introduce the \emph{localized trace} associated to each conjugacy class $(g)$ of $g\in G$ and define the corresponding localized index. We show  that localized indices is a well-defined topological invariant for the $G$-invariant Dirac operator $\Dirac^{\cE}$. 
In particular, the localized index  at  the group identity is just  the $L^2$-index, obtained from taking the canonical von Neumann trace on  the group von Neumann algebra $\cN G$ of the higher index. 
We also show  that the localized index can be computed from the heat kernel of the Dirac operator.

\subsection{Localized traces}
\label{sec:LocalizedTrace}

Denote by $(g)$ the conjugacy class of $g$ in $G$.  
Define a map  $\tau^{(g)}: \C G\rightarrow\C$  given by 
\begin{equation}\label{Def:(g)-trace}
  \sum_{h\in G} \alpha_h h=  \sum_{h\in(g)} \alpha_h.
\end{equation}

\begin{lemma}\label{TraceLemma}
The linear map $\tau^{(g)}$ in (\ref{Def:(g)-trace}) is a trace for $\C G$, that is, $\tau^{(g)} (ab) = \tau^{(g)} (ba)$ for any $a, b\in \C G$.
\end{lemma}

\begin{proof}
Let $a=\sum_{g\in G} a_g g$ and $b=\sum_{g\in G} b_g g$, where all but finite coefficients are $0$, i.e. $a, b\in \C G$. 
Let $c_g, d_g$ be the coefficients of the products:
\[
ab=\sum_{g\in G} c_g g\quad \text{ and }\quad ba=\sum_{g\in G} d_g g.
\]
Thus 
$c_k=\sum_{h\in G} a_{kh^{-1}}b_h$ and $d_k=\sum_{h\in G} b_{kh^{-1}}a_h.$
Let $K$ be defined in~(\ref{ProofTrace}). Then
 \[
\tau^{(g)}\left(\sum_{h\in G} \alpha_h h\right)=  \sum_{k\in K} \alpha_{kgk^{-1}}.
\] 
By Definition~\ref{Def:(g)-trace} and Lemma~\ref{lem:K^GZ_G=G}, we have
\begin{align*}
\tau^{(g)}(ba)&=\sum_{k\in K, h\in G} a_h b_{kgk^{-1}h^{-1}}\quad \text{and}\\
\tau^{(g)}(ab)&=\sum_{k\in K, h\in G}a_{kgk^{-1}h^{-1}}b_h=\sum_{k\in K}\left(\sum_{h\in G}a_{kgk^{-1}h}b_{h^{-1}}\right)\\
&=\sum_{k\in K}\left(\sum_{h\in G}a_hb_{h^{-1}kgk^{-1}}\right)=\sum_{h\in G}\left(\sum_{k\in K}a_hb_{h^{-1}kgk^{-1}}\right)\\
&=\sum_{h\in G}\left(\sum_{k\in hK}a_hb_{kgk^{-1}h^{-1}}\right).
\end{align*}
It is easy to verify that $hK$ also satisfies~(\ref{ProofTrace}) for each $h\in G$. 
Then $\tau^{(g)}(ab)=\tau^{(g)}(ba).$ 
The lemma is proved.
\end{proof}

\begin{definition}[\bf Localized $(g)$-trace]
\label{def:Localized(g)Trace}
Let $\cS(G)$ be a Banach algebra being a unconditional completion of $\C G$ satisfying
\begin{equation}
\label{eq:InclusionAlgebras}
L^1(G)\subset \cS(G)\subset C^*_{r}(G).
\end{equation}
A {\bf localized $(g)$-trace} on $\cS(G)$ is a continuous trace map 
\begin{equation}
\label{eq:Localized(g)trace}
\tau^{(g)}: \cS(G)\longrightarrow\C,
\end{equation}
which extends
  the  map (\ref{Def:(g)-trace}).
\end{definition}

\begin{remark}
The localized $(g)$-trace map always exists. We can choose $\cS(G)$ to be $L^1(G).$  Note that  $L^1(G)$ is an unconditional completion of $\C G$. The continuity of $\tau^{(g)}: L^1(G)\rightarrow \C$  can be  proved as follows. For any $\epsilon>0$, choose $\delta=\varepsilon,$ for all $\|a\|_{L^1}<\delta$ where $a=\sum_{g\in G} a_g g$, then we have
\begin{equation*}
|\tau^{(g)}(a)|:=\left|\sum_{h\in(g)}a_h\right|\le \sum_{h\in G} |a_h|=\|a\|_{L^1}<\varepsilon.
\end{equation*}
\end{remark}

\begin{remark}
\label{rem:Finiteness(g)Trace}
Let $e$ be the group identity of $G$.
The localized $(e)$-trace is ``global" in the sense that it is given by the canonical continuous trace on $C^*_{r}(G)$:
\begin{equation}
\label{eq:tau^(e)}
\tau^{(e)}: C^*_{r}(G)\longrightarrow \C \qquad  \sum_{h\in G} \alpha_h h\mapsto\alpha_e.
\end{equation}
In fact, $\tau^{(e)}$ can be further extended  to a continuous normalized trace on the group von Neumann algebra $\cN G,$ the weak closure of $C_r^*(G).$
Note that we have the $*$-homomorphisms
 \[
C^*(G)\twoheadrightarrow C^*_{r}(G)\hookrightarrow\cN G,
\]
which induce the following homomorphisms on the level of $K$-theory 
\begin{equation}
\label{eq:ThreeInclusions}
K_*(C^*(G))\longrightarrow K_*(C^*_{r}(G))\longrightarrow K_*(\cN G).
\end{equation}
Recall that a trace $\tau$ on a $C^*$-algebra $A$ is {\bf normalized} if it is a state, i.e., 
\begin{equation*}
\tau(a^*a)\ge 0, \quad \forall a\in A\qquad \text{ and }\quad \tau(e)=1.
\end{equation*}
The trace $\tau^{(e)}$ is normalized.
For a general conjugacy class $(g)$  consisting of  infinite elements, then $\tau^{(g)}$ is \emph{not} normalized, hence it may not be a continuous trace on $\cN G$.
\end{remark}

\begin{remark}
\label{rem:localizedS}
If $G$ is abelian or if $(g)$ is a finite set, then $\tau^{(g)}$, for $g\neq e$,  can be  extended  to a  continuous trace $C^*_{r}(G)\rightarrow\C$. In general, however, a trace map $C^*_{r}(G)\rightarrow\C$ can fail to be continuous. Thus, $\cS(G)$ is not necessarily equal to $C^*_{r}(G).$
\end{remark} 

The localized $(g)$-trace $\tau^{(g)}$ in Definition~\ref{def:Localized(g)Trace} and the $(g)$-trace $\tr^{(g)}$ introduced in Section~\ref{sec:EllipticPDO} are closely related.
Let $S: \Gamma_c(\f{X}, \cE)\rightarrow \Gamma_c(\f{X}, \cE)$ be a $G$-invariant operator that extends to a bounded operator on $L^2(\f{X}, \cE)$. 
Denote by $\tilde S: C_c(G, L^2(\f{X}, \cE))\rightarrow C_c(G, L^2(\f{X}, \cE))$ the lift of $S$ given by 
\begin{equation*}
(\tilde S u)(g)=S(u(g)),\qquad  \forall u\in C_c(G, L^2(\f{X}, \cE)), \forall g\in G.
\end{equation*}
Let $\cS(G)$ be an unconditional completion of $\C G$ such that  the localized $(g)$-trace on $\cS(G)$ is continuous. 
Then the following properly supported operator
\begin{equation}
\label{eq:AveragingSmoothingOpS}
S_{\cA}:=q\tilde S \iota=\sum_{g\in G}g\cdot(c^{\frac12}Sc^{\frac12})
\end{equation}
can be extended to a bounded operator on the $\cS(G)$-module $\cA_{\cS(G)}$ ({\em cf.}~(\ref{eq:A_S(G)})).

\begin{lemma}
\label{lem:VeryImportantLemma}
Let $S: L^2(\f{X}, \cE)\rightarrow L^2(\f{X}, \cE)$ be a bounded selfadjoint $G$-invariant smoothing operator.
Let $S_{\cA}: \cA_{\cS(G)}\rightarrow\cA_{\cS(G)}$ be the operator given by~(\ref{eq:AveragingSmoothingOpS}). 
Then, 
\begin{enumerate}
\item $S_{\cA}$ on $L^2(\f{X}, \cE)$ is of $(g)$-trace class in the sense of Definition~\ref{def:(g)trace} for all $g\in G$. 
\item $\Tr_s S_{\cA}\in \cS(G)$ ({\em cf.} Definition~\ref{def:Localized(g)Trace}) and its localized $(g)$-trace coincide with the $(g)$-trace of $\cS_{\cA}$, i.e.,
\begin{equation}
\label{eq:EQUALITY}
\tr_s^{(g)}(S_{\cA})=\tau^{(g)}\left(\Tr_s S_{\cA}\right)\in\R.
\end{equation}
\end{enumerate}
\end{lemma}

\begin{proof} (1)  As the cut-off function $c$ is compactly supported, the operator $c^{\frac12}Sc^{\frac12}$ is compactly supported. Then $S_{\cA}$ given by $\sum_{g\in G}g\cdot(c^{\frac12}Sc^{\frac12})$ as in~(\ref{eq:AveragingSmoothingOpS}) is properly supported and is of $(g)$-trace class ({\em cf.} Lemma~\ref{lem:SmoothingToProperlySupportedSmoothing} and Proposition~\ref{prop:FormulaFor(g)Trace}).

(2)   Recall from Section~\ref{sec:K-IndexD}, $\Gamma_c(\f{X}, \cE)$ has both the pre-Hilbert space structure~(\ref{eq:GammaXEHilbertSpIP}) and the $\C G$-module structure~(\ref{eq:C_cGInnerProduct}). 
There is an injection $\iota$ from $\Gamma_c(\f{X}, \cE)$ to $C_c(G, L^2(\f{X}, \cE))$ given by~(\ref{eq:injection}), and it preserve the $\C G$-inner product. Now, $\cA_{\cS(G)}$ is the closure of $\iota(\Gamma_c(\f{X}, \cE))$ under the $\C G$-valued inner product~(\ref{eq:C_cGInnerProduct}).

Let $\{u_i\}_{i\in\N}\in\Gamma_c(\f{X}, \cE)$ so that $\{\iota(u_i)\}_{\iota\in\N}$ forms an orthonormal basis for the $\cS(G)$-module $\cA_{\cS(G)}$.
Without loss of generality, let us ignore the $\Z/2\Z$-grading on $\cE$ and work on trace instead of supertrace.
Then 
\begin{equation*}
(\Tr S_{\cA})(k))=\sum_{i}\langle S_{\cA}u_i, u_i\rangle_{\C G}(k)\qquad \forall k\in G.
\end{equation*}
As $S_{\cA}$ is properly supported, $u_i$ is compactly supported and the action of $G$ on $\f{X}$ is proper, $\Tr S_{\cA}(k)$ vanishes for all but finite $k\in G.$
Therefore, $\Tr S_{\cA}\in\C G\subset \cS(G).$ 

To see~(\ref{eq:EQUALITY}), we calculate the localized $(g)$-trace of $\Tr S_{\cA}$ as follows.
\begin{equation*}
\tau^{(g)}((\Tr S_{\cA})(\cdot))=\sum_{k\in(g)}\sum_{i}\langle S_{\cA}u_i, u_i\rangle_{\C G}(k).
\end{equation*}
By definition of the $C^*(G)$-inner product~(\ref{eq:C_cGInnerProduct}), we have 
\begin{equation*}
\begin{split}
\langle u_i, u_j\rangle_{L^2}=&\langle u_i, u_j\rangle_{\C G}(e)=\delta_{ij} \quad \forall i, j\in\N;\\
\langle h\cdot u_i, g\cdot u_j\rangle_{L^2}=&\langle u_i, u_j\rangle_{\C G}(h^{-1}g)=0 \quad  \forall g\neq h, \forall i, j\in\N.
\end{split}
\end{equation*}
Then 
\begin{equation}
\label{eq:BasisForL^2}
\{g\cdot u_i\}_{g\in G, i\in\N}\in \Gamma_c(\f{X}, \cE)
\end{equation} 
forms an orthonormal subset of $L^2(\f{X}, \cE).$

We claim that (\ref{eq:BasisForL^2})  forms  a basis for $L^2(\f{X}, \cE)$. If not, let $v\in L^2(\f{X}, \cE)$ is a vector perpendicular to all elements in (\ref{eq:BasisForL^2}). Then 
\begin{equation*}
\langle v, u_i\rangle_{\C G}(g)=\langle v, g\cdot u_i\rangle_{L^2}=0 \quad \forall g\in G, \forall i\in\N.
\end{equation*}
This implies that $\iota(v)$ is perpendicular to the basis $\{\iota(u_j)\}_{j\in\N}$ in $\cA_{\cS(G)}$. So $\iota(v)=0$. As $\iota$ is injective we conclude that $v=0.$
Thus, elements of the set~(\ref{eq:BasisForL^2}) form an orthonormal basis for $L^2(\f{X}, \cE).$

To compute the trace of $S_{\cA}$, we consider 
\begin{equation}
\label{eq:TraceS_A}
\langle S_{\cA}u_i, u_i\rangle_{\C G}(k)=\int_{\f{X}}^{orb}\langle S_{\cA}u_i(x), k\cdot u_i(x)\rangle_{\cE_x}d\vol_{\f{X}}(x)
\end{equation}
obtained by~(\ref{eq:C_cGInnerProduct}). 
Note that $S_{\cA}$ is $G$-invariant by~(\ref{eq:AveragingSmoothingOpS}).
In particular, we may replace $S_{\cA}$ on the right hand side of~(\ref{eq:TraceS_A}) by
\begin{equation*}
S_{\cA}=\sum_{h\in G}h\cdot(cS_{\cA})=\sum_{h\in G}h(cS_{\cA})h^{-1}.
\end{equation*} 
Therefore for all $k\in G,$ we have
\begin{align*}
& \langle S_{\cA}u_i, u_i\rangle_{\C G}(k)\\
=&\sum_{h\in G}\int_{\f{X}}^{orb}\langle h(cS_{\cA})h^{-1}u_i(x), ku_i(x)\rangle_{\cE_x}d\vol_{\f{X}}(x)\\
=&\sum_{h\in G}\int_{\f{X}}^{orb}\langle (cS_{\cA})h^{-1}u_i(x), (h^{-1}kh)h^{-1}u_i(x)\rangle_{\cE_x}d\vol_{\f{X}}(x).
\end{align*}
As we sum all $k\in(g)$ we obtain
\begin{align*}
&\sum_{k\in(g)}\langle S_{\cA}u_i, u_i\rangle_{\C G}(k)\\
=&\sum_{h\in G}\sum_{k\in(g)}\int_{\f{X}}^{orb}\langle (cS_{\cA})h^{-1}u_i(x), kh^{-1}u_i(x)\rangle_{\cE_x}d\vol_{\f{X}}(x)\\
=&\sum_{k\in(g)}\sum_{h\in G}\int_{\f{X}}^{orb}\langle k^{-1}(cS_{\cA})h^{-1}u_i(x), h^{-1}u_i(x)\rangle_{\cE_x}d\vol_{\f{X}}(x).
\end{align*}
Observe that $\{\iota (u_i)\}_{i\in\N}$ forms an orthonormal basis for $\cA_{\cS(G)}$ and the set
$\{g\cdot u_i\}_{g\in G, i\in\N}$ forms an orthonormal basis for $L^2(\f{X}, \cE).$ 
Then by summing the equality for all $i\in\N,$ we conclude that 
\begin{equation*}
\sum_{k\in(g)}(\Tr S_{\cA})(k)=\sum_{k\in(g)}\Tr(k^{-1}cS_{\cA}).
\end{equation*}
By Definition~\ref{def:(g)trace} and Definition~\ref{def:Localized(g)Trace}
 this is equivalent to say that 
\begin{equation*}
\tau^{(g)}(\Tr S_{\cA}(\cdot))=\tr^{(g)}(S_{\cA}).
\end{equation*}
The lemma is proved. 
\end{proof}

\subsection{Localized indices}
The localized trace in the sense of Definition \ref{def:Localized(g)Trace} induces a group homomorphism on the level of $K$-theory: 
\begin{equation}
\label{eq:Tau(g)onKTheory}
\tau^{(g)}_{\ast}: K_0(\cS(G))\longrightarrow\R, 
\end{equation}
as follows.  
Let $P$ be a projection in $M_n( \cS(G))$, the algebra of $n\times n$ matrices with entries in $\cS(G)$. Define
\[\tau^{(g)}_{\ast}(P)=\Tr (\tau^{(g)} (P)).
\]
Any element of $K_0(\cS(G))$ is represented by $[P_1]-[P_2]$ where $P_1$ and $P_2$ are
projections in the matrix  algebra  with entries in $\cS(G)$. 
Then 
the map~(\ref{eq:Tau(g)onKTheory}) is well-defined  because $\tau^{(g)}$ is a continuous trace map.
The map~(\ref{eq:Tau(g)onKTheory}) is real-valued 
 as any projection  can be written as the difference of two positive operators on which  $\tau^{(g)}_*$ takes real values.

Let $\mu_{\cS(G)}[\Dirac^{\cE}]\in K_0(\cS(G))$ be the Banach algebra version higher index of $\Dirac^{\cE}$ 
given by~(\ref{eq:K-theoreticIndexBan}).
Recall that in Theorem~\ref{Thm:TrivialRepK-theoreticIndex=OrbifoldIndex} 
we applied the homomorphism $\rho$ in~(\ref{eq:Homomorphism and Trace}) to get the orbifold index on $G\backslash \f{X}.$
Also, combining~(\ref{eq:tau^(e)}) with~(\ref{eq:ThreeInclusions}) we have the {\bf $L^2$-index} $\tau^{(e)}_*\mu[\Dirac^{\cE}]$ of the Dirac operator $\Dirac^{\cE}.$

Note that $\rho$ is a trace and $\rho=\sum_{(g)}\tau^{(g)}$ when adding up $\tau^{(g)}$ over all conjugacy classes of $G$.
This means that the higher index of $\Dirac^{\cE}$ can be localized to each conjugacy class. 

\begin{definition}[\bf Localized index]\label{DefEquiInd}
The number 
\begin{equation*}
\ind_{(g)}\Dirac^{\cE}:=\tau^{(g)}_*(\mu_{\cS(G)}[\Dirac^{\cE}])
\end{equation*} 
is called the {\bf localized $(g)$-index} of $\Dirac^{\cE}$. 
In general we call them {\bf localized indices}.
\end{definition}

\begin{remark}
When the conjugacy class $(g)$ of $g\in G$ has finite elements, $\tau^{(g)}$ extends to a continuous trace on $C^*(G)$. Thus, $\ind_{(g)}\Dirac^{\f{X}}=\tau^{(g)}_*\mu[\Dirac^{\f{X}}]$. In particular, $L^2$-index factor through the higher index in the $K$-theory for the reduced group $C^*$-algebra. 
Another case worth mentioning is that when $G$ has the RD property, we have an isomorphism of $K$-theory $K_0(\cS(G))\cong K_0(C_r^*(G))$. Then all localized $(g)$-indices are images of an element in $K_0(C^*_r(G)).$
In general, not all localized $(g)$-indices follows from the higher index in $K$-theory of the group $C^*$-algebra. Fortunately, this does not affect the results of our paper.
\end{remark}

Let $Q$ be a parametrix of $\Dirac^{\cE}_+$ in the sense of Proposition \ref{prop:EllipticParametrix}, then there exist $G$-invariant smoothing operators $S_0$ and $S_1$ where
\begin{equation}
\label{eq:Elliptic1-QD=smooth}
1-Q\Dirac^{\cE}_+=S_0\text{ and }1-\Dirac^{\cE}_+Q=S_1.
\end{equation}
The following proposition provides an explicit calculation of the localized $(g)$-index in terms of $\tr^{(g)}$ ({\em cf.} Definition~\ref{def:(g)trace}). 
Note that by Lemma~\ref{lem:SmoothingToProperlySupportedSmoothing}, $\sum_{g\in G}g(c^{\frac12}S_j^2c^{\frac12})$ is a $G$-invariant properly supported smoothing operator for $j=0,1$, then it of $(g)$-trace class, i.e., $\tr^{(g)}\left(\sum_{g\in G}g(c^{\frac12}S_j^2c^{\frac12})\right)<\infty$.

\begin{proposition} 
\label{prop:LocalizedIndex=SupertraceSmoothing} 
The localized $(g)$-index of $\Dirac^{\cE}$ is calculated by
\begin{equation*}
\ind_{(g)} \Dirac^{\cE}=\tau^{(g)}_*\left(\nu[\Dirac^{\cE}]\right)=\tr^{(g)}\left(\sum_{g\in G}g(c^{\frac12}S_0^2c^{\frac12})\right)-\tr^{(g)}\left(\sum_{g\in G}g(c^{\frac12}S_1^2c^{\frac12})\right).
\end{equation*}
Here, $\tau^{(g)}$ is the localized $(g)$-trace given by Definition~\ref{def:Localized(g)Trace}, which induces a $K$-theory homomorphism $\tau^{(g)}_*: K_0(\cS(G))\rightarrow\R$, and $\tr^{(g)}$ is the $(g)$-trace given by Definition~\ref{def:(g)trace}. 
\end{proposition}

\begin{proof}
Let $\tilde F$ be the lift of $F=\Dirac^{\cE}(1+(\Dirac^{\cE})^2)^{-\frac12}$ from $\Gamma_c(\f{X}, \cE)$ to $\cS(G, L^2(\f{X}, \cE))$ defined by~(\ref{eq:LiftedOperatorF}).
Recall from~(\ref{eq:HigherIndexKKban}) that the Banach algebra version higher index for $\Dirac^{\cE}$ is given by
\begin{equation*}
\mu_{\cS(G)}[F]=[(\cA_{\cS(G)}, 1_{\C}, q\circ\tilde F\circ \iota)]\in KK^{ban}(\C, \cS(G)).
\end{equation*}
where $\cA_{\cS(G)}=q(\cS(G, L^2(\f{X}, \cE)))$ is the Banach $\cS(G)$-module as in~(\ref{eq:A_S(G)}).

Denote $\cB(\cA_{\cS(G)})$ by the set of bounded operators on $\cA_{\cS(G)}$.
Let $\cK(\cA_{\cS(G)})$ be the closed ideal of compact operators over $\cA_{\cS(G)}$.
The algebra $\cK(\cA_{\cS(G)})$ is the closure in the norm for $\cB(\cA_{\cS(G)})$ of the set of integral operators with $G$-invariant continuous kernel and with proper support.
Let $\cS(\cA_{\cS(G)})\subset \cK(\cA_{\cS(G)})$ be an ideal of local Banach algebra in $\cB(\cA_{\cS(G)})$ where
\begin{enumerate}
\item $\cS(\cA_{\cS(G)})$ is closed under holomorphic  functional calculus.
\item $\cS(\cA_{\cS(G)})$ contains the algebra of $G$-invariant properly supported operators with smooth kernel, on which the localized trace, in the sense of Lemma~\ref{lem:VeryImportantLemma}, is well-defined.
\end{enumerate}
In view of~\cite[Section 2]{Connes:Kyoto83}, the algebra $\cS(\cA_{\cS(G)})$ exists and the densely defined  localized traces, viewed as degree $0$ cyclic cocycles, can be extended to this algebra.

By the Fredholm picture of $KK^{ban}(\C, \cS(G))$, the operator $q\tilde F \iota$ is invertible in $\cB(\cA_{\cS(G)})$ up to a compact operator. 
In fact, as $\iota \circ q=\mathrm{id}$ and $q\circ \iota=p\ast$ is the identity map on $\cA_{\cS(G)}$, from~(\ref{eq:Elliptic1-QD=smooth}) we obtain
\begin{equation*}
1-(q\widetilde{Q_{\f{X}}}\iota)(q\widetilde{\Dirac^{\cE}}\iota)=q\widetilde{S_0}\iota
\text{   and   }
1-(q\widetilde{\Dirac^{\cE}}\iota)(q\widetilde{Q_{\f{X}}}\iota)=q\widetilde{S_1}\iota \text{   on } \cA_{\cS(G)}.
\end{equation*}
Observe that $q\widetilde{S_j}\iota\in\cS(\cA_{\cS(G)})$. 
Thus, $P:=q\widetilde{\Dirac^{\cE}}\iota$ and $Q:=q\widetilde{Q_{\f{X}}}\iota$ represent elements of $K_1(\cB(\cA_{\cS(G)})/\cS(\cA_{\cS(G)})).$
In the meanwhile denote by 
\begin{equation*}
T_j=(q\widetilde{S_j}\iota)^2=q\widetilde{S_j}^2\iota.
\end{equation*}
Moreover, the set of compact operators $\cK(\cA_{\cS(G)})$ over the Banach $\cS(G)$-module $\cA_{\cS(G)}$ is Morita equivalent to $\cS(G)$. Thus, we have 
\begin{equation*}
K_0(\cS(\cA_{\cS(G)}))\rightarrow K_0(\cK(\cA_{\cS(G)}))\cong K_0(\cS(G)).
\end{equation*}
Therefore, the standard identification between $KK^{ban}(\C, \cS(G))$ and $K_0(\cS(G))$ are given by the boundary map
\begin{align*}
KK^{ban}(\C, \cS(G))&\rightarrow K_1(\cB(\cA_{\cS(G)})/\cS(\cA_{\cS(G)}))\rightarrow K_0(\cS(\cA_{\cS(G)}))\rightarrow  K_0(\cS(G))\\
[(\cA, 1_{\C}, P)]&\mapsto[P]\mapsto \left[\begin{pmatrix}T_0^2&T_0(1+T_0)Q\\PT_1&1-T_1^2\end{pmatrix}\right]-\left[\begin{pmatrix}0&0\\0&1\end{pmatrix}\right]\in M_2(\cS(\cA_{\cS(G)})^+).
\end{align*} 
Here, $\cS(\cA_{\cS(G)})^{+}$ is the unitalization of $\cS(\cA_{\cS(G)})$. 
By convention, the trace of this identity element in $\cS(\cA_{\cS(G)})^{+}\backslash \cS(\cA_{\cS(G)})$ is always assumed to be $0$ (See~\cite{Connes:Kyoto83}).
Thus, we have
\begin{equation*}
\tau^{(g)}_*(\mu[\Dirac^{\cE}])=\tau^{(g)}_*([T_0^2])-\tau^{(g)}_*([T_1^2]).
\end{equation*}
To find the trace on the right hand side, we shall regard $q\widetilde{S_j}\iota$ as a matrix with coefficient in $\cS(G)$.
For $j=1,2$, similar to the calculation of~(\ref{eq:OrbioldKIndexCompression}), we have 
\begin{equation*}
T_j^2=q\widetilde{S_j}^2 \iota=\sum_{g\in G}g\left(c^{\frac12}S_j^2c^{\frac12}\right).
\end{equation*}
Therefore, by~(\ref{eq:Tau(g)onKTheory}), which defines $\tau^{(g)}_*$, and Lemma~\ref{lem:VeryImportantLemma}, we obtain  
\begin{equation*}
\tau^{(g)}_*([T_j^2])=\tr^{(g)}(T_j^2)\in\R.
\end{equation*}
The proposition is then proved.
\end{proof}

An immediate corollary to Proposition~\ref{prop:LocalizedIndex=SupertraceSmoothing} is that the localized $(g)$-trace may be computed by the supertrace of the heat kernel of $\Dirac^{\cE}$.
In fact, for each $t>0$, choose a parametrix $Q_{\f{X}}$ so that $S_0=e^{-t\Dirac^{\cE}_+\Dirac^{\cE}_-/2}$ and $S_1=e^{-t\Dirac^{\cE}_-\Dirac^{\cE}_+/2}$ in~(\ref{eq:Elliptic1-QD=smooth}).

\begin{corollary}
\label{cor:MckeanSinger}
$\ind_{(g)}(\Dirac^{\cE})=\tr_s^{(g)}\left(e^{-t(\Dirac^{\cE})^2}\right).
$\end{corollary}

We are now ready to state the main theorem of this section. We combine the heat kernel asymptotics in Theorem~\ref{thm:HeatKernelAsymptotics} with Corollary~\ref{cor:MckeanSinger} to calculate the localized index in the following theorem.
 
 \begin{theorem}
\label{thm:MainTheorem} 
 Let $\f{X}$ be a complete Riemannian orbifold where a discrete group $G$ acts properly, co-compactly and isometrically. Let $\Dirac^{\cE}$ be  the $G$-invariant Dirac operator on $\f{X}$.
The localized $(g)$-index is calculated by
\begin{equation}\label{eq:LocalIndexFormula}
\ind_{(g)}(\Dirac^{\cE})= \int^{orb}_{(G\backslash \f{X})_{(g)}}\hat{A}_{(g)} (\f{X})   ch^{\cS}_{(g)} (G\backslash\cE),
\end{equation}
where $\hat{A}_{(g)} (\f{X})$ and $ch^{\cS}_{(g)}$ are the delocalized characteristic classes of $G\backslash\f{X}$ restricted to the $(g)$-twisted sector $(G\backslash\f{X})_{(g)}$ given by~(\ref{eq:RestrictionDelocChar}).
\end{theorem}

This localized indices for $\Dirac^{\cE}$ gives rise to refined topological invariants for the Dirac operator on the orbifold $G\backslash \f{X}.$ In fact, in view of Theorem~\ref{thm:MainTheorem} and Theorem~\ref{thm:HeatKernelAsymptotics}, the following theorem is immediate. We provide in addition a $K$-theory proof.

\begin{theorem}\label{Thm:OrbifoldIndex=SumLocalizedIndex}
The orbifold index on $G\backslash \f{X}$ is the sum of localized $(g)$-indices over all conjugacy classes of $G$, i.e. 
\[
\ind \Dirac^{G\backslash \cE} = \sum_{(g)}\ind_{(g)}(\Dirac^{\cE}).
\]
\end{theorem}

\begin{proof} Note that $\ind \Dirac^{G\backslash \cE} =\rho_*(\mu[\Dirac^{\cE}])$ and 
$\ind_{(g)}(\Dirac^{\cE})=\tau^{(g)}_*(\mu[\Dirac^{\cE}]) $. 
Denote by $P_0$ and $P_1$ two $\cS(G)$-valued projection matrices ($P_i^2=P_i=P_i^*$ where $i=0, 1$) such that $\mu[\Dirac^{\cE}]=[P_0]-[P_1]\in K_0(\cS(G)).$
As the localized indices for Dirac operators are finite, $|\tau^{(g)}_*([P_i])|<\infty.$
We need only to show that 
\begin{equation}
\label{eq:PfInd=sumInd}
\rho_*([P])=\sum_{(g)}\tau^{(g)}_*([P]), \qquad P=P_i,\qquad  i=0\text{ or }1.
\end{equation}
Denote by $\rho(P)$ (resp., $\tau^{(g)}(P)$) the $\C$-valued matrix whose $(i,j)$th-entry is $\rho$ (resp., $\tau^{(g)}$) applied to the $(i,j)$th-entry of $P$. 
Then, the left hand side of~(\ref{eq:PfInd=sumInd}) is equal to the rank of $\rho(P)$, and the right hand side of~(\ref{eq:PfInd=sumInd}) is the sum of the trace of $\tau^{(g)}(P)$ over all conjugacy classes of $G$, which is  also the trace of $\rho(P)$ observing that $\rho(P)=\sum_{(g)}\tau^{(g)}(P).$ Hence, it is sufficient to show that the rank of $\rho(P)$ equals its trace.

As $\rho$ is a homomorphism, the image $\rho(P)$ of the projection $P$ is still a projection, i.e., $\rho(P)^2=\rho(P)=\rho(P)^*.$ The $\C$-valued projection $\rho(P)$ is then unitary equivalent to a diagonal matrix $Q$ whose entries are either $1$ or $0$. Note that the trace and rank are invariant under unitary equivalence.  
Thus, the rank of $\rho(P)$ is the same as its trace.
This completes the proof of the theorem.
\end{proof}

\begin{remark}
Theorem~\ref{thm:MainTheorem} gives rise to a local index formula for Dirac operators. 
Then the localized index can be defined for any $G$-invariant elliptic operator $D$ on $\f{X}$ by adapting the argument of~\cite{W:2012} to the case of orbifold. In fact, $D$ gives rise to an element in $K^0_G(C^*_{red}(\f{X}))$ by Lemma~\ref{le:KHomologyElement}. Then using a similar construction as in~\cite[Section 7]{ABP:Invent73} and in~\cite{W:2012}, we can find a $G$-invariant Dirac type operator $\Dirac$, representing the same $K$-homology class as $D$. Hence $\Dirac$ and $D$ have the same higher index and localized indices. Thus, localized index formulae are well-defined for $[D]\in K^0_G(C^*_{red}(\f{X}))$ and is calculated by local index formula~(\ref{eq:LocalIndexFormula}) for the Dirac operator $\Dirac.$ 
\end{remark}

\section{Applications and Further Remarks}
\label{sec:Applications and Further Remarks}

\subsection{$L^2$-Lefschtez fixed  point formula} To illustrate that Theorem \ref{thm:MainTheorem}  is indeed an $L^2$-Lefschtez fixed  point formula for non-compact orbifolds. We restrict ourselves to  the case of a complete Riemannian manifold where a discrete group $G$ acts properly, co-compactly and isometrically.

Let $\f{X}$ be a good orbifold arising from a  complete Riemannian manifold  $M$ with a proper , co-compact  and isometric action of a   discrete group $G$.     
In this situation, the twisted sector of the orbifold $\f{X}=G\backslash M$ is simply indexed by the conjugacy classes of the group $G.$
Then over $M$ associated to each conjugacy class $(g)\subset G$ we have the localized $(g)$-index   for the $G$-invariant Dirac operator $\Dirac^{\cE}$.  
By Theorem~\ref{thm:MainTheorem} we know that 
\begin{enumerate}
\item When $g$ is the group identity, the localized index of $\Dirac^{\cE}$, also known as the $L^2$-index, gives rise to the top stratum of the Kawasaki index formula for the Dirac operator $\Dirac^{G\backslash\cE}$ on $\f{X}=G\backslash M.$ 
\item When $g\in G$ is not the group identity,  the localized indices of $\Dirac^{\cE}$ characterize the lower strata of the orbifold index formula of $\Dirac^{G\backslash\cE}$. 
\end{enumerate}
Therefore, we have related the higher index of $\Dirac^{\cE}$ to the orbifold index restricted to each twisted sector $\f{X}_{(g)}$ by the localized $(g)$-trace ({\em cf.} Definition~\ref{def:Localized(g)Trace}). 

Denote by $M^g$ the fixed point submanifold of $M$ by $g\in G.$ Then the component for the inertia orbifold $I(G\backslash M)$
has the following structure indexed by the conjugacy class $(g)$ of $G$
\begin{equation}
(G\backslash M)_{(g)}=G\backslash\cup_{h\in(g)}M^h=Z_G(g)\backslash M^g.
\end{equation}

In the following, we shall derive from Theorem~\ref{thm:MainTheorem} a formula of $\ind^{(g)}\Dirac^{\cE}$ as integration over fixed point submanifolds by introducing a suitable cut-off function. 
For example, when $g=e$, the localized $(e)$-index of $\Dirac^{\cE}$ is 
the $L^2$-index of $\Dirac^{\cE}$ and is equal to the top stratum of the formula for $\ind\Dirac^\cE_{G\backslash M}$:
\begin{equation}
\label{eq:L^2} 
L^2\text{-}\ind(\Dirac^{\cE})=\int^{orb}_{G\backslash M} \hat{A}(G\backslash M )\ch^{\cS}(G\backslash \cE)=\int_{M}c(x)\hat A(M )\ch^{\cS}(\cE).
\end{equation} 
where $c$ is a cut-off function on $M$ with respect to the $G$ action.
We shall show in this subsection a localized index formula for all $g\in G$ in the fashion of~(\ref{eq:L^2}).  

Note that $Z_G(g)$ acts on $M^g$ isometrically. Given a cut-off function $c$ on $M$ with respect to the $G$ action.  We construct a function $c^{(g)}$ on $ M^{g}$ as follows:
\begin{equation}
\label{eq: cut-offSubmfd}
c^{(g)}(y)=\sum_{k\in G/Z_G(g)}c(k^{-1}y) \quad y\in M^g
\end{equation}
where $G/Z_G(g)$ is identified as a subset $K$ of $G$ in view of Lemma~\ref{lem:K^GZ_G=G} (See~(\ref{ProofTrace}) for the definition of $K$).
By Lemma~\ref{lem:K^GZ_G=G}, the function given by~(\ref{eq: cut-offSubmfd}) is in fact a cut-off function on $M^g$ with respect to the $Z_G(g)$-action: 
\begin{equation*}
\sum_{l\in Z_G(g)}c^{(g)}(l^{-1}y)=\sum_{l\in Z_G(g)}\sum_{k\in K}c(l^{-1}k^{-1}y)=\sum_{g\in G}c(g^{-1}y)=1.
\end{equation*}
The localized $(g)$-index of $\Dirac^{\cE}$ is  given 
by 
\begin{equation*}\begin{array}{lll}
\ind_{(g)}(\Dirac^{\cE})& = & \disp{\int^{orb}_{  \f{X}_{(g)}}  } \dfrac{ \hat{A}( \f{X}_{(g)})
\ch^\cS_{(g)}(G\backslash \cE)}{\det \big(1-\Phi_{(g)}  e^{R_{\cN_{(g)}}/2\pi i}  \big)^{\frac12}}
\\[3mm]
&=& \disp{\int_{M^g}} c^{(g)}(x)   \dfrac{\hat A(M^g)\ch^\cS_g(\cE)}{\det(1-g e^{R_{\cN^g}/2\pi i})^{\frac12}}.
\end{array}
\end{equation*}

We summarise these results as the following  $L^2$-version of the Lefschtez fixed  point formula for a complete Riemannian manifold $M$ with a proper co-compact action of a discrete group $G$. 

\begin{theorem}
\label{thm:LocalizedIndexKawasakiIndex}
Let $M$ be a complete Riemannian manifold where a discrete group $G$ acts properly, co-compactly and isometrically. Let $\Dirac^{\cE}$ be the $G$-invariant Dirac operator on $M$ and $\Dirac^\cE_{G\backslash M}$ be the corresponding Dirac operator on the quotient orbifold $G\backslash M.$ 
Then   
$
L^2\text{-}\ind \Dirac^{\cE} = \ind_{(e)} (\Dirac^{\cE}), 
$
the  localized  index of $\Dirac^{\cE}$ at the identity conjugacy class $(e)$. For   $ g \neq e$, the localized $(g)$-index of $\Dirac^{\cE}$ is  given 
by 
\[
\ind_{(g)}(\Dirac^{\cE})=    \disp{\int _{M^g}} c^{(g)}(x)   \dfrac{\hat A(M^g)\ch_g^{\cS}(\cE)}{\det(1-ge^{R_{\cN^g}/2\pi i})^{\frac12}}. 
\]
where $c^{(g)}$ is the cut-off function on $M^g$ with respect to the action of $Z_G(g)$ given by~(\ref{eq: cut-offSubmfd}).

\end{theorem}

\subsection{Selberg trace formula}

We shall present an interesting connection between the localized indices and the orbital integrals in the Selberg trace formula. To begin with, we recall the set up of the Selberg trace formula~\cite{Selberg:56} (See also the survey article \cite{Ar:Clay05}).

Let $G$ be a real unimodular Lie group and $\Gamma$ be a discrete co-compact subgroup of $G$.
Denote by $R$ the right regular unitary representation of $G$ on $L^2(G):$
\begin{equation*}
R(g)f(h)=f(g^{-1}h)\qquad \forall g, h\in G, \forall f\in L^2(G).
\end{equation*}
It extends to a representation of $L^1(G)$ on $L^2(G)$ as follows:
\begin{equation}
\label{eq:RRrepL^1(G)}
R(f)=\int_G f(g)R(g)dg \qquad \forall f\in L^1(G).
\end{equation}
As $\Gamma$ acts on the left of $G$, $R$ is reduced to 
\begin{equation}
\label{ea:RightRegularRep}
R: L^1(G)\longrightarrow \End(L^2(\Gamma\backslash G)).
\end{equation}
Let $f\in C^{\infty}(G)\cap L^1(G)$ be a test function, where $R(f)$ is a trace class operator on $L^2(\Gamma\backslash G).$ The Selberg trace formula is an equality relating two ways in calculating $\Tr R(f)$, where $\Tr$ is the operator trace on $L^2(\Gamma\backslash G).$
On one hand, $\Tr R(f)$ has the spectral decomposition indexed by all irreducible unitary representations of $G$ (denoted by $\Irr(G)$):
\begin{equation}
\label{eq:STFspectral}
\Tr R(f)=\sum_{\pi\in\Irr(G)}m(\pi)\Tr(\pi(f)),
\end{equation}
where $m(\pi)$ is the multiplicity of $\pi$ in $R.$ The equality~(\ref{eq:STFspectral}) is called the {\bf spectral side} of the Selberg trace formula.
On the other hand, the Schwartz kernel $K(x,y)$ of $R(f)$, which has the expression
\begin{equation}
\label{eq:KernelR(f)}
K(x,y)=\sum_{\gamma\in\Gamma}f(x^{-1}\gamma y),
\end{equation}
can be used to calculate $\Tr R(f)$ as follows:
\begin{equation}
\label{eq:OrbitalIntegal}
\Tr R(f)=\int_{\Gamma\backslash G} K(x,x)dx=\sum_{(\gamma)\subset \Gamma}\vol(Z_G(\gamma)/Z_{\Gamma}(\gamma))\int_{Z_G(\gamma)\backslash G}f(x^{-1}\gamma x)dx,
\end{equation}
where the sum is over representatives of all conjugacy classes of $G.$ This equality~(\ref{eq:OrbitalIntegal}) is called the {\bf geometric side} of the Selberg trace formula. We denote the $(\gamma)$-summand in~(\ref{eq:OrbitalIntegal}) by $\cO_{\gamma}$ and call it {\bf orbital integral}:
\begin{equation*}
\cO_{\gamma}:=\vol(Z_G(\gamma)/Z_{\Gamma}(\gamma))\int_{Z_G(\gamma)\backslash G}f(x^{-1}\gamma x)dx.
\end{equation*}
The orbital integral is easier to calculate and is an important tool in finding the multiplicity of a representation.

To relate Selberg trace formula to the localized index, we shall further recall the similar setting appeared in~\cite{CM82, BM83}. 
Let $H$ be a maximal compact subgroup of $G$ with volume $1.$ 
Let $M=G/H$ be the Riemannian symmetric manifold of noncompact type. 
Then the discrete cocompact subgroup $\Gamma$ of $G$ acts properly and co-compactly on $M$.
Let $\cE$ be a $\Z/2\Z$-graded homogeneous vector bundle over $M$, i.e., there exists a $\Z/2\Z$-graded $H$-representation $E$, where $H$ respects the grading of $E$, so that $\cE=G\times_H E$. 
Denote by $\Dirac^{\cE}: L^2(M, \cE)\rightarrow L^2(M, \cE)$ the $G$-invariant Dirac operator. 
We then have the identification 
\begin{equation*}
L^2(M, \cE)\cong (L^2(G)\otimes E)^H.
\end{equation*} 
Therefore, the corresponding Dirac operator $\Dirac^{\Gamma\backslash \cE}$ on the quotient orbifold $\Gamma\backslash M$ is regarded as an operator on $(L^2(\Gamma\backslash G)\otimes E)^H$.
Let us denote by $K_t$ (resp., $\bar K_t$) the heat kernel for $\Dirac^{\cE}$ (resp., $\Dirac^{\Gamma\backslash \cE}$). 
For the same reason of the homogeneity, 
\begin{equation*}
K_t\in (C^{\infty}(G\times G)\otimes\End(E))^{H\times H}\quad \bar K_t\in (C^{\infty}(\Gamma\backslash G\times \Gamma\backslash G)\otimes\End(E))^{H\times H}.
\end{equation*}
As $D^{\cE}$ is $G$-invariant, $K_t$ gives rise to a well-defined function $k_t\in (C^{\infty}(G)\otimes\End(E))^{H\times H}$ given by 
\begin{equation}
\label{eq:HeatKernalTestingFunction}
k_t(x^{-1}y)=K_t(x,y), \quad \forall x, y\in G.  
\end{equation}

We shall still consider the right regular representation of a test function, but we replace the representation space $L^2(\Gamma\backslash G)$ by the $\Z/2\Z$-graded space $(L^2(\Gamma\backslash G)\otimes E)^H$ and the operator trace by the supertrace. 
Note that by~(\ref{eq:KernelR(f)}) and Theorem~\ref{thm:HeatKernelAsympMain}, we observe that the Schwartz kernel of $R(k_t)$ is exactly $\bar K_t(x,y),$ the heat kernel for $\Dirac^{\Gamma\backslash\cE}$
Applying Theorem~\ref{thm:HeatKernelAsympMain} to~(\ref{eq:KernelR(f)}), we see that the supertrace $\Tr_s R(k_t)$ is finite. 
Therefore, even though $k_t$ is not compactly supported on $G$, we can choose $k_t$ as a test function. 
Comparing the localized index formula and the Selberg trace formula in this situation, we obtain the following theorem, which states that the orbital integrals for $\Tr R(k_t)$ have a one-to-one correspondence with the localized indices for $\Dirac^{\cE}.$

\begin{theorem}
\label{thm:SelbergTraceFormula}
Let $k_t$ be the test function~(\ref{eq:HeatKernalTestingFunction}) determined by the heat operator for the Dirac operator $\Dirac^{\cE}$ on the Riemannian symmetric manifold $M=G/H$ of noncompact type. 
Let $R(k_t)$ be the right regular representation of $L^1(G)$ on $(L^2(\Gamma\backslash G)\otimes E)^H$ in the sense of~(\ref{ea:RightRegularRep}).
Then the geometric side of the Selberg trace formula of $\Tr_s(R(k_t))$, which was expressed as a sum of orbital integrals $\cO_{\gamma}$ in~(\ref{eq:OrbitalIntegal}) over all conjugacy classes of $\Gamma$, corresponds exactly to the sum of localized indices for $\Dirac^{\cE}$. 
Moreover, 
\begin{equation*}
\Tr_s(R(k_t))=
\sum_{(\gamma)\subset\Gamma}  \cO_\gamma=  \ind\Dirac^{\Gamma\backslash\cE}, \qquad   \cO_\gamma=  \ind_{(\gamma)}\Dirac^{\cE}.
\end{equation*}
\end{theorem}

\begin{proof}
By the geometric side~(\ref{eq:OrbitalIntegal}) of the Selberg trace formula and the definition of the localized indices ({\em cf.} Theorem~\ref{thm:MainTheorem}), we shall only need to show that the orbital integral  
\begin{equation}
\label{eq:OGamma}
\cO_{\gamma}=\vol(Z_G(\gamma)/Z_{\Gamma}(\gamma))\int_{Z_G(\gamma)\backslash G}\Tr_s [k_t(x^{-1}\gamma x)\gamma]dx
\end{equation}
where $\Tr_s$ is the supertrace of $\End E$, coincides with 
\begin{equation*}
\tr_s^{(\gamma)_{\Gamma}}e^{-t(\Dirac^{\cE})^2}:=\sum_{h\in(\gamma)}\int_Gc(x)\Tr_s[K_t(x, hx)h]d x,
\end{equation*}
where $c$ is a cut-off function on $G$ with respect to the $\Gamma$-action.

Let us rewrite $\tr_s^{(\gamma)_{\Gamma}}e^{-t(\Dirac^{\cE})^2}$ by 
\begin{equation}\
\label{eq:STFtrace}
\begin{split}
\tr_s^{(\gamma)_{\Gamma}}e^{-t(\Dirac^{\cE})^2}=&\sum_{k\in K}\int_Gc(x)\Tr_s[k_t(x^{-1}k\gamma k^{-1} x)k^{-1}\gamma k]dx\\
=&\int_G c^{(\gamma)}(x)\Tr_s[k_t(x^{-1}\gamma x)\gamma]dx,
\end{split}
\end{equation}
where $K$ is the set generating the conjugacy class $(\gamma)_{\Gamma}$ defined in~(\ref{ProofTrace}) and 
\begin{equation*}
c^{(\gamma)}(x):=\sum_{k\in K}c(kx).
\end{equation*}
Identify the space of the right cosets $Z_G(\gamma)\backslash G$ of $Z_G(\gamma)$ in $G$ as a subset of $G$ consisting of representatives of the right cosets. Then any $x\in G$ can be decomposed uniquely into $x=ba$ where $b\in Z_G(\gamma)$ and $a\in Z_G(\gamma)\backslash G.$ 
Notice that $k_t(a^{-1}b^{-1}\gamma ba)\gamma=k_t(a^{-1}\gamma a)\gamma$ for all $b\in Z_G(\gamma)$.
Thus,~(\ref{eq:STFtrace}) is equal to
\begin{equation}
\label{eq:STFProof}
\tr_s^{(\gamma)_{\Gamma}}e^{-t(\Dirac^{\cE})^2}=\int_{Z_G(\gamma)\backslash G}\Tr_s[k_t(a^{-1}\gamma a)\gamma]\left[\int_{Z_G(\gamma)} c^{(\gamma)}(ba)db\right]da.
\end{equation}
Lemma~\ref{lem:K^GZ_G=G} implies that $K \cdot Z_{\Gamma}(\gamma)=\Gamma$. 
Then for any $l\in G$, we have 
\begin{equation*}
\sum_{b\in Z_{\Gamma}(\gamma)}c^{(\gamma)}(bla)=\sum_{k\in K, b\in Z_{\Gamma}(\gamma)}c(kbla)=\sum_{h\in \Gamma}c(hla)=1.
\end{equation*}
The term $\int_{Z_G(\gamma)} c^{(\gamma)}(ba)db$ in~(\ref{eq:STFProof}) can be calculated similarly as the argument we used to derive~(\ref{eq:STFProof}):
\begin{equation*}
\int_{Z_G(\gamma)} c^{(\gamma)}(ba)db=\int_{Z_G(\gamma)/Z_{\Gamma}(\gamma)}\left[ \sum_{l\in Z_{\Gamma}(\gamma)}c^{(\gamma)}(hla)dl\right] dh=\vol(Z_G(\gamma)/Z_{\Gamma}(\gamma)).
\end{equation*}
Together with~(\ref{eq:STFProof}) and~(\ref{eq:OGamma}) we see that $\tr_s^{(\gamma)_{\Gamma}}e^{-t(\Dirac^{\cE})^2}=\cO_{\gamma}$.
The theorem is then proved by comparing~(\ref{eq:OrbitalIntegal}) and Theorem~\ref{Thm:OrbifoldIndex=SumLocalizedIndex}.
\end{proof}

\subsection{An application in positive scalar curvature} 
Localized indices produce finer topological invariants for the $G$-orbifold $\f{X}$ and the quotient $G\backslash\f{X}$. They also reveal some geometric information of the orbifold. We shall for example present a result of positive scalar curvature.
As we know, the higher index of an elliptic operator is the obstruction of the invertibility of the operator. However, the higher index is difficult to compute. Therefore, as a weaker condition but easier to compute, the nonvanishing of the localized indices gives rise to an obstruction of invertibility of the operator.
We then formulate the nonvanishing result as follows. 

Let $\f{Y}$ be a spin compact orbifold obtained from the quotient of a complete orbifold $\f{X}$ by a discrete $G$ group acting properly, co-compactly and isometrically. Let $\cS$ be the spinor bundle over $\f{X}$. Denote by $\Dirac^{\cS}$ (resp. $\Dirac^{G\backslash\cS}$) the Dirac operator on $\f{X}$ (resp. $\f{Y}$).
\begin{theorem}
\label{thm:PositiveScalarCurvature}
If any of the localized indices of the $G$-invariant Dirac operator $\Dirac^{\cS}$ on $\f{X}$ is nonzero, then the quotient orbifold $\f{Y}$ cannot have positive scalar curvature.
\end{theorem}

\begin{proof}
If the quotient orbifold $\f{Y}=G\backslash \f{X}$ has a metric leading to a positive scalar curvature, then as $G$ acts by isometry, the covering orbifold $\f{X}$ also has a positive scalar curvature, denoted by $r_{\f{X}}$. Then by the Lichnerowicz formula~(\ref{D^2}), we have 
\begin{equation*}
(\Dirac^{\cS})^2=\Delta^{\cS}+\frac{1}{4}r_{\f{X}},
\end{equation*}
 which is a strictly positive operator. 
Hence $\Dirac^{\cS}$ is invertible. Therefore the higher index $\mu_{\cS(G)}[\Dirac^{\cE}]$ vanishes. 
As the localized indices factor through the Banach algebra version of the higher index map, all the localized indices are $0$, which contradicts with assumption. 
The theorem is then proved.
\end{proof}

\begin{remark}
If the orbifold index of the Dirac operator $\Dirac^{G\backslash\cS}$ on the quotient is $0$, it does not prove the nonexistance of positive scalar curvature of $\f{Y}.$
But if the localized indices, which sum up to be $0$ by Theorem~\ref{Thm:OrbifoldIndex=SumLocalizedIndex}, are not all $0$, then Theorem~\ref{thm:PositiveScalarCurvature} shows that $\f{Y}$ cannot admit a positive scalar curvature. 

On the other hand, while the localized indices (for $g\neq e$) would vanish for spaces with positive scalar curvature, we expect the localized indices to be useful in the study of non-positive curved space, for example, Riemannian symmetric manifolds of noncompact type, and their co-compact orbifold quotient.
\end{remark}


\begin{thebibliography}{CFKS}

\bibitem[ALR]{ALR}  A.  Adem, J.  Leida  and  Y. Ruan,   \emph{ Orbifolds and stringy topology,}  Cambridge Tracts in Mathematics {\bf 171}, Cambridge University Press, 2007.

\bibitem[Ar]{Ar:Clay05}
J. Arthur, \emph{An introduction to the trace formula},  {Harmonic analysis, the trace formula, and {S}himura varieties}, {Clay Math. Proc.}, {\bf 4}, {1--263},  {Amer. Math. Soc.}, {Providence, RI}, (2005).

\bibitem[At]{Atiyah:1976} M. F. Atiyah, \emph{Elliptic operators, discrete groups and von Neumann algebras,} Asterisque {\bf 32-33}, 43--72 (1976).

\bibitem[AB]{AB}  M.F. Atiyah and R. Bott, \emph{ A Lefschetz fixed point formula for elliptic
complexes II,}  Applications. Ann. of Math. {\bf 88}, 451--491(1968).

\bibitem[ABP]{ABP:Invent73} M. Atiyah, R. Bott and V. K. Patodi, \emph{On the heat equation and the index theorem,} Invent. Math. {\bf 19}, 279--330 (1973). 

\bibitem[BM]{BM83} D. Barbasch and H. Moscovici, \emph{$L^2$-index and the Selberg trace formula.} Journal of functional analysis {\bf 53}, 151--201 (1983).

\bibitem[BC1]{BC:88} P. Baum and A. Connes, \emph{Chern character for discrete groups.} A f\'ete of topology. Academic Press, Boston, MA , 163--232 (1988).

\bibitem[BC2]{BC}  P. Baum  and  A. Connes, \emph{ K-theory for discrete groups.}  In Operator Algebras and Applications,  Cambridge University Press, Cambridge, 1--20 (1989).

\bibitem[BCH]{BCH}  P. Baum, A. Connes   and N. Higson,  \emph{ Classifying space for proper actions and $K$-theory of group $C^*$-algebras. }  Contemp. Math.  Vol. {\bf 167}, 240--291 (1994).

\bibitem[BGV]{BGV} N. Berline, E. Getzler and M. Vergne, \emph{ Heat kernels and Dirac operators}, Spinger. 

\bibitem[Bis]{Bis}  J.-M. Bismut, \emph{ Equivariant immersions and Quillen metrics, }  J. Diff. Geom. {\bf 41}, 53--159 (1995).

\bibitem[BiCr]{BiCr} R. Bishop and R. Crittenden, \emph{Geometry of manifolds}, Academic Press, N.Y., 1971.

\bibitem[Bl]{Ba:KTOA} B. Blackadar,
\emph{{$K$}-theory for operator algebras,}
Mathematical Sciences Research Institute Publications Vol 5, 2nd edition, Cambridge University Press, 1998.

\bibitem[Bo]{Bo} J. Borzellino, \emph{Orbifold of maximum diameter}, Indiana Univ. Math. J. {\bf 42}, 37--53 (1993).

\bibitem[Bu]{Bunke:2007} U. Bunke, \emph{Orbifold index and equivariant $K$-homology.} Math. Ann. {\bf 339} 175--194 (2007).

\bibitem[CGT]{CGT} J. Cheeger, M. Gromov and M. Taylor, \emph{ Finite propagation speed, kernel estimates for functions of the Laplace operator, and the geometry of complete Riemannian manifolds,}  J. Differential Geom. {\bf 17}, 15--53 (1982).

\bibitem[Co1]{Connes:Kyoto83} {A. Connes}, \emph{Cyclic cohomology and the transverse fundamental class of a foliation}, {Geometric methods in operator algebras ({K}yoto, 1983)}, {Pitman Res. Notes Math. Ser.}, {\bf 123}, {52--144} (1986)

\bibitem[Co2]{Connes:1994}  A. Connes,  \emph{ Noncommutative geometry.}   Academic Press Inc., San Diego, CA, 1994.
 
\bibitem[CM]{CM82} A. Connes and H. Moscovici, \emph{$L^2$-index theorem for homogeneous spaces of Lie groups.} Ann. Math. Vol. {\bf 115}, No. 2, 291--330 (1982).


\bibitem[CaW]{CarW:2011} P. Carrillo and B. Wang, \emph{ Twisted longitudinal index theorem for foliations and wrong way functoriality}. Adv.  in Math., Vol. {\bf 226}, No. 6, 4933--4986, (2011).

\bibitem[Do]{Do}  H. Donnelly, \emph{Asymptotic expansions for the compact quotients of properly discontinuous
group actions,}   Illinois J. Math. {\bf 23}, 485--496 (1979).


\bibitem[Dui]{Dui} J. J. Duistermaat, \emph{ The heat kernel Lefschetz fixed point formula for the
spin$^c$ Dirac operator},   Progress in nonlinear differential equations and
their applications, v. 18, 1996.

\bibitem[EE]{EE} S. Echterhoff and H. Emerson, \emph{Structure and {$K$}-theory of crossed products by proper actions}, {Expo. Math.}, {\bf 29} {No. 3}, 300--344 (2011).

\bibitem[EEK]{EEK} S. Echterhoff, H. Emerson and H.J. Kim, \emph{A Lefschetz fixed-point formula for certain orbifold $C^*$-algebras. } J. Noncommut. Geom. {\bf 4}, no. 1, 125--155 (2010).

\bibitem[Em]{Em} H. Emerson, \emph{Lefschetz numbers for $C^*$-algebras.}
  Canad. Math. Bull. {\bf 54}, no. 1, 82--99 (2011).
  
\bibitem[F1]{Farsi:1992-1} C. Farsi, \emph{$K$-theoretical index theorems for good orbifolds.} Proceedings of the American Mathematical Society, Vol. {\bf 115}, No. 3 769--773 (1992). 

\bibitem[F2]{Farsi:1992-2} C. Farsi, \emph{$K$-theoretic index theorems for orbifolds.} Quart. J. Math. Oxford (2), {\bf 43}, 183--200 (1992).


\bibitem[FT]{FX} G. Felder and X. Tang, \emph{Equivariant Lefschetz number of differential operators,}  Math. Z. {\bf 266}, no.2, 451-470 (2010). 


   \bibitem[HS]{HS:1987} M. Hilsum and G. Skandalis, \emph{ Morphismes {$K$}-orient\'es d'espaces de feuilles et  fonctorialit\'e en th\'eorie de {K}asparov (d'apr\`es une
 conjecture d'{A}. {C}onnes)},  Ann. Sci. \'Ecole Norm. Sup. (4),  {\bf 20}, no. 3, 325--390 (1987).
    

\bibitem[HW]{HuWa} J. Hu and B. Wang, {\em  Delocalized Chern character for stringy orbifold K-theory.} 
    Trans.  of AMS., (2012). 

\bibitem[K]{Kawasaki:1981} T. Kawasaki, \emph{The index of elliptic operators over $V$-manifold.} Nagoya Math. J. Vol. {\bf 84}, 135--157 (1981).

\bibitem[K1]{Kasparov:1983} G. Kasparov, \emph{The index of invariant elliptic operators, K-theory, and Lie group representations.} English translation: Soviet Mathematics-Doklady. {\bf 27}, 105--109, (1983).

\bibitem[K2]{Kasparov:1988dw}
G.~Kasparov,
\emph{Equivariant ${KK}$-theory and the {N}ovikov conjecture.}  Inventiones Mathematicae, {\bf 91} (1), 147--201, (1988).

\bibitem[K3]{Kasparov:2008}
G.~Kasparov,
\newblock K-theoretic index theorems for elliptic $K$-theoretic index theorems for elliptic and transversally elliptic operators.
\newblock {\em Preprint}, 2012.

\bibitem[KL]{KL} B. Kleiner and  J. Lott, \emph{ Geometrization of three-dimensional orbifolds via Ricci flow,} arXiv:1101.3733v2.

\bibitem[LYZ]{LYZ} J. D. Lafferty, Y. L. Yu and W. P. Zhang, \emph{A direct geometric proof of the Lefschetz
fixed point formulas.}  Trans. Amer. Math. Soc. {\bf 329}, 571-583 (1992).

\bibitem[La]{LafforgueThesis} V. Lafforgue, \emph{$KK$-th\'eorie bivariante pour les alg\`ebres de Banach et conjecture de Baum-Connes}, Th\`ese de Doctorat de l'universit\'e Paris 11 (mars 1999). 

     \bibitem[LM]{LM} B. Lawson and M-L. Michelsohn,  \emph{Spin geometry}. Princeton Mathematical Series, {\bf 38}. Princeton University Press,   1989.
  

     \bibitem[LU]{LU} E. Lupercio and  B. Uribe, \emph{Gerbes over orbifolds and twisted K-theory,} 
Comm. Math. Phys.  {\bf 245}, no. 3,  449--489 (2004).
     
     
     \bibitem[Ma]{Ma}  X.  Ma, \emph{ Orbifolds and analytic torsions.}  Trans. Amer. Math. Soc. {\bf 357}, no. 6, 2205--2233 (2005).
    
       \bibitem[MaMa]{MaMa}   X.  Ma and G. Marinescu,
\emph{Holomorphic Morse inequalities and Bergman kernels,}
Progress in Mathematics {\bf 254}, 
Birkhäuser Boston, Inc., Boston, MA. 2007, 422 pp.
     
     \bibitem[MM]{MarMat}  M. Marcolli and V. Mathai,  \emph{Twisted index theory on good orbifolds. I. Noncommutative Bloch theory. }
Commun. Contemp. Math. {\bf 1}, no. 4, 553--587 (1999).
 
\bibitem[MZ]{Mathai-Zhang} {V. Mathai and W. Zhang},
     \emph{Geometric quantization for proper actions (with an appendix by Ulrich Bunke).} {Adv. Math.} {\bf 225}, {3}, {1224--1247} (2010),
          
     \bibitem[MP1]{MoePr:1997}   I. Moerdijk and D. A. Pronk, \emph{Orbifolds, sheaves and groupoids}, $K$-Theory {\bf 12}, 3--21 (1997).
     
   \bibitem[MP2]{MoePr} I. Moerdijk and D. A. Pronk, \emph{Simplcial cohomolgy of orbifolds},    Indag. Math. (N.S.)  {\bf 10},  no. 2, 269--293 (1999).

\bibitem[Po1]{Po:CMP} R. Ponge, \emph{A new short proof of the local index formula and some of its applications}. 
Comm.\ Math.\ Phys.\ \textbf{241}, 215--234 (2003). 

\bibitem[PW]{PW:NCGCGI.PartII} R. Ponge and H. Wang, \emph{Noncommutative geometry and conformal geometry. II. The local 
equivariant index theorem}. E-print, arXiv, June 2013. 

\bibitem[Ren]{Ren:1980} J.  Renault, \emph{A groupoid approach to $C^*$-algebras},  Lecture Notes in Mathematics.  Vol. {\bf 793}, Springer, Berlin, 1980.

\bibitem[S]{Selberg:56} A. Selberg, \emph{Harmonic analysis and discontinuous groups in weakly symmetric Riemannian spaces with applications to Dirichlet series}, J. Indian Math Soc. {\bf 20}, 47--87 (1956).

\bibitem[TTW]{TaTsW} X. Tang, H. Tseng and B. Wang, \emph{Orbifold $K$-theory for ineffective orbifolds}, in preparation.

\bibitem[Th]{Thurston:1997} W. Thurston,  \emph{The Geometry and Topology of 3-Manifolds}, Princeton Mathematical Series, {\bf 35}, Princeton University Press, 1997.


\bibitem[W]{W:2012} H. Wang, \emph{$L^2$-index formula for properly cocompact group actions.} To appear in ``Journal of Noncommutative Geometry" {arXiv:1106.4542}. 


\bibitem[Yu]{Yu} G. Yu, \emph{ Higher index theory of elliptic operators and geometry
of groups},  Proceed. of the ICM, Madrid, Spain, 2006, 1623--1639. 
\end{thebibliography}
\end{document}